\documentclass[11pt]{amsart}
\usepackage{amsmath,amssymb}
\usepackage{amsfonts}
\usepackage{amsbsy}
\usepackage{fullpage}
\usepackage{stmaryrd}
\usepackage{latexsym}
\usepackage{mathrsfs}
\usepackage{accents}
\usepackage[all]{xy}
\usepackage{tikz}
\usepackage{tikz-cd}
\usepackage{longtable}
\usepackage{cancel}
\usepackage{cite}
\usepackage{url}
\usepackage{hyperref}
\hypersetup{colorlinks}

\usepackage{xcolor}
\usepackage{boxedminipage}
\usepackage{framed}

\usepackage{bm}
\usepackage{enumitem}
\usepackage{datetime2}
\usepackage{comment}

\newcommand{\fg}{\mathfrak{g}}

\newcommand{\fM}{\mathfrak{M}}

\newcommand{\cF}{\mathcal{F}}
\newcommand{\cL}{\mathcal{L}}

\newcommand{\N}{\mathbb{N}}
\newcommand{\R}{\mathbb{R}}

\newcommand{\bbF}{\mathbb{F}}

\newcommand{\bbU}{\mathbb{U}}

\newcommand{\scA}{\mathscr{A}}
\newcommand{\scB}{\mathscr{B}}
\newcommand{\scC}{\mathscr{C}}
\newcommand{\scD}{\mathscr{D}}

\newcommand{\scM}{\mathscr{M}}
\newcommand{\scN}{\mathscr{N}}

\newcommand{\scQ}{\mathscr{Q}}

\newcommand{\scS}{\mathscr{S}}

\DeclareMathAlphabet{\mathpzc}{OT1}{pzc}{m}{it}
\newcommand{\svee}{{\scriptscriptstyle{\vee}}}

\newcommand{\cin}{C^\infty}

\newcommand{\op}[1]{{#1}^{\mbox{\sf{\tiny{op}}}}}
\newcommand{\hpsi}{\widehat{\psi}}

\newcommand{\lb}{\left[ \cdot\,,\cdot\right]}
\newcommand{\Pb}{\left\{ \cdot\,,\cdot\right\}}

\newcommand{\hooklongrightarrow}{\lhook\joinrel\longrightarrow}

\DeclareMathOperator{\colim}{colim}

\DeclareMathOperator{\Der}{Der}
\DeclareMathOperator{\cDer}{C^\infty Der}

\DeclareMathOperator{\CDer}{\cin Der}
\DeclareMathOperator{\Hom}{Hom}
\DeclareMathOperator{\id}{id}
\DeclareMathOperator{\im}{im}

\DeclareMathOperator{\Spec}{Spec}

\DeclareMathOperator{\pr}{\mathsf{pr}}

\DeclareMathOperator{\llb}{\llbracket}
\DeclareMathOperator{\rrb}{\rrbracket}
\DeclareMathOperator{\lbr}{\bm{\{} \!}
\DeclareMathOperator{\rbr}{\! \bm{\}} }
\DeclareMathOperator{\ad}{ad}
\DeclareMathOperator{\grph}{graph}
\DeclareMathOperator{\Free}{\mathsf{Free}}

\newcommand{\cmring}{\mathsf{ComRing}}
\newcommand{\cring}{\cin\mathsf{Ring}}
\newcommand{\clr}{\cin\mathsf{LR}}
\newcommand{\pcring}{\mathsf{P}\cin\mathsf{Ring}}

\newcommand{\Mod}{\mathsf{Mod}}
\newcommand{\Man}{\mathsf{Man}}
\newcommand{\PMan}{\mathsf{PMan}}

\newcommand{\LCRS}{\mathsf{L\cin RS}}

\newcommand{\Set}{\mathsf{Set}}
\newcommand{\FinSet}{\mathsf{FinSet}}

\theoremstyle{definition}
\newtheorem{thm}{Theorem}[section]
\newtheorem{lemma}[thm]{Lemma}
\newtheorem{theorem}[thm]{Theorem}
\newtheorem{proposition}[thm]{Proposition}
\newtheorem{corollary}[thm]{Corollary}
\newtheorem*{corollary*}{Corollary}

\newtheorem*{claim*}{Claim}
\newtheorem{definition}[thm]{Definition}
\newtheorem{remark}[thm]{Remark}
\newtheorem*{remark*}{Remark}
\newtheorem{example}[thm]{Example}
\newtheorem{notation}[thm]{Notation}

\numberwithin{equation}{thm}

\begin{document}

\title{Lie-Rinehart and Poisson algebras over $\cin$-rings}
\author{Eugene Lerman}
\address{Department of Mathematics, University of Illinois at Urbana-Champaign,  214 Harker Hall,
1305 W. Green St,
Urbana, IL 61801 USA}
\email{lerman@illinois.edu}

\author{Ruben Louis}
\address{Department of Mathematics, University of Illinois at Urbana-Champaign, 214 Harker Hall,
1305 W. Green Street, Urbana, IL 61801 USA}
\email{rlouis@illinois.edu}

\maketitle

\begin{abstract} 
We define the analogue of Lie-Rinehart algebras over $\cin$-rings. We show that given a Poisson $\cin$-ring $\scA$ its module $\Omega^1_\scA$ of $\cin$-K\"{a}hler differentials is (part of) a Lie-Rinehart algebra. Conversely, given a Lie-Rinehart algebra $\scM \xrightarrow{\rho} \CDer (\scA)$ over a $\cin$-ring $\scA$, there is a natural Poisson bracket on the $\cin$-ring $\cF(\scM)$  associated with the $\scA$-module $\scM$ (the $\cin$-ring analogue of an $\scA$-algebra freely generated by the module $\scM$). In the case where $\scA$ is the $\cin$-ring of smooth functions on a manifold $M$ and $\scM$ is the module $\Gamma(E)$ of sections of a Lie algebroid $E\to M$, the $\cin$-ring $\cF(\Gamma(E))$ is the ring of functions $\cin(E^{\svee})$ on the total space of the vector bundle $E^{\svee} \to M$ dual to the vector bundle $E$.
\end{abstract}
\tableofcontents
\section{Introduction}
We study Lie-Rinehart algebras in the context of differential geometry. Fair amount of work has already been done on this subject, notably by Huebschmann.  It was not entirely successful due to a disconnect between differential geometry and algebraic geometry over commutative rings.
The issue is this. 
It is well known that one can embed (contravariantly) the category $\Man$ of (second countable Hausdorff) manifolds into the category $\R$-$\mathsf{alg}$ of $\R$-algebras.  This fact is often referred to as Milnor's exercise; it is a theorem of Pursell \cite[Section 8]{Pursell}.  More precisely the global sections functor
\[
\Gamma: \op{\Man} \to \R\textrm{-}\mathsf{alg}, \qquad 
(M\xrightarrow{f}N) \mapsto (\cin(N)\xrightarrow{f^*} \cin(M))
\]
is fully faithful.
However there is no functor 
\[
\scS: \op{\R\textrm{-}\mathsf{alg}} \to \scC
\]
to some category $\scC$ containing the category of manifolds so that 
\[
\scS(\cin(N)\xrightarrow{f^*} \cin(M))\, \simeq\, (M\xrightarrow{f}N).
\]
The constructions coming from algebraic geometry over commutative rings do not work.  So if one starts with a Poisson manifold $(M, \pi)$ ($\pi$ is the Poisson bivector field) then the algebra $\cin(M)$ of smooth functions is a Poisson algebra over the reals, but its real spectrum is not, in general, the manifold $M$.   The problem is compounded by the fact that $\cin(\R^n)$ is not a very nice commutative ring. For example  it is not Noetherian. Nor is $\cin(\R^n)$ free as an $\R$-algebra.

Our preferred solution to the problems outlined above is to replace $\R$-algebras with $\cin$-rings (see Subsection~\ref{subsec:c-ring}).  We take the category $\scC$ to be the category $\LCRS$ of local $\cin$-ring spaces (see \cite{Joy}) into which the category $\Man$ of manifolds embeds. Thanks to a theorem of Dubuc \cite{Dubuc, Joy} there is a pair of adjoint functors
\[
\xy
  (-21,0)*+{\Gamma :\LCRS}="1";
  (-9,1)*+{}="3";
  (-9,-1)*+{}="5";
  (21,0)*+{\op{\cring}:\Spec}="2";
  (5,1)*+{}="4";
  (5,-1)*+{}="6";
  {\ar@{->}_{} "3";"4"}; {\ar@{->}_{} "6";"5"};
  \endxy,
\]
where $\Gamma$ is the global sections functor and $\Spec$ is Dubuc's spectrum functor. Moreover, for any manifold $M$, the affine scheme $\Spec(\cin(M))$ is the manifold $M$ itself (with its structure sheaf $\cin_M$ of smooth functions). This does solve the problem with Poisson manifolds and Poisson algebras: for any Poisson $\cin$-ring $(\scA, \Pb)$ the affine $\cin$-scheme $\Spec(\scA)$ is a Poisson scheme (see \cite{L-poisson}).  So in particular the functor $\Spec$ recovers  a Poisson manifold from its Poisson algebra of functions.  The results in \cite{L-poisson} also shed light on singular coisotropic reduction by allowing to attach a Poisson scheme to objects that until recently have only been viewed algebraically, such as, for example, the Sniatycki-Weinstein reduction.  Note that there is no restriction on the Poisson ring to be finite dimensional --- the functor $\Spec$ does not care --- $\Spec(\scA)$ makes sense for an inifnite dimensional Poisson $\cin$-ring $\scA$.

We note parenthetically that $\cin(\R^n)$ is very nice  as a $\cin$-ring: it is freely generated by the coordinate functions $x_1,\ldots, x_n:\R^n\to \R$. See \cite{MR} and Appendix~\ref{app:free}.

A similar problem arises with vector bundles and modules. Recall that by Serre-Swan theorem \cite{Jet} the module $\Gamma(E)$ of sections of a vector bundle $E\xrightarrow{p}M$ is a finitely generated projective module over $\cin(M)$. And conversely any finitely generated projective module $\scM$ over $\cin(M)$ is the module of global sections of some vector bundle $E\to M$.  It is not clear how to handle geometrically the modules that are not projective or not finitely generated.  Such modules arise, for example, in the study of singular foliations \cite{AS,LLL1}. 

Our approach to geometrizing arbitrary modules comes from the following two observations:
\begin{itemize}
    \item[1.] For any vector bundle $H\xrightarrow{p}M$ the homomorphism $p^*:\cin(M)\to \cin(H)$ makes $\cin(H)$ into a $\cin(M)$-algebra and
    \item[2.] For a vector bundle $E\to M$ the module $\Gamma(E)$ of global sections embeds into the $\cin(M)$-algebra $\cin (E^\svee)$ of smooth functions on the dual vector bundle $E^\svee \to M$.
\end{itemize}
The first main result of the paper is:
for any $\cin$-ring $\scA$ and any $\scA$-module $\scM$  there is a ``free" $\scA$-algebra $\cF(\scM)$; the reader may wish to think of $\cF(\scM)$ as a $\cin$-ring analogue of the symmetric algebra $S^\bullet (\scM)$. More precisely we prove in Theorem~\ref{prop1.1} that the forgetful functor 
\[
U: \scA\textrm{-}\mathsf{Alg} \to \scA\textrm{-}\Mod
\]
from $\scA$-algebras (technically speaking $\scA\textrm{-}\mathsf{Alg}$ is  the coslice category $\scA/\cring$) to $\scA$-modules has a left adjoint 
\[
\cF: \scA\textrm{-}\Mod \to \scA\textrm{-}\mathsf{Alg}.
\]

Moreover if $\scA = \cin(M)$ for a manifold $M$ and $\scM = \Gamma(E)$ for a vector bundle $E\to M$ then 
\[
\cF(\Gamma(E)) = \cin(E^\svee);
\]
see Theorem~\ref{thm:4.01}. So in particular if $\scM = \mathfrak X(M)$, the module of vector fields, then $\cF(\mathfrak X(M)) = \cin (T^\svee M)$. And if $\scM = \Omega^1(M)$, the module of de Rham 1-forms, then $\cF(\Omega^1(M)) = \cin(TM)$.

While the existence of the left adjoint $\cF$ can be proved by an application of the adjoint functor theorem, our proof is constructive and the $\cin$-ring $\cF(\scM)$ can be computed explicitly even $\scM$ is not projective nor finitely generated.  We demonstrate this in many examples.

Closely related to the issue of modules is the correspondence between Lie algebroids and Lie--Rinehart algebras.  Recall that any Lie algebroid $E\xrightarrow{\rho} TM$ over a manifold $M$ gives rise to a Lie-Rinehart algebra $\rho_*:\Gamma(E) \to \Gamma(TM) \simeq \Der (\cin(M))$. On the other hand there are naturally occurring Lie--Rinehart algebras, particularly in situations involving singularities, where the corresponding modules are not sections of vector bundle (i.e., they are not finitely generated projective modules). For instance, Poisson vector fields of a Poisson manifold may be viewed as a Lie--Rinehart algebra over the algebra of Casimir functions; similarly, the symmetries of a singular foliation form a Lie--Rinehart algebra over the algebra of functions constant along the leaves \cite{RL,LLL1}. Other examples include singular subalgebroids \cite{Zambon} and vector fields on affine varieties. In differential geometry, generic Lie–Rinehart subalgebras of vector fields often exhibit pathological behavior, and relatively little can be said about their structure in general. 

To bring in the technology of $\cin$-rings, we  define the analogue of Lie-Rinehart algebras and their morphisms  over $\cin$-rings (see Subsection~\ref{subsec:L-R}). It involves replacing commutative rings with $\cin$-rings and derivations with $\cin$-ring derivations.  As a sanity check we prove that given a Poisson $\cin$-ring $\scA$ its module $\Omega^1_\scA$ of $\cin$-K\"{a}hler differentials is (part of) a Lie-Rinehart algebra (see Theorem~\ref{thm:4.1}).  As a more interesting result, we  show that given a Lie-Rinehart algebra $\scM \xrightarrow{\rho} \CDer (\scA)$ over a $\cin$-ring $\scA$, there is a natural Poisson bracket on the $\cin$-ring $\cF(\scM)$  --- see Theorem~\ref{thm:main}. Conversely, given a \(C^\infty\)-ring \(\scA\), an \(\scA\)-module \(\scM\), and a linear Poisson structure on \(\cF(\scM)\) (in the sense of Definition~\ref{def:lin_Poisson}), there exists a Lie--Rinehart algebra structure $\scM \xrightarrow{\rho} \CDer (\scA)$ over the $\cin$-ring $\scA$. Moreover, the last two constructions are in one-to-one correspondence.

We believe that our results  have ramifications in the framework of singular foliations in the sense of \cite{AS, LLL1} and the references therein, where to every singular foliation one associates a Poisson structure encoding its singularities. This perspective should allow one to reformulate questions about singular foliations in the language of Poisson geometry. For instance, the problem of integrability of a singular foliation into a Lie algebroid can be interpreted as a Poisson extension problem from the viewpoint of Poisson geometry. Our approach also applies more generally to singular subalgebroids in the sense of \cite{Zambon}. {When such a subalgebroid is finitely generated, the associated Poisson structure is induced by an \(\mathcal I\)-Poisson manifold in the sense of \cite[Definition~4.3]{NS}.} These ramifications will be developed elsewhere in a forthcoming work. In addition, in a separate forthcoming paper, we shall study the analogue of this problem for Lie \(\infty\)-algebroids, in particular for the universal Lie \(\infty\)-algebroid associated with a Lie--Rinehart algebra in the sense of \cite{LGL}.\\

The paper is organized as follows. We start by briefly reviewing some $\cin$-ring theoretic background in Section~\ref{sec:prelim}, recall the definition of a Poisson $\cin$-ring
and introduce the analogue of Lie-Rinehart algebras for $\cin$-rings.  These are tuples 
$((\scM, [\cdot\,, \cdot]_\scM),  \rho: \scM\to \CDer(\scA))$, where $\scA$ is a $\cin$-ring, $\scM$ is an $\scA$-module which is also a real Lie algebra, $\CDer(\scA)$ is the $\scA$-module of $\cin$-ring derivations of $\scA$ (Definition~\ref{def:der2}) and $\rho$ is a map of $\scA$-modules and of real Lie algebras.  Additionally, the Lie bracket on $\scM$ is compatible with the anchor $\rho$ in the usual way. See Definition~\ref{def:cin-LR}.  Naturally every Lie algebroid $E\to TM$ over a manifold $M$ gives rise to a Lie-Rinehart algebra over the $\cin$-ring $\cin(M)$. %
This is because for a manifold $M$ a derivation of the commutative $\R$-algebra $\cin(M)$ is also an honest differential operator hence a $\cin$-ring derivation: $\Der(\cin(M)) = \CDer(\cin(M))$.
In Section~\ref{sec:free1} for a fixed $\cin$-ring $\scA$,  we explicitly construct a left adjoint $\cF:\scA\textrm{-}\Mod \to \scA/\cring$ to the forgetful functor $U: \scA/\cring \to \scA\textrm{-}\Mod$. Here and elsewhere $\scA\textrm{-}\Mod $ denotes the category of $\scA$-modules and $\scA/\cring$ denotes the category of $\scA$-algebras. In Section~\ref{sec:free2} we show that the functor $\cF$ sends the module $\Gamma(E)$ of sections of a vector bundle $E\to M$ to the $\cin$-ring of smooth functions $\cin(E^\svee)$ on the dual bundle $E^\svee\to M$.

In Section~\ref{sec:5}
we show that for a Poisson $\cin$-ring $\scA$ (Definition~\ref{def:Poisson_ring}) there is a Lie bracket on the module $\Omega^1_\scA$ of $\cin$-K\"{a}hler differentials (Definition~\ref{def:Kahler}) and a map of $\scA$-modules $\rho: \Omega^1_\scA \to \CDer(\scA)$ that makes $((\Omega^1_\scA, [\cdot\,, \cdot]_{\Omega^1_\scA}),  \rho: \Omega^1_\scA\to \CDer(\scA))$ into a $\cin$-Lie-Rinehart algebra.  See Theorem~\ref{thm:4.1}.  The special case of this theorem is the fact that for a Poisson manifold $(P,\pi)$ there is a Lie bracket on the module $\Omega^1(P)$ of differential 1-forms and that $\pi^ {\sharp} : T^ {\ast} P \longrightarrow TP$ is a Lie algebroid.  The analogous result in commutative algebra was proved by Huebschmann, who showed that for a Poisson algebra $(A, \{\cdot\,, \cdot\})$ there is a Lie bracket on the module $\Omega^1_A$ of ordinary K\"{a}hler differentials and a map of modules $\pi^\#: \Omega^1_A \to \Der (A)$ making $((\Omega^1_A, [\cdot\,, \cdot]_{\Omega^1_A}), \pi^\#: \Omega^1_A \to \Der (A))$ into a Lie-Rinehart algebra \cite[Theorem~3.8]{Hueb}.

In Section~\ref{sec:6} we generalize a theorem of Ted Courant \cite{courant} who proved that for any Lie algebroid $r: E\to M$ over a manifold $M$ there is a linear Poisson bracket on the total space of the dual vector bundle $E^{\svee}\to M$.  The analogous result for a Lie-Rinehart algebra $((\scM, [\cdot\,, \cdot]_\scM), \rho:\scM\to \Der(A))$ with $A$ a commutative ring is that the symmetric algebra $S^\bullet(\scM)$ of the $A$-module $\scM$ is a Poisson algebra \cite[Example 3.16]{Hueb}.  

We prove that for any $\cin$-ring $\scA$ and for any Lie-Rinehart algebra 
$((\scM, [\cdot\,, \cdot]_\scM), \rho: \scM\to \CDer(\scA))$
 the free $\cin$-ring $\cF(\scM)$ on the module $\scM$ is naturally a Poisson $\cin$-ring --- see Theorem~\ref{thm:main}.  

 In the last section, Section~\ref{sec:7}, we touch upon cotangent bundles of quotients of manifolds by group actions.  The reader may wish to compare our approach with the approach in \cite{HPR}.  Given an action of a compact Lie group $G$ on a manifold $M$ we describe three, in generally different,  Poisson $\cin$-rings that may reasonably  be called $\cin(T^\svee M/G)$, the algebra of ``functions" on the ``cotangent bundle" of the orbit space $M/G$.

 We end the paper with several technical appendices.  In Appendix~\ref{app:free} we  review a well-known construction of the free $\cin$-ring functor $\bbF: \Set \to \cring$ (the left adjoint to the forgetful functor $\bbU: \cring \to \Set$). We show that the $\bbF(X)$-module $\Omega^1_{\bbF(X)}$ of K\"ahler differentials of a free $\cin$-ring $\bbF(X)$ is freely generated by the set of symbols $\{dx\}_{x\in X}$ (see Remark~\ref{rmrk:free}) and use this to prove that a skew-symmetric biderivation on a  free $\cin$-ring $\bbF(X)$ (cf.\ Remark~\ref{rmrk:3.12a}) can be defined and is uniquely determined by this choice by specifying its value on generators.  See Theorem~\ref{thm:bider_gen}.  

 In Appendix~\ref{app:B} we show that for a $\cin$-ring $\scA$, a skew symmetric $\cin$-biderivation $\lbr\cdot\,, \cdot \rbr: \scA \times \scA \to \scA$ is a Poisson bracket if and only if the corresponding { Jacobiator} $J_{\lbr\cdot\,,\, \cdot \rbr}$ vanishes identically on the generators of the $\cin$-ring $\scA$.  

 In Appendix~\ref{app:counter-ex} we provide an example of a $\cin$-ring $\scA$ and an $\R$-algebra derivation $X:\scA\to \scA$  which is not a \(C^\infty\)-derivation, see Lemma~\ref{lem:C1}. Thus 
$\Der (\scA)\nsubseteq  \CDer(\scA)$.
To the best of our knowledge,  this is the first explicit example of this kind.

\smallskip

{\bf Acknowledgments.} The authors would like to thank Rui Loja Fernandes for several valuable discussions during the early stages of this work.
 
\section{Background}\label{sec:prelim}
\subsection{$\cin$-rings} \label{subsec:c-ring}\mbox{}\\[4pt]
Recall that a {\sf $\cin$-ring} is a product preserving functor from the
category of coordinate vector spaces $\R^n$, $n\in \N$, and smooth maps to the
category of sets \cite{MR, Joy}.  Equivalently a $\cin$-ring is a set
$\scA$ together with a collection of $n$-ary operations, $n\geq 0$,
one $n$-ary operation
\[
  f_\scA:\scA^n\to \scA
\]
for each smooth function $f\in \cin(\R^n)$. These operations are required to be associative. Additionally, one requires that if $\pr_i: \R^n\to \R$ is the projection on the $i$th factor then the corresponding operation $(\pr_i)_\scA : \scA^n \to \scA$ is given by
\begin{equation} \label{eq:201}
(\pr_i)_\scA (a_1,\ldots, a_n) = a_i.   
\end{equation}

Maps between $\cin$-rings are natural transformation
between functors.  Equivalently they are maps of sets preserving the
operations. That is, $\varphi:\scA \to
\scB$ is a map (morphism) of $\cin$-rings if for every $n\geq 0$, all $f\in
\cin(\R^n)$ and all $a_1,\ldots, a_n\in \scA$
\[
\varphi(f_\scA (a_1,\ldots, a_n) )= f_\scB (\varphi(a_1),\ldots, \varphi(a_n)).
\]
Here, as above, $f_\scA:\scA^n\to \scA$ denotes the operation
defined by the function $f$ (our notation is different from the one used by Joyce \cite{Joy}).

Since the identity map $\id_\scA$ preserves the operations on a $\cin$-ring $\scA$ and since the composite of two operation-preserving maps is operation-preserving, $\cin$-rings form a category.
\begin{notation} We denote the category of $\cin$-rings by $\cring$. 
\end{notation}
\begin{remark}
Note that maps of $\cin$-rings are automatically unital. This is because they preserve the nullary operation corresponding to the function $f:\R^0 =\{*\} \to \R$, $f(*) = 1$; this nullary operation on a $\cin$-ring $\scA$ is the unity $1_\scA$.
\end{remark}

\begin{example}\label{ex:smooth-functions}
For any manifold $M$ the algebra $\cin(M)$ of smooth functions is
naturally a $\cin$-ring: given $f\in \cin(\R^n)$ the corresponding $n$-ary
operation
\[
f_{\cin(M)} :\cin(M)^n\to \cin(M)
\]
is given by
\[
f_{\cin_M  } (a_1, \ldots, a_n) := f\circ (a_1,\ldots, a_n) \quad \textrm{ for all 
$a_1, \ldots, a_n\in \cin(M)$}.
\]
In particular $\cin(\R^n)$ is a $\cin$-ring for any $n\geq 0$.  In
fact $\cin(\R^n)$ is a {\sf  free} $\cin$-ring (see Appendix \ref{app:free}) on the generators
$x_1,\ldots, x_n: \R^n\to \R$. 

Note that $\R = \cin(\R^0)$ is a free
$\cin$-ring on the empty set of generators.  Note also that for the $\cin$-ring  $\R$ the $n$-ary operations are tautological: for any $n\geq 0$, for any $f\in \cin(\R^n)$ the operation $f_\R:\R^n\to \R$ is the function $f$:
\[
f_\R(a_1,\ldots, a_n) := f(a_1,\ldots, a_n).
\]    
\end{example}

\begin{remark}
The $\cin$-ring $\R$ is initial in the category $\cring$: for any $\cin$-ring $\scA$ there is a unique map of $\cin$-rings $\varphi:\R\to \scA$ with $\varphi (1) = 1_\scA$.
\end{remark}

\begin{remark}
It follows quickly from the definition that any
$\cin$-ring has an underlying $\R$-algebra. The $+$ on the $\R$-algebra underlying a $\cin$-ring $\scA$ comes from the function $f(x,y) = x+y$, the $\cdot$ from the function $f(x,y) = xy$ and multiplication by a scalar $\lambda\in \R$ comes from $f(x) = \lambda x$.

We might not notationally distinguish between a $\cin$-ring and its underlying $\R$-algebra.
\end{remark}

\subsection{Free objects and generators} \label{subsec:free} \mbox{}\\[4pt]
If a functor $U:\scD\to \scC$ has a left adjoint $F:\scC\to \scD$, we think of an object $F(c)$ of the category $\scD$ as being freely generated by an object $c\in \scC$.  We think of the component
\[
\eta_c : c \to UF(c)
\]
of the unit of adjunction $\eta: \id \Rightarrow UF$ as a map from generators to $F(c)$.

For example, for any $\cin$-ring $\scA$ there is a forgetful functor
\[
\mathbf{U}:\scA\textrm{-}\Mod \to \Set
\]
from the category $\scA\textrm{-}\Mod$ of $\scA$-modules (Definition~\ref{def:module} and Notation~\ref{not:A-mod}) to sets.  The functor $\mathbf{U}$ has a left adjoint $\Free: \Set\to \scA\textrm{-}\Mod$ (Remark~\ref{rmrk:free}). In this case the components of the unit of adjunction are injective, and for a given set $I$ we have a map 
\[
I\to \Free(I)
\]
(really to $\mathbf{U}(\Free(I))$ but we suppress $\mathbf{U}$) that embeds $I$ as a basis of the module $\Free(I)$.\\

There is a forgetful functor $\bbU: \cring \to \Set$ from the category of $\cin$-rings to sets.  It has a left adjoint
\[
\bbF:\Set \to \cring
\]
(see Appendix~\ref{app:free}).  Given a set $X$ we think of the $\cin$-ring $\bbF(X)$ as being freely generated by the set $X$.  In this case we denote the unit of adjunction by $\delta$:
\[
\delta_X:X\to \bbU(\bbF(X)),
\]
which happens to be injective. We therefore may suppress the map $\delta_X$ and think of the set $X$ as being contained in the $\cin$-ring $\bbF(X)$.  Thus $X\subset \bbF(X)$ is a set of generators of the $\cin$-ring $\bbF(X)$.\\

For a $\cin$-ring $\scA$ there is a forgetful functor 
\[
U: \scA/\cring \to \scA\textrm{-}\Mod
\]
from the category $\scA/\cring$ of $\scA$-algebras (Definition~\ref{def:A-alg}) to $\scA$-modules.  The functor $U$ has a left adjoint
\[
\cF :\scA\textrm{-}\Mod \to \scA/\cring;
\]
see Section~\ref{sec:free1}. Again for any $\scA$-module $\scM$ the component of the unit of the adjunction
\[
\eta_\scM:\scM \to U(\cF(\scM))
\]
is injective (Lemma~\ref{lem:3.06}), and we can think of the $\scA$-algebra $\cF(\scM)$ as being freely generated by the module $\scM$.

\begin{definition}
 Given a $\cin$-ring $\scA$, a set $X$ and a surjective map $\Pi:\bbF(X)\to \scA$ of $\cin$-rings, (where as above and as in Appendix~\ref{app:free} the $\cin$-ring  $\bbF(X)$ is the $\cin$-ring freely generated by the set $X$) we say that the set $X$ {\sf generates} the $\cin$-ring $\scA$.
  
 If the set $X$ can be chosen to be finite, we say that the $\cin$-ring $\scA$ is {\sf finitely generated}.
\end{definition}

\begin{remark}
    For a $\cin$-ring $\scA$ there is a surjective map $\Pi:\bbF(\scA)\to \scA$. So any $\cin$-ring is generated by {\em some} set.
\end{remark}

\begin{remark}
    A $\cin$-ring $\scA$ is finitely generated if and only if it is isomorphic to the quotient $\cin(\R^n)/J$ for some ideal $J$ (cf.\ Remark~\ref{rmrk:congr}).
\end{remark}

\begin{definition} \label{def:module}
A {\sf module } over a $\cin$-ring $\scA$ is a module over the
underlying $\R$-algebra (cf.\ \cite{Joy, LdR}).
\end{definition}

\begin{remark} \label{rmrk:congr}
    Any $\cin$-ring $\scA$ is a module over $\scA$.  Any submodule $J$ of $\scA$ is an ideal in (the $\R$-algebra underlying) $\scA$.  It is a somewhat surprising but elementary fact for any ideal $J$ of a $\cin$-ring $\scA$ the quotient $\R$-algebra $\scA/J$ is a $\cin$-ring.  See \cite{MR, Joy}.
\end{remark}
\begin{example} \label{ex:2.16}
Let $M$ be a manifold and $Z\subset M$ a closed subset.  Recall that a
Whitney smooth function on $Z$ is the restriction of a function $f\in
\cin(M)$ to $Z$.   We write $\cin(Z)$ for the set of all Whitney
smooth functions on $Z$.  It is a standard fact that $\cin(Z)$ is an
$\R$-algebra which is the quotient of $\cin(M)$ by the ideal of smooth
functions that vanish on $Z$.  By Remark~\ref{rmrk:congr}, the
$\R$-algebra $\cin(Z)$ is a $\cin$-ring.    The fact that $\cin(Z)$ is
a $\cin$-ring can also be shown directly.
\end{example}
\begin{notation}\label{not:A-mod}
    The collection of all $\scA$-modules form a category that we denote by $\scA$-${\Mod}$.
\end{notation}

\begin{remark}[Modules, Beck modules and square zero extensions]\label{rmrk:m-bm-sz}
A module $\scM$ over a
$\cin$-ring $\scA$ is a Beck module \cite{Beck, Barr}. Namely,  just as in the
case of commutative rings,  given a $\cin$-ring $\scA$ and an
$\scA$-module $\scM$, there is a $\cin$-ring structure on the product  $\scA\times \scM$. We denote it by $\scA \ltimes \scM$ and refer to it as the {\sf square zero extension of the $\cin$-ring $\scA$ by the module $\scM$}. (The notation $\ltimes$ has nothing to do with the semi-direct products of groups.) Then   the projection $p:\scA\ltimes \scM\to \scA$ on the
first factor makes $\scA \ltimes \scM$ into an Abelian group object in the slice category
$\cring/\scA$ of $\cin$-rings over $\scA$.  
In more detail the $C^\infty$-ring structure on $\scA\ltimes \scM$ is as follows. For a function
$f \in C^\infty(\mathbb{R}^n)$ and 
$(a_1,m_1),\dots,(a_n,m_n)\in \scA\ltimes \scM$ the corresponding operation is 
\begin{equation}
f_{\scA\ltimes \scM}((a_1,m_1),\dots,(a_n,m_n)) := 
(f_\scA(a_1,\dots,a_n), 
\sum_{i=1}^n 
(\partial_i f)_{\scA}(a_1,\dots,a_n)\,\cdot\, m_i).
\end{equation}
In particular, by taking $n=2$ and  $f(x,y) = xy$ we get the multiplication on $\scA \ltimes  \scM$:
\begin{equation}
    (a, m) \cdot (a', m') := (aa', \, am' + a'm):
\end{equation}
And by taking $f(x,y) = x+y$ we get the addition map 
\[
+: (\scA\ltimes \scM)\times(\scA\ltimes \scM)  \to \scA\ltimes \scM, \qquad 
((a_1, m_1), (a_2, m_2)) \mapsto (a_1+a_2, m_1+ m_2).
\]
 \end{remark}  

\begin{remark} \label{rmrk:free}
For an ordinary $\R$-algebra $R$, forgetful functor from $R$-modules to sets has a left adjoint. Consequently, for any $\cin$-ring $\scA$ the forgetful functor $\mathbf{U}:\scA\textrm{-}\Mod \to \Set$ has a left adjoint $\Free_\scA:\Set\to \scA\textrm{-}\Mod$.  

We may  also denote the $\scA$-module freely generated by a set $I$ by $\scA^{\oplus I}$, since it is a direct sum of $I$-copies of the free rank 1 $\scA$-module $\scA$.  Thus 
\[
\Free_\scA(I) = \oplus_I \scA \equiv \scA^{\oplus I}.
\]
\end{remark}
\begin{definition}\label{def:der2}
Let $\scM$ be a module over a $\cin$-ring $\scA$.  A {\sf $\cin$-ring
derivation} (or a {\sf $\cin$-derivation} for short) with values in
the module 
$\scM$ is an $\R$-linear map $v:\scA\to \scM$ so that for any $n$, any $f\in
\cin(\R^n)$ and any $a_1,\ldots, a_n\in \scA$
\[
v\left(f_\scA(a_1,\ldots,a_n)\right) = \sum_{i=1}^n (\partial_i
  f)_\scA(a_1,\ldots, a_n)\cdot v (a_i).
\]  
\end{definition}

\begin{remark}
    Note that $\cin$-ring derivations of $\scA$ with values in a module $\scM$ correspond exactly to maps from $\scA\xrightarrow{\id}\scA$ to the square zero extension $\scA\ltimes \scM\xrightarrow{p} \scA$ in the slice category in $\cring/\scA$, as they should.
\end{remark}
\begin{example}
    Let $M$ be a manifold.  Then a vector field $v$ on $M$ thought of as map 
    $v:\cin(M)\to \cin(M)$ is a $\cin$-ring derivation (and not just an $\R$-algebra derivation):
\[
v( f\circ (a_1,\ldots, a_n) ) = \sum_{i=1}^n \big(\partial_i
f\circ (a_1,\ldots, a_n) \big)\cdot v(a_i).
\]
for $n>0$, $f\in \cin(\R^n)$ and $a_1, \ldots, a_n\in \cin(M)$
\end{example}
    
\begin{example} \label{ex:exter_der}
    Let $M$ be a manifold. Then the exterior derivative $d:\cin(M)\to \Omega^1(M)$ is a $\cin$-ring derivation with values in the module $\Omega^1(M)$ of ordinary differential forms.
\end{example}

\begin{notation} \label{nota:cder}
    We denote the set of all $\cin$-ring derivations of a $\cin$-ring $\scA$ with values in a module $\scM$ by $\CDer(\scA, \scM)$.  In the case where $\scM = \scA$ we write $\CDer(\scA)$.
\end{notation}
\begin{remark}
The set $\CDer(\scA, \scM)$ of derivations of a $\cin$-ring $\scA$ with values in a module $\scM$ is an $\scA$-module with the module operations defined ``point-wise:"
\[
(a \cdot v) (b):= a \cdot v(b)
\]
and 
\[
(v+w)(b) := v(b) + w(b)
\]
for all $a,b\in \scA$ and $v,w\in \CDer(\scA, \scM)$.
\end{remark}

Example~\ref{ex:exter_der} generalizes to arbitrary $\cin$-rings.
\begin{definition}\label{def:Kahler}
 A {\sf module of K\"ahler
differentials} over a $\cin$-ring $\scA$ is an $\scA$-module $\Omega^1_\scA$
together with a $\cin$-ring derivation $d\equiv d_\scA:\scA\to
\Omega^1_\scA$ so that for any $\scA$-module $\scN$ and any $\cin$-ring derivation
$X:\scA\to \scN$ there exists a unique map of modules $\varphi_X:
\Omega^1_A\to \scN$ with
\[
\varphi_X \circ d = X.
 \] 
\end{definition}

\begin{remark}
The universal property of the derivation $d:\scA\to \Omega^1_\scA$ says
that for any $\scA$-module $\scM$, the map of modules
\[
  d^*:\Hom_{\scA\textrm{-}\Mod} (\Omega^1_\scA, \scM)\to \Der (\scA, \scM),\quad
  d^*(\varphi): = \varphi \circ d 
\]
is an isomorphism (of $\scA$-modules).
\end{remark}  

\begin{lemma}[Dubuc and Kock]
 \label{thm:DK}
For any $\cin$-ring $\scA$ there exists a module of K\"ahler
differentials $\Omega_\scA^1$ together with the structure map
$d\equiv d_\scA:\scA\to \Omega_\scA^1$.
\end{lemma}
\begin{proof}
    See \cite{DK}.  For an alternative proof see \cite{Joy}.
\end{proof}

\begin{remark} \label{rmrk:2.17}
    For any manifold $M$ the module of K\"ahler differentials $\Omega^1_{\cin(M)}$ is the module of de Rham 1-forms $\Omega^1(M)$.  This is mentioned without proof in \cite{Joy} and proven in \cite{FSJ}.
\end{remark}

\begin{remark}
    The $\scA$-module $\Omega^1_\scA$ is generated (as an $\scA$-module) by the set $\{da\}_{a\in \scA}$ of ``exact" K\"ahler differentials.
\end{remark}
\begin{remark}\label{alg_Kahler}
    Let \(\scA\) be a \(C^\infty\)-ring. Viewing \(\scA\) merely as an \(\mathbb R\)-algebra, one may consider its module of \(\mathbb R\)-algebraic Kähler differentials, which we denote by
\(
d^{\mathrm{alg}}\colon \scA\to\Omega^1_{\scA,\mathrm{alg}}
\). This object should be distinguished from the module of $\cin$-algebraic Kähler differentials 
\(
d\colon \scA\to \Omega^1_{\scA}
\)
of \(\scA\).  
In the special case where
\(
\scA=C^\infty(M),
\)
for a smooth manifold \(M\), one has
\(
\Omega^1_{\scA}=\Omega^1(M),
\)
the projective \(C^\infty(M)\)-module of smooth differential \(1\)-forms on \(M\) (cf.\ Remark~\ref{rmrk:2.17}). There is a canonical morphism modules
\(
\Omega^1_{C^\infty(M),\mathrm{alg}}
\rightarrow
\Omega^1_{C^\infty(M)}
=
\Omega^1(M).
\)
This morphism is not an isomorphism. For instance, for $M=\mathbb R$, the identity
\[
d^{\mathrm{alg}}(\sin t)=\cos(t)\,d^{\mathrm{alg}}t
\]
does not hold in \(\Omega^1_{C^\infty(\mathbb R),\mathrm{alg}}\), see \cite{Osborn}. More generally, Osborn proved that
\(
d^{\mathrm{alg}}(f(t))=f'(t)\,d^{\mathrm{alg}}t
\)
holds in \(\Omega^1_{C^\infty(M),\mathrm{alg}}\) if and only if \(f(t)\) is an  algebraic function of $t$. Consequently, the universal algebraic derivation
\[
d^{\mathrm{alg}}:C^\infty(\mathbb R)\longrightarrow
\Omega^1_{C^\infty(\mathbb R),\mathrm{alg}}
\]
is an \(\mathbb R\)-derivation which is not a \(C^\infty\)-derivation. Furthermore, the module
\(
\Omega^1_{C^\infty(M),\mathrm{alg}}
\)
can be extremely large. Indeed, G\'omez  \cite{Gomez} proved that if the manifold \(M\) has dimension at least \(1\), then the cardinality of any generating set of
\(
\Omega^1_{C^\infty(M),\mathrm{alg}}
\)
as a \(C^\infty(M)\)-module is at least the cardinality of \(\mathbb R\).
\end{remark}

\begin{remark} \label{rmrk:Omega_functor}
Just as in the case of commutative rings the assignment $\scA \mapsto \Omega^1_\scA$ is functorial in the $\cin$-ring $\scA$: given a map $\varphi: \scA\to \scB$ of $\cin$-rings  there is a (unique) map $\Omega^1_\varphi: \Omega^1_\scA \to \Omega^1_\scB$ making the diagram
\begin{equation} \label{eq:2.12.1}
\xy
(-10,10)*+{\scA}="1";
(15,10)*+{\scB}="2";
 (-10,-8)*+{\Omega^1_\scA}="3";
(15,-8)*+{\Omega^1_\scB}="4"; 
{\ar@{->}_{d_\scA} "1";"3"};
{\ar@{->}^{d_\scB} "2";"4"};
{\ar@{->}^{\varphi} "1";"2"};
{\ar@{->}_{\Omega^1_\varphi} "3";"4"};
\endxy
\end{equation}
commute.  This is because $d_\scB\circ \varphi: \scA \to \Omega^1_\scB$ is a derivation of $\scA$ and therefore has to factor uniquely through $\Omega^1_\scA$.
\end{remark}

\begin{definition} \label{def:2.270} Let $\Omega^1_\scA$ be the module of
  $\cin$-K\"ahler differentials of a $\cin$-ring $\scA$.  For $k>0$ we
 denote the  $k$-th exterior power of the $\scA$-module
 $\Omega^1_\scA$ by $\Lambda^k (\Omega^1_\scA)$.
Note that the wedge is taken over the $\R$-algebra $\scA$, not over
$\R$.  By convention 
\[
\Lambda^0 (\Omega^1_\scA):= \scA.
 \]
 We set
 \[
\Lambda^\bullet\Omega^1_\scA := \{ \Lambda^k\Omega^1_\scA\}_{k\geq 0}
\]
which is a graded-commutative algebra (often referred to as commutative graded algebra or a CGA) over the $\cin$-ring $\scA$.
\end{definition}

We now turn $\Lambda ^\bullet \Omega^1_\scA$ into a commutative 
  {\em differential}  graded algebra (a CDGA).
\begin{theorem}\label{thm:6.2}
Let $\scA$ be a $\cin$-ring.  The universal differential $d_\scA:\scA
\to \Omega^1_\scA$ extends to unique degree 1 $\R$-linear map 
\[
d: \Lambda^\bullet\Omega^1_\scA \to \Lambda^{\bullet+1}\Omega^1_\scA 
\]
so that for all $k>0$ and all $a, b_1,\ldots, b_k\in \scA$
\begin{equation} \label{eq:6.d}
d(a \, d_\scA b_1\wedge  \ldots \wedge d_\scA b_k) = d_\scA a \wedge d_\scA b_1\wedge  \ldots \wedge d_\scA b_k.
\end{equation}
Consequently $d\circ d =0$.
\end{theorem}

\begin{proof}
    This is \cite[Theorem 4.2]{LdR}
\end{proof}

There is also a version of Cartan calculus for $\cin$-rings \cite{LCar} which we will need.  We now briefly review it.

\begin{definition}[The contraction map]
    For any module $\scM$ over an $\R$-algebra $A$ a homomorphism
$\varphi:\scM =\Lambda^1\scM \to A =\Lambda^0\scM$ extends uniquely to a degree -1 derivation $
\varphi^\bullet:\Lambda^\bullet \scM\to \Lambda^{\bullet -1} \scM $ of the 
CGA $(\Lambda^\bullet \scM, \wedge)$.   Consequently, there is an
isomorphism of $A$-modules
\[
 \Hom(\scM, A) \xrightarrow{\, \simeq \,} \Der^{-1} (\Lambda^\bullet \scM).
\]  
Since for any $\cin$-ring $\scC$ the module of $\cin$-derivations
$\cin\Der (\scC)$ is isomorphic, as a $\scC$-module, to
$\Hom(\Omega^1_\scC, \scC)$ (cf.\ \eqref{eq:2.19.2}) there is a canonical injective map
\[
\iota_\scC: \cin\Der (\scC) \to \Der^{-1} (\Lambda^\bullet\Omega^1_\scC),
\]
the {\sf contraction map}.

The contraction map $\iota_\scC$ is a map of $\scC$-modules.
\end{definition}

\begin{definition}[Lie derivative] \label{def:Lie_der}
 Given  a derivation $v\in \cDer(\scC)$ of a $\cin$ ring $\scC$ we   {\em define}
the {\sf Lie derivative} $\cL_\scC(v): \Lambda^\bullet \Omega^1_\scC  \to
\Lambda^\bullet \Omega^1_\scC $ by the Cartan magic formula:
\[
\cL_\scC(v) := \iota_\scC(v) \circ d + d\circ \iota_\scC(v).  
\]
The Lie derivative $\cL_\scC(v)$ is a degree 0 derivation of the CGA
$(\Lambda^\bullet\Omega^1_\scC, \wedge)$  for every $v\in \cin\Der (\scC)$.
We thus obtain an $\R$-linear map
\[
\cL_\scC: \cin\Der (\scC) \to \Der^{0} (\Lambda^\bullet\Omega^1_\scC),
\qquad v\mapsto \cL_\scC(v) = \iota_\scC(v) \circ d + d\circ \iota_\scC(v),
\]  
which we call the {\sf Lie derivative map. }     
\end{definition}

We are now in position to formulate Cartan
calculus on a $\cin$-ring.  We drop the subscript ${\, }_\scC$ to reduce
the clutter.

\begin{theorem} \label{prop:cartan_calc_level0}
For any $\cin$-ring $\scC$ the contraction $\iota: \cin\Der (\scC) \to \Der^{-1}
(\Lambda^\bullet\Omega^1_\scC)$  map and the Lie derivative $\cL:  \cDer(\scC)
\to \Der^{0} (\Lambda^\bullet\Omega^1_\scC)$ map, which were
described above,  satisfy the
following identities: 
\begin{align*}
&(\textrm{i})\quad \cL(v) \circ d = d \circ \cL(v) \qquad \qquad \qquad \qquad (\textrm{ii}) &\cL([v,w]) = [\cL(v), \cL(w)]\\
&(\textrm{iii})\quad[\cL(v), \iota(w)] = \iota([v,w]) \qquad \qquad \qquad (\textrm{iv}) &[\iota(v),
                                                      \iota(w)] = 0\\
&(\textrm{v}) \quad \cL(v) = [d, \iota(v)]
\end{align*}
for all $v,w\in \cDer (\scC)$.
\end{theorem}

\begin{proof}
    This \cite[Theorem~1.1]{LCar}.
\end{proof}

\subsection{Poisson $\cin$-rings}
\begin{definition} \label{def:Poisson_ring}
A {\sf Poisson $\cin$-ring} is a $\cin$-ring $\scA$ together with a
Poisson bracket $\{\cdot\,, \cdot\}: \scA\times \scA \to \scA$ on the
underlying $\R$-algebra so that for any $a\in \scA$ the map
\[
ad(a)\equiv \{a, \cdot \}: \scA\to \scA, \quad b\mapsto \{a, b\}
\]  
is a $\cin$-ring derivation: for any $n>0$, any $f\in \cin(\R^n)$ and any $a_1,\ldots, a_n\in
\scA$
\begin{equation} 
  \left\{a, f_\scA(a_1,\ldots,a_n) \right\} = \sum_{i=1}^n (\partial_i
  f)_\scA(a_1,\ldots, a_n)\cdot \{ a, a_i\}.
 \end{equation} 
\end{definition}

\begin{example}
Let $(P, \pi\in \Gamma \Lambda^2 TM)$ be a Poisson manifold and
$\{f,g\} = \langle df\wedge dg , \pi\rangle $ the corresponding
bracket.  Then $(\cin(P), \{\cdot\,, \cdot \})$ is a $\cin$-Poisson
ring.  This is because $ad(f)$ is a vector field for every $f\in \cin(P)$.
\end{example}

\begin{definition} \label{def:Poisson_map}
A {\sf map} of Poisson $\cin$-rings is a map of $\cin$-rings that
preserve the Poisson brackets.
\end{definition}

\begin{definition}[A category $\pcring$ of Poisson $\cin$-rings] \label{def:pscring} 
 There is an ``evident" category $\pcring$ of {\sf Poisson $\cin$-rings}. An object of $\pcring$ is a pair $(\scA, \{\cdot\,, \cdot\})$ where $\scA$ is a $\cin$-ring and $\{\cdot\,, \cdot\}:\scA\times \scA \to \scA$ is a Poisson bracket.  A morphism (a map) in $\pcring$ is a map (homomorphism) of $\cin$-rings preserving the brackets.   
\end{definition}

\begin{remark}
The  scare quotes around the word {\sf evident} in Definition~\ref{def:pscring} are there because there should be a more general notion of morphism of $\cin$-rings.  Indeed, in case of Poisson manifolds, there are Poisson maps and there are Poisson {\em relations}.  So there should be analogues of Poisson relations for $\cin$-rings.  These analogues would have to be co-relations in the category of $\cin$-rings (with an extra structure to make them ``coisotropic")  since the functor  which embeds Poisson manifolds and Poisson maps into $\pcring$ is contravariant. We will address this elsewhere.
\end{remark}

We note an easy result that describes a relation between the categories $\PMan$ of Poisson manifold and $\pcring$ of Poisson $\cin$-rings.  Recall that the objects of the category $\PMan$ of Poisson manifolds are Poisson manifolds and morphisms are smooth maps that preserves their Poisson structures.  This means that if $f$ and $g$
are smooth functions on the target manifold, the Poisson map $\phi:M\to N$
 satisfies:
\[
\{f\circ \phi , g\circ \phi \}_M=\{f,g\}_N\circ \phi.
\]
\begin{lemma} \label{lem:pman-pring}
The global sections functor 
    \[
    \Gamma: \op{\PMan} \to \pcring
    \]
from the category of Poisson manifolds $\PMan$ described above to the category $\pcring$ of Poisson $\cin$-rings is a fully faithful embedding. 
\end{lemma}
\begin{proof}
By ``Milnor's exercise", i.e. by a theorem of Pursell \cite[Section 8]{Pursell}, the global sections functor
\[
\Gamma: \op{\Man} \to \cring, \quad \Gamma(M\xrightarrow{\phi} N) = \cin(N)\xrightarrow{\phi^*} \cin(M)
\]
is fully faithful.  A map $\phi:M\to N$ between two Poisson manifolds is a Poisson map if and only if the pullback map $\phi^*: \cin(N)\to \cin(M)$ is a map of Poisson algebras.
\end{proof}

It will be convenient to have a slight generalization of a notion of a Poisson $\cin$-ring.
\begin{definition} \label{def:1} \label{def:bracket}
Let $\scA$ be a $\cin$-ring.  A {\sf bracket} is  a skew-symmetric $\cin$-ring biderivation  
$\lbr\cdot\,, \cdot\rbr:\scA \times \scA \to
\scA$.  
 That is, 
\begin{itemize} \item  $\lbr a,b\rbr = -\lbr b,a\rbr $ for all
  $a,b\in \scA$ and
\item    for all $n$,
for all $f\in \cin(\R^n)$ and all $a_1,\ldots, a_n, b\in \scA$
\[
\lbr f_\scA(a_1,\ldots, a_n), b\rbr  = \sum (\partial_i f)_\scA(a_1, \ldots,
a_n) \lbr a_i ,b\rbr.
\]
\end{itemize} 
\end{definition}
\begin{lemma} \label{lem:2.19}
A bracket $\lbr \cdot\,, \cdot\rbr :\scA \times \scA \to \scA$ on a $\cin$-ring $\scA$ (Definition~\ref{def:1}) corresponds to a
unique map of $\scA$-modules
\[
  B: \Lambda^2 \Omega^1_\scA \to \scA
\]
and conversely: for any $B\in \Hom (\Lambda^2 \Omega^1_\scA,\scA)$
there is a unique bracket $\lbr \cdot,\cdot\rbr :\scA \times \scA \to \scA$ with
\begin{equation}
\lbr a,b\rbr  = B(da\wedge db)
\end{equation}
for all $a, b\in \scA$.  Here as before $\Lambda^2 \Omega^1_\scA$ is the second exterior power of the $\scA$-module $\Omega^1_\scA$ of $\cin$-K\"ahler differentials (cf.\ Definition~\ref{def:2.270}).

Moreover, for any bracket $\lbr \cdot\,, \cdot\rbr :\scA \times \scA \to \scA$  there is an $\scA$-module map 
\[
\rho: \Omega^1_\scA \to \CDer(\scA), \quad \rho (\sum a_i db_i) = \sum a_i \lbr b_i, \cdot \rbr.
\] 
\end{lemma}  
\begin{proof}
Given a bracket  $\lbr \cdot\,, \cdot\rbr :\scA \times \scA \to \scA$  consider
\[
\ad :\scA \to \CDer(\scA), \qquad \ad(a) b := \lbr a, b\rbr .
\]
Since the bracket is a $\cin$-derivation in both slots, the universal
property of the derivation $d= d_\scA:\scA \to \Omega^1_\scA$ implies that there is a
unique $\scA$-linear map
\[
  \rho: \Omega^1_\scA \to \CDer(\scA) \quad (\simeq \Hom(\Omega^1_\scA, \scA))
\]
so that 
\[
\rho (da) = \ad(a)\quad  (= \lbr a, \cdot\rbr )
\]
for all $a\in \scA$. 
By following $\rho$ with the isomorphism
\begin{equation} \label{eq:2.19.2}
    \imath: \CDer(\scA)\to \Hom(\Omega^1_\scA, \scA), \qquad X\mapsto \imath(X)
\end{equation}
we get $\imath\circ \rho\in \Hom_\scA(\Omega^1_\scA, \Hom_\scA (\Omega^1_\scA, \scA))$.
Under the canonical
isomorphism
\[
\Hom(\Omega^1_\scA, \Hom(\Omega^1_\scA, \scA))
\stackrel{\sim}{\longleftrightarrow} \Hom (\Omega^1_\scA\otimes_\scA
\Omega^1_\scA, \scA).
\]
$\imath\circ \rho$ corresponds to an $\scA$-linear map $\widetilde{B}: \Omega^1_\scA\otimes_\scA
\Omega^1_\scA\to  \scA$ with
\begin{equation}
\widetilde{B} (\alpha \otimes \beta) = \imath (\rho(\alpha)) \beta
\end{equation}
for all $\alpha, \beta \in \Omega^1_\scA$.
In particular
\begin{equation}
\widetilde{B} (da \otimes db) = \imath (\rho(da))\, db = \rho(da)\, (b) = \lbr a,b\rbr
\end{equation}
for all $a,b\in \scA$.
 
Since the set $\{da\mid a\in \scA\}$ generates
the module $\Omega^1_\scA$ of K\"ahler differentials and since the
bracket is skew-symmetric, the map $\widetilde{B}$ is skew-symmetric.  That is, it
factors through
\[
  B:\Lambda^2\Omega^1_\scA\to \scA, \qquad B(\alpha \wedge \beta) = \imath(\rho(\alpha))\beta.
\]
with
\begin{equation} \label{eq:1}
B(da\wedge db) = \lbr a,b \rbr
\end{equation}
for all $a,b\in \scA$.
Conversely, any element $B\in \Hom (\Lambda^2\Omega^1_\scA, \scA)$
defines a (unique) bracket on $\scA$ which is given by \eqref{eq:1}.
\end{proof}

\subsection{Categories of Lie-Rinehart algebras} \label{subsec:L-R}
We recall the definition of a Lie-Rinehart algebra (a Lie-Rinehart pair) over a commutative ring $k$.  This definition is due to neither Lie nor Rinehart.  In the form that appears below (Definition~\ref{def:Lie-Rinehart}) it is due to Palais \cite{Palais}.  A similar notion was introduced earlier by Hertz \cite{Hertz}. It is equivalent to the one in Definition~\ref{def:Lie-Rinehart} if the module in question is faithfully generated 
in the sense of Definition~\ref{def:tor-free}. 

\begin{definition} \label{def:Lie-Rinehart}
Let $k$ be a commutative ring.  A {\sf Lie-Rinehart algebra} consists of a commutative $k$-algebra $A$ and an $A$-module $\scM$ so that $\scM$ is a $k$-Lie algebra and a map $\rho: \scM\to \Der_k(A)$ of $A$-modules (where $\Der_k(A)$ is the module of $k$-derivations of $A$) so that 
\begin{enumerate}
        
        \item $\rho$ is a map of $k$-Lie algebras: $\rho ([m_1, m_2]_\scM) = [\rho(m_1), \rho(m_2)]$ for all $m_1, m_2\in \scM$ and
        \item $
        [m_1, am_2]_\scM = (\rho(m_1) a) m_2 + a [m_1, m_2]_\scM$\\
        for all $m_1, m_2\in \scM$ and $a\in A$.
\end{enumerate}
Thus, a Lie-Rinehart algebra is a tuple $(k, (\scM, [\cdot\,, \cdot]_\scM),  \rho: \scM\to \Der_k(A))$. %
\end{definition}
There is a standard example of Lie-Rinehart algebras in differential geometry literature. It comes from a Lie algebroid:
\begin{example}\label{ex:2.21}
  $k= \R$, $E\to M$ a Lie algebroid.  Then $A= \cin(M)$ is the ring of smooth functions, which is an algebra over the commutative ring $\R$ of  real numbers, and $\scM= \Gamma(E)$ is the $\cin(M)$-module of sections of the algebroid.  The anchor map $\rho:E\to TM$ of the algebroid induces a map $\rho: \Gamma(E)\to \Gamma(TM) \simeq \Der (\cin(M))$ (really $\rho_*$, but the abuse of notation is standard), which is a map of real Lie algebras and of $\cin(M)$-modules.  The tuple 
 $(\R, (\Gamma(E), [\cdot\,, \cdot]_E), \rho: \Gamma(E)\to \Der(\cin(M))\simeq \Gamma(TM))$ is a Lie-Rinehart algebra in the sense of Definition~\ref{def:Lie-Rinehart}. 

Note that in this example the $\R$-algebra $\cin(M)$  of smooth functions on the manifold $M$ is a $\cin$-ring and that every $\R$-algebra derivation of $\cin(M)$ is a differential operator of first order, hence a $\cin$-ring derivation.  
\end{example}
Example~\ref{ex:2.21} above suggests the following definition of a Lie-Rinehart algebra over a $\cin$-ring.  We believe the definition is new.
\begin{definition} \label{def:cin-LR}
A {\sf Lie-Rinehart algebra} over a $\cin$-ring $\scA$  is a  $\scA$-module $\scM$ which is  a $\R$-Lie algebra,  and a map $\rho: \scM\to \CDer(\scA)$ of $\scA$-modules (where $\CDer(\scA)$ is the module of $\cin$-ring derivations of $\scA$, Notation~\ref{nota:cder}) so that 
\begin{enumerate}      
        \item $\rho$ is a map of real Lie algebras: $\rho ([m_1, m_2]_\scM) = [\rho(m_1), \rho(m_2)]$ for all $m_1, m_2\in \scM$ and
        \item $
        [m_1, am_2]_\scM = (\rho(m_1) a) m_2 + a [m_1, m_2]_\scM$\\
        for all $m_1, m_2\in \scM$ and $a\in A$.    
\end{enumerate}
Thus, a Lie-Rinehart algebra over $\cin$-rings is a tuple  $((\scM, [\cdot\,, \cdot]_\scM),  \rho: \scM\to \CDer(\scA))$.

If the bracket $\lb_\scM: \scM\times \scM\to \scM$ is skew-symmetric but fails to satisfy the Jacobi identity we refer to $((\scM, [\cdot\,, \cdot]_\scM), \scA, \rho: \scM\to \CDer(\scA))$ as an {\sf almost Lie-Rinehart algebra}.
\end{definition}

\begin{remark}
    Note that a Lie--Rinehart algebra over a \(C^\infty\)-ring \(\scA\) is also a Lie--Rinehart algebra in the usual sense over its underlying \(\mathbb{R}\)-algebra \(\scA\), since the anchor map takes values in \(\CDer(\scA)\subseteq \Der (\scA)\). However, the converse is not true, since an $\R$-algebra derivation is not necessarily a \(C^\infty\)-ring derivation. See Appendix~\ref{app:counter-ex}.
\end{remark}

\begin{example}
    Let \(\scA\) be a \(C^\infty\)-ring. Then \(((\scM = \cin\mathrm{Der}(\scA), [\cdot\,,\cdot]), \rho = \mathrm{id}\colon \cin\mathrm{Der}(\scA)\to \cin\mathrm{Der}(\scA))\) is a Lie--Rinehart algebra, where \([\cdot\,,\cdot]\) denotes the commutator bracket.
\end{example}
\begin{example} \label{ex:2.034}
If $E\xrightarrow{\rho} M$ is a Lie algebroid then 
\[    (\Gamma(E), [\cdot\,, \cdot]_E),  \Gamma(E)\xrightarrow{\rho}\CDer(\cin(M))
\]
is a Lie-Rinehart algebra over the $\cin$-ring $\cin(M)$.
\end{example}
\begin{remark} \label{rmrk:basic}
In Example \ref{ex:2.034}, the image of the anchor map $\rho\colon E\to TM$ at the level of sections, $\rho(\Gamma(E)) \subset \mathfrak{X}(M)$, is closed under taking the commutator of vector fields. Consequently, it (i.e., $\rho(\Gamma(E)) $ with all the appropriate structure maps) forms a Lie–Rinehart algebra over $C^\infty(M)$, whose anchor map is given by the inclusion. We will refer to 
\[
((\rho(\Gamma(E)), [\cdot\,, \cdot]_E), \rho(\Gamma(E))\hookrightarrow\CDer(\cin(M))
\]
as the {\sf basic} Lie-Rinehart algebra of the Lie algebroid $E\xrightarrow{\rho} TM$.
  \end{remark}

We next discuss morphisms of Lie-Rinehart algebras. 
There is more than one way to define them. 
The simplest definition, which should be adequate for our purposes, proceeds as follows.  
Given a map of $\cin$-rings $\varphi:\scA\to \scB$ there is, in general,  no map from  
the set of derivations $\CDer(\scA)$ to the set of derivations $\CDer(\scB)$ --- not on the level of sets nor of Lie algebras nor of $\scA$-modules  ($\CDer(\scB)$ is an $\scA$-module by way of $\varphi:\scA\to \scB$). Recall, however that a pair of derivations $(Y, X)\in \CDer(\scB)\times \CDer(\scA)$ is {\sf $\varphi$-related} if the diagram
\[
\xy
(-10,10)*+{\scA}="1";
(15,10)*+{\scB}="2";
 (-10,-8)*+{\scA}="3";
(15,-8)*+{\scB}="4"; 
{\ar@{->}_{ X} "1";"3"};
{\ar@{->}^{Y} "2";"4"};
{\ar@{->}^{\varphi} "1";"2"};
{\ar@{->}_{\varphi} "3";"4"};
\endxy
\]
commutes.  The simplest definition of a morphism of Lie-Rinehart algebras is then as follows:

\begin{definition} \label{def:mor_LR}
A {\sf morphism} from a Lie-Rinehart algebras $(\scM\xrightarrow{\rho}\CDer(\scA))$ to a Lie-Rinehart algebra $(\scN\xrightarrow{\rho'}\CDer(\scB))$
is a pair $(\varphi:\scA\to \scB, \Phi:\scM\to \scN)$ where $\varphi$ is a map of $\cin$-rings, $\Phi$ is a map of real Lie algebras and of $\scA$-modules so that for any $m\in \scM$ the derivations $\rho'(\Phi(m))$ and $\rho(m)$ are $\varphi$-related. 
Equivalently 
\begin{equation}
(\rho'\times \rho)\, (\grph(\Phi)) \subseteq \CDer(\varphi),
\end{equation}
where $\CDer(\varphi)$ is the collection of all pairs of related derivations:
\[
\CDer(\varphi):= \{ (Y, X)\in \CDer(\scB)\times \CDer(\scA) \mid Y\circ \varphi = \varphi \circ X\}.
\]
\end{definition}
\begin{remark}
Note that set  $\CDer(\varphi)$ in Definition~\ref{def:mor_LR} is a real Lie subalgebra of the product $\CDer(\scB)\times \CDer(\scA)$. It is also an $\scA$-submodule: For  all $(Y, X)\in \CDer(\varphi)$ and all $a\in \scA$
\[
(\varphi(a) Y, a X) \in \CDer(\varphi)
\]
as well.   And for all     $(Y_1, X_1), (Y_2, X_2)\in \CDer(\varphi)$ we have
\[
(Y_1+Y_2, X_1 + X_2) \in \CDer(\varphi).
\]  
\end{remark}
We will check shortly that  morphisms as defined above do compose and consequently that Lie-Rinehart algebras form a category.  But first, a remark on a limitation of Definition~\ref{def:mor_LR}.

\begin{remark}
The notion of morphism as given in Definition~\ref{def:mor_LR} is fairly restrictive.  For example, given a map $\varphi:\scA\to \scB$ of $\cin$-rings there is no morphism from the Lie-Rinehart algebra $\CDer(\scA)\xrightarrow{\id} \CDer(\scA)$ to the Lie-Rineart algebra $\CDer(\scB)\xrightarrow{\id} \CDer(\scB)$.   We can fix that by replacing a map $\Phi:\scM\to \scN$ in Definition~\ref{def:mor_LR} with a relation $R\subseteq \scN\times \scM$ subject to the condition that 
\[
(\rho'\times \rho)\, (R) \subseteq \CDer(\varphi).
\]  

If $E_1\xrightarrow{\rho_1} M_1$, $E_2\xrightarrow{\rho_2} M_2$ are two Lie algebroids, and $F:E_1\to E_2$ is a vector bundle map covering a smooth map $f:M_1\to M_2$, then in general the map $F$ does not induce a map on sections. 
One way to deal with this issue is to impose a condition on $F$: its graph should be a Lie subalgebroid of the direct product $E_2 \times E_1 \to M_2 \times M_1$ --- see for instance \cite[Definition~7.1]{Me}.  But such a condition does not give us a map between the corresponding Lie-Rinehart algebras $(\Gamma(E_1) \xrightarrow{\rho_1} \CDer(\cin(M_1))$ and $(\Gamma(E_2) \xrightarrow{\rho_2} \CDer(\cin(M_2))$. 

Both of these issues with Definition~\ref{def:mor_LR} can be fixed by replacing the module map $\Phi:\scM\to \scN$ by a relation $R\subset \scN\times \scM$.
This  will be discussed elsewhere. 
\end{remark}

\begin{lemma} \label{lem:composition_morphisms}
    If $(\varphi, \Phi): (\scM\xrightarrow{\rho}\CDer(\scA))  \to (\scN\xrightarrow{\rho'}\CDer(\scB))$ and $(\psi, \Psi): (\scN\xrightarrow{\rho'}\CDer(\scB)) \to (\scQ\xrightarrow{\rho''}\CDer(\scC))$ are two maps of Lie-Rinehart algebras then so is $(\psi\circ \varphi, \Psi\circ \Phi)$.
\end{lemma}
\begin{proof}
The lemma follows from the fact that if $Y\in \CDer(\scB)$ if $\varphi$-related to $X\in \CDer(\scA)$ and $Z\in \CDer(\scC)$ is $\psi$-related to $Y$ then $Z$ is $\psi\circ \varphi$-related to $X$.
\end{proof}
\begin{remark} \label{rmrk:CLR}
 A routine argument using Lemma~\ref{lem:composition_morphisms}   shows that Lie-Rinehart algebras and their morphisms  (as in Definition~\ref{def:mor_LR}) form a category $\clr$ of ($\cin$-)Lie-Rinehart algebras.
\end{remark}
\section{Free (``symmetric") algebras over $\cin$-rings} \label{sec:free1}
Recall that for a commutative ring $R$ and an $R$-module $\fM$ the symmetric algebra 
\[
S^\bullet(\fM) = R \oplus \fM \oplus S^2(\fM) \oplus \cdots  \equiv \bigoplus_{i=0}^\infty S^i (\fM)
\]
is a free $R$-algebra generated by $\fM$: given any map $\varphi : R \to R'$  of commutative rings (i.e., given any $R$-algebra $R'$) and an $R$-linear map $f : \fM \to R'$ (i.e., a map in the category of $R$-modules) there is a unique map $\psi : S^\bullet(\fM) \to R'$ of $R$-algebras with $\psi|_R = \varphi$ and $\psi|_\fM = f$.
In this section we discuss what the $\cin$-ring analogue of the symmetric algebra is, why it exists  and provide an explicit construction.  In the case where the $\cin$-ring is the ring $\cin(M)$ of smooth functions on a manifold $M$ and the module in question is the module $\Gamma(E)$ of sections of a vector bundle $E\to M$, the ``symmetric" algebra turns out to be the $\cin$-ring of {\em all} smooth functions $\cin(E^{\svee})$ on the total space of the dual bundle $E^{\svee}\to M$ (see Theorem~\ref{thm:4.01} below). The ordinary symmetric algebra $S^\bullet (\Gamma(E))$, on the other hand, is the algebra of smooth functions on $E^{\svee}$ that are {\em polynomial} along the fibers.   Thus $S^\bullet (\Gamma(E))$ is much smaller than $\cin(E^\svee)$.

We start by discussing $\scA$-algebras (where $\scA$ is a $\cin$-ring). To motivate the definition, which is known to experts, 
recall that given a commutative ring $R$, an $R$-algebra can be defined as a commutative ring $S$ together with a unital ring homomorphism $\varphi_S:R\to S$.   Thus, an $R$-algebra is an object in the undercategory (coslice category) $R/\cmring$, where $\cmring$ denotes the category of commutative rings. This point of view allows one to define the analogue of algebras in the setting of  $\cin$-rings (and, more generally, in the setting of Lawvere theories, but we will not need that).

\begin{definition} \label{def:A-alg}
Let $\scA$ be a $\cin$-ring.  An {\sf $\scA$-algebra}  is a $\cin$-ring $\scB$ together with a map $\varphi_\scB:\scA\to \scB$ of $\cin$-rings.  That is, an $\scA$-algebra is an object of the  undercategory (coslice category) $\scA/\cring$.
\end{definition}
\begin{remark}
$\scA$-algebras form a category: the coslice category $\scA/\cring$.
Recall that a
morphism in $\scA/\cring$ from $\scA\xrightarrow{\varphi_\scB} \scB$ to
$\scA\xrightarrow{\varphi_{\scB'}} \scB'$ is a map $f:\scB\to \scB'$ in $\cring$ so that
 the triangle $
\xy
(-0,10)*+{\scA  }="1";
(15,-8)*+{\scB'}="3";
 (-15,-8)*+{\scB}="2";
{\ar@{->}_{f} "2";"3"};
{\ar@{->}^{\varphi_{\scB'}} "1";"3"};
{\ar@{->}_{\varphi_\scB} "1";"2"};
\endxy
$ commutes (in $\cring$).
\end{remark}
We also have the category $\scA\textrm{-}\Mod$ of $\scA$-modules.  Recall that
these are ordinary modules over the $\R$-algebra underlying the
$\cin$-ring $\scA$ (Definition~\ref{def:module}).

There is a forgetful functor 
\[
U: \scA/\cring \to \scA\textrm{-}\Mod
\]
of $\scA$: it sends an $\scA$-algebra $\scA\xrightarrow{\varphi}\scB$ to $\scB$, which is 
viewed as an Abelian group and an $\scA$-module by way of $\varphi$: $a\cdot b:=
\varphi(a)b$ for all $a\in \scA$ and all $b\in \scB$.  The adjoint functor theorem guarantees that the forgetful functor $U$ has a left adjoint, which we now formally record.

\begin{theorem} \label{prop1.1} Let $\scA$ be a $\cin$-ring. The forgetful
  functor $U: \scA/\cring \to \scA\textrm{-}\Mod$ has a left adjoint
  $\cF: \scA\textrm{-}\Mod \to \scA/\cring$.
\end{theorem}  

\begin{remark}
The proof of this theorem by way of the adjoint functor theorem is not  constructive.  In particular if the ring $\scA$ is the $\cin$-ring of smooth functions on a manifold $M$ and $\scM$ is the $\cin(M)$-module of sections of some vector bundle $E\to M$, it is not at all clear what the $\cin(M)$-algebra  $\cF(\scM)$ is.
\end{remark}
\begin{remark}
Strictly speaking both the forgetful functor $U$ and its left adjoint $\cF$ should be decorated with the $\cin$-ring $\scA$:
\[\xy
  (-21,0)*+{\cF_\scA :\scA\textrm{-}\Mod}="1";
  (-9,1)*+{}="3";
  (-9,-1)*+{}="5";
  (21,0)*+{\scA/\cring:U_\scA}="2";
  (5,1)*+{}="4";
  (5,-1)*+{}="6";
  {\ar@{->}_{} "3";"4"}; {\ar@{->}_{} "6";"5"};
  \endxy
\]
    
\end{remark}
\begin{notation} \label{nota:phi_eta}
    We denote the unit of adjunction $\cF \dashv U$ of  Theorem~\ref{prop1.1} by $\eta$.
    We denote the map from the $\cin$-ring $\scA$ to the $\cin$-ring $\cF(\scM)$ that makes $\cF(\scM)$ into an $\scA$-algebra by $\varphi_\scM: \scA\to \cF(\scM)$.
    We refer to the maps $\eta_\scM: \scM \to U(\cF(\scM))$ and $\varphi_\scM:\scA\to \cF(\scM)$ as the {\sf structure maps}.
\end{notation}

\begin{definition} \label{def:sym-alg}
    Given a $\cin$-ring $\scA$ and an $\scA$-module $\scM$, we refer to the corresponding  $\scA$-algebra $\cF(\scM)$ either as {\sf free}  or as a {\sf ``symmetric"} algebra generated by the module $\scM$ (cf.\ Subsection~\ref{subsec:free}).
\end{definition}
Before we give a constructive proof of Theorem~\ref{prop1.1} we record a quick consequence which will be useful for us later.

\begin{lemma} \label{lem:3.06}
    For any $\cin$-ring $\scA$ and any $\scA$-module $\scM$ the structure maps $\varphi_\scM:\scA\to \cF(\scM)$ and $\eta_\scM:\scM \to U(\cF(\scM))$ are both injective.
\end{lemma}
\begin{proof}  
The identity map $\id:\scA\to \scA$ makes $\scA$ into an $\scA$ algebra. Consider the zero map $\psi: \scM\to U(\scA)$:
\[
    \psi(m) = 0
\]
for all $m\in \scM$.  We then have a map $\hpsi:\cF(\scM)\to \scA$ of $\scA$-algebras. Thus,
\[
\hpsi \circ \varphi_\scM = \id_\scA.
\]
Hence $\varphi_\scM$ is injective.  To show that $\eta_\scM$ is injective we consider the square zero extension $\scA\ltimes \scM$ (Remark~\ref{rmrk:m-bm-sz}).  There is a canonical map $\varphi: \scA \to \scA \ltimes \scM$  that makes the extension $\scA\ltimes \scM$ into an $\scA$-algebra.  It corresponds to the zero derivation: $0: \scM\to \scM$.  Explicitly
\[
\varphi(a) = (a, 0)
\]
for all $a\in \scA$.
The underlying $\scA$-module is $\scA \oplus \scM$:
\[
U(\scA\ltimes \scM) = \scA \oplus \scM. 
\]
We  
have an $\scA$-module map 
\[
\psi: \scM \to \scA\oplus \scM, \qquad \psi(m) := (0, m).
\]
The corresponding map $\hpsi: \cF(\scM) \to \scA\ltimes \scM$ has the property that 
\[
(U(\hpsi) \circ \eta_\scM)\, (m) = (0, m)
\]
for all $m\in \scM$.  Therefore, the map $\eta_\scM$ is injective.
\end{proof}

To obtain a constructive proof of Theorem~\ref{prop1.1} we use universal arrows.
Recall  the definition.
\begin{definition}\label{def:universal_arrow}
Given a functor $U:D\to C$, a {\sf universal arrow} from an object $c\in C$ to the functor $U$ is an object $F(c)$ of the category $D$ and a morphism $\eta_c:c\to U(F(c))$ with the property that for any object $d\in D$ and any morphism $\psi:c\to U(d)$ there exists a unique morphism $\widetilde{\psi}: F(c)\to d$ so that the diagram
\[
  \xy
  (-15,10)*+{c }="1";
(10,10)*+{U(F(c))}="2";
 (10,-8)*+{U(d)}="3";
 {\ar@{->}^<<<<<<<{\eta_c} "1";"2"};
 {\ar@{->}_{\psi} "1";"3"};
 {\ar@{->}^{U(\widetilde{\psi})} "2";"3"};
\endxy
\]
commutes in $C$.   
\end{definition}
Thus, if 
$\xy
  (-15,0)*+{F:C }="1";
  (-9,1)*+{}="3";
  (-9,-1)*+{}="5";
  (10,0)*+{D:U}="2";
  (5,1)*+{}="4";
  (5,-1)*+{}="6";
  {\ar@{->}_{} "3";"4"}; {\ar@{->}_{} "6";"5"};
  \endxy
$   
     
is a pair of adjoint functors and $\eta: \id\Rightarrow U\circ F$ is the unit of adjunction, then for any object $c\in C$ the pair $(\eta_c: c\to U(F(c)), F(c))$ is a universal arrow from $c$ to $U$.
Conversely, given a functor $U:D\to C$ so that for any object $c\in C$ there is a universal arrow  $(\eta_c: c\to U(F(c)), F(c))$ from $c$ to the functor $U$, then the map
\[
c \mapsto F(c)
\]
extends to a functor $F:C\to D$ which is left adjoint to $U$ and has $\eta = \{\eta_c\}_{c\in C}$ as the unit of adjunction.

In particular, it follows from Theorem~\ref{prop1.1} and the discussion of universal arrows that for any $\cin$-ring $\scA$ and any $\scA$ module $\scM$ we have an $\scA$-algebra $\varphi_\scM :\scA\to  \cF(\scM)$ and a map of $\scA$-modules 
\[
    \eta_\scM:\scM\to U(\cF (\scM))
\]
and a map $\scM\xrightarrow{\eta_\scM} \cF(\scM)$ of $\scA$-modules so that
for any $\scA$-algebra $\scA\xrightarrow{\varphi_\scB} \scB$ and any map
$\psi:\scM\to U(\scB)$ of $\scA$-modules there exists a unique map
$\hpsi:\cF(\scM))\to \scB$ of $\scA$-algebras making the diagram
\begin{equation}\label{eq:1.2.1}
  \xy
  (-15,10)*+{\scM  }="1";
(10,10)*+{U(\cF(\scM))}="2";
 (10,-8)*+{U(\scB)}="3";
 {\ar@{->}^<<<<<<<{\eta_\scM} "1";"2"};
 {\ar@{->}_{\psi} "1";"3"};
 {\ar@{->}^{U(\hpsi)} "2";"3"};
\endxy
\end{equation}
commute.  

Conversely, a collection of universal arrows $\{\eta_\scM:\scM\to U(\cF (\scM))\}_{\scM\in\scA\textrm{-}\Mod }$ gives rise to a functor $\cF:\scA\textrm{-}\Mod \to \scA/\cring $ left adjoint to the forgetful functor $U: \scA/\cring \to \scA\textrm{-}\Mod$.  Since any two functors left adjoint to $U$ are isomorphic, we may as well take the functor $\cF$ produced by the universal arrows to be ``the" left adjoint to $U$.

\begin{remark}[Pedantic] \label{rmrk:pedantic_arrow}
Suppose $U:D\to C$ is a functor, $f:c\to c'$ an isomorphism in the category $C$, $g:d \to d'$ an isomorphism in the category $D$ and $\eta_c:c\to U(d)$, $\eta_{c'}:c'\to U(d')$ are a pair of morphism in $C$ that make the diagram
\[
\xy
(-15,10)*+{c}="1";
(15,10)*+{U(d)}="2";
 (-15,-8)*+{c'}="3";
(15,-8)*+{U(d')}="4"; 
{\ar@{->}_{f} "1";"3"};
{\ar@{->}^{U(g)} "2";"4"};
{\ar@{->}^{\eta_{c}} "1";"2"};
{\ar@{->}_{\eta_{c'}} "3";"4"};
\endxy
\]   
commute.  If $\eta_c$ is a universal arrow from $c$ to $U$, then the map $\eta_{c'}$ is also  universal arrow (from $c'$ to $U$).
\end{remark}
We now start our proof of existence of universal arrows from $\scA$-modules to the forgetful functor $U:\scA/\cring \to \scA\textrm{-}\Mod$ and thereby a construction of the left adjoint functor $\cF$.

\begin{lemma} \label{lem:3.08}
  Let $\scA$ be a $\cin$-ring. For any {\em free} $\scA$-module $\scM$ there is a universal arrow to the functor   $U: \scA/\cring \to \scA\textrm{-}\Mod$.
\end{lemma}

Our proof of Lemma~\ref{lem:3.08} uses existence of the coproduct $\otimes_\infty$ of $\cin$-rings and existence of the left adjoint $\bbF: \Set \to \cring$  to the forgetful functor $\bbU: \cring \to \Set$. 
\begin{remark} \label{rmrk:3.12a}
We think of $\bbF(Y)$ as the  $\cin$-ring freely generated by the set $Y$.  
We denote the unit of adjunction $\bbF\dashv \bbU$ by $\delta$. Thus, for each set $Y$ we have a map 
\[
\delta= \delta_Y:Y\to \bbF(Y),
\]
where we suppressed the forgetful functor $\bbU$.  
We may want to suppress the map $\delta$ (which is injective) as well, and view elements of $Y$ as elements of (the set underlying the $\cin$-ring) $\bbF(Y)$.  Often the $\cin$-ring $\bbF(Y)$ is constructed as the ring $\cin (\R^Y)$ of all smooth functions on the vector space $\R^Y$ that depend on only finitely many variables. And then the meaning of ``smooth" is supposed to be clear (cf.\ \cite[Example~4.31]{Joy}).  In particular for any natural number $k$ the ring $\cin(\R^k)$ is freely generated by $Y=\{1,\ldots, k\}$.  The coordinate function $y_i: \R^k\to \R$ is the image of $i\in \{1,\ldots, k\}$ under the map $\delta: \{1,\ldots, k\}\to \cin(\R^k)$.  See Appendix~\ref{app:free} for more details.
\end{remark}
\begin{remark} \label{rmrk:3.12b}
Since the category $\cring$ of $\cin$-rings is cocomplete, the coproduct $\otimes_\infty$ exists.  The issue is computing it.  We will use the following facts:\\[-10pt]
\begin{enumerate}
    \item \label{item:3.13.1} If $\scA = \bbF(X)$ and $\scB = \bbF(Y)$, i.e., if $\scA$ is freely generated (as $\cin$-ring) by a set $X$ and $\scB$ is freely generated by a set $Y$,  then $\scA \otimes _\infty \scB = \bbF(X\sqcup Y)$ and the coproduct structure maps
    \[
   \scA \xrightarrow{ \imath_\scA} \scA \otimes_\infty \scB \xleftarrow{ \imath_\scB}
    \scB 
    \]
are induced by $X \hookleftarrow X\sqcup Y \hookrightarrow Y$ respectively.   This is because the free functor $\bbF$,  being left adjoint, preserves colimits. \\ 
\item If $\scA = \cin(M)$ and $\scB = \cin(N)$ where $M$, $N$ are manifolds, then $\scA \otimes_\infty \scB = \cin(M\times N)$ and the structure maps
\[
    \imath_\scA: \cin(M) \to \cin(M) \otimes_\infty \cin(N), \qquad
    \imath_\scB: \cin(N) \to \cin(M) \otimes_\infty \cin(N) 
\]
are induced by the canonical projections $p_M:M\times N\to M$, $p_N:M\times N\to N$, respectively.  See \cite{MR}.

\item \label{item:3.013.3} If $\Pi:\bbF(Y)\to \scA$ is surjective and $X$ is a set, then
\[
\scA \otimes_\infty \bbF(X) = (\bbF(Y)\otimes_\infty \bbF(X)/ \langle \imath_{\bbF(Y)} (\ker \Pi ) \rangle
\]
 and the map 
\[
\Pi\otimes_\infty \id: \bbF(Y)\otimes_\infty \bbF(X) \to 
\scA \otimes_\infty \bbF(X) 
=(\bbF(Y)\otimes_\infty \bbF(X)/ \langle \imath_{\bbF(Y)} (\ker \Pi ) \rangle
\]
is the quotient map.  Hence 
\[
\ker (\Pi\otimes_\infty \id) = \langle \imath_{\bbF(Y)} (\ker \Pi ) \rangle.
\]

\end{enumerate}
\end{remark}
\begin{proof}[Proof of Lemma~\ref{lem:3.08}]
    Denote the basis of the free module $\scM$ by $\{e_y\}_{y\in Y} $ (where $Y$ is an appropriate indexing set).  Let 
\[
\cF(\scM):= \scA \otimes _\infty \bbF(Y),
\] 
the coproduct of the $\cin$-ring $\scA$ and the free $\cin$-ring $\bbF(Y)$ generated by the set $Y$.  The coproduct comes with the structure maps
\begin{equation} \label{eq:3.8.01}
   \scA \xrightarrow{ \imath_\scA} \scA \otimes_\infty \bbF(Y) \xleftarrow{ \imath_{\bbF(Y)}}
    \bbF(Y).
\end{equation}
The map $\imath_\scA:\scA \to \cF(\scM)$ makes $\cF(\scM)$ into a $\scA$-algebra.  Abusing notation (i.e., we suppress the forgetful functor to $\Set$) we get a map of sets
\[
\imath_{\bbF(Y)} \circ \delta:Y \to \scA \otimes _\infty \bbF(Y) \equiv \cF(\scM).
\]
Since $\{e_y\}_{y\in Y}$ is a basis of the module $\scM$ there is a unique map 
\[
\eta\equiv \eta_\scM:\scM \to U(\cF(\scM))
\]
of $\scA$-modules with 
\[
\eta_\scM (e_y) = (\imath_{\bbF(Y)} \circ \delta)(y).
\]
Explicitly, since elements of $\scM$ are sums $\sum_{y\in Y} a_y e_y$ with all but finitely many coefficients $a_y$ are zero,  the map $\eta_\scM$ is 
\begin{equation}\label{eq:eta_free}
\eta_\scM\left(\sum_{y\in Y} a_y e_y \right) = \sum_{y\in Y} a_y \cdot y
\end{equation}
where on the left we severely abused notation and wrote $a_y$ for $\imath_\scA (a_y)$ and $y$ for $(\imath_{\bbF(Y)} \circ \delta)(y)$, $\cdot$ for  the multiplication in $\bbF(\scM)$ and $\sum$ for  the addition.

We next check that $\eta_\scM$ is a universal arrow. Suppose $\varphi_\scB:\scA\to \scB$ is an $\scA$-algebra and $\psi: \scM\to U(\scB)$ is a map of $\scA$-modules.  We argue that there is a unique map 
$\hpsi: \cF(\scM) \to \scB$ of $\scA$-algebras so that the diagram
\begin{equation} \label{eq:3.8.02}
  \xy
  (-15,10)*+{\scM }="1";
(10,10)*+{U(\cF(\scM))}="2";
 (10,-8)*+{U(\scB)}="3";
 {\ar@{->}^<<<<<<<{\eta_\scM} "1";"2"};
 {\ar@{->}_{\psi} "1";"3"};
 {\ar@{->}^{U(\widehat{\psi})} "2";"3"};
\endxy
\end{equation}
commutes in $\scA$-$\Mod$. This is straightforward.  Here are the details.  Since $\bbF(Y)$ is freely generated by the set $Y$ there is a unique map $\overline{\psi}:\bbF(Y)\to \scB$ of $\cin$-rings with 
\[
\overline{\psi}(\delta(y))  = \psi (e_y)
\]
for all $y\in Y$. The universal property of the coproduct \eqref{eq:3.8.01} gives us a unique map of $\cin$-rings $\hpsi: \scA \otimes_\infty \bbF(Y) \to \scB$ so that the diagram
\begin{equation}
  \xy
  (-15,10)*+{\scA }="1";
(10,10)*+{\scA\otimes_\infty \bbF(Y)}="2";
 (10,-8)*+{\scB}="3";
 (40,10)*+{\bbF(Y)}="22";
 {\ar@{->}^<<<<<<<{\imath_\scA} "1";"2"};
 {\ar@{->}_<<<<<<<{\imath_{\bbF(Y)}} "22";"2"};
 {\ar@{->}_{\varphi_\scB} "1";"3"};
 {\ar@{-->}^{\hpsi} "2";"3"};
  {\ar@{->}^{\overline{\psi}} "22";"3"};
\endxy
\end{equation}
commutes.  Thus $\hpsi$ is a map of $\scA$-algebras.  Moreover,
\[
U (\hpsi)(\eta_\scM (e_y) ) = U (\hpsi)(\imath_{\bbF(Y)}(\delta(y))) = U(\overline{\psi}) (\delta(y)) = \psi (e_y)
\]
for any element $y\in Y$, so \eqref{eq:3.8.02} commutes.  Finally, suppose that $\widetilde{\psi}: \cF(\scM) \to \scB$ is another map of $\scA$-algebras so that 
\[
U(\widetilde{\psi}) \circ \eta_\scM = \psi.
\]
We argue that $\hpsi = \widetilde{\psi}$ and, thus, that the map $\hpsi$ is unique.   For any $y\in Y$
\[
\psi(e_y) = U(\widetilde{\psi}) (\eta_\scM (e_y)) =  U(\widetilde{\psi}) (\imath_{\bbF(Y)}(\delta(y)) )= U (\widetilde{\psi} \circ \imath_{\bbF(Y)}) (\delta(y)).
\]
On the other hand, for any $y\in Y$,
\[
\overline{\psi} (\delta(y)) = \psi(e_y)
\]
as well.   Since the $\cin$-ring $\bbF(Y)$ is freely generated by the set $Y$,
\[
\widetilde{\psi} \circ \imath_{\bbF(Y)} = \overline{\psi}.
\]
Since $\overline{\psi} = \hpsi \circ \imath_{\bbF(Y)}$, it follows that
\[
\hpsi \circ \imath_{\bbF(Y)} = \widetilde{\psi} \circ \imath_{\bbF(Y)}.
\]
On the other hand both $\hpsi$ and $\widetilde{\psi}$ are both maps of $\scA$-algebras, so 
\[
\widetilde{\psi} \circ \imath_\scA = \varphi_\scB = \hpsi \circ \imath_\scA.
\]
The universal property of the coproduct \eqref{eq:3.8.01} implies that $\widetilde{\psi} = \hpsi$.
Thus, the map $\hpsi:\cF(\scM)\to \scB$ is unique and $\eta_\scM :\scM \to U (\scA \otimes _\infty \bbF (Y))$ is a universal arrow.
\end{proof}

\begin{corollary} \label{cor:3pi}
   In the case when the $\cin$-ring $\scA$ is the reals $\R$ (cf.\ Example~\ref{ex:smooth-functions}), the free (``symmetric") $\R$-algebra $\cF(\scM)$ is $\cin(\scM^\svee)$, the $\cin$-ring of smooth functions on the dual module.
\end{corollary}
    
\begin{proof}
Choose a basis $Y\subset \scM$ of the $\R$-module (real vector space) $\scM$.  Then 
\[
\scM \simeq \R^{\oplus Y}.
\] 
Hence the (algebraic) dual module $\scM^\svee$ is 
\[
\scM^\svee \simeq \Hom_\R(\R^{\oplus Y}, \R)  = \prod_{y\in Y}\R=  \R^Y,
\]
the product of $Y$-copies of $\R$. On the other hand by (the proof of) Lemma~\ref{lem:3.08}
\[
\cF(\scM) = \cF(\R^{\oplus Y}) = \R\otimes_\infty \bbF(Y).
\]
Since $\R$ is initial in the category $\cring$ of $\cin$-rings, $\R\otimes_\infty \bbF(Y) = \bbF(Y)$.  The free $\cin$-ring $\bbF(Y)$ may be taken to be $\cin(\R^Y)$, the $\cin$-ring of smooth functions on the vector space $\R^Y$ --- see Appendix~\ref{app:free}.  Therefore,
\[
\cF(\scM) \simeq \cin(\R^Y) = \cin(\scM^\svee),
\]
where the equality may be taken to be the {\em definition} of the $\cin$-ring  $\cin(\scM^\svee)$.
\end{proof}

\begin{lemma} \label{lem:3.09} As before let $\scA$ be a $\cin$-ring. Suppose $\eta_\scM: \scM\to U (\scC)$ is a universal arrow from an $\scA$-module $\scM$ to the functor $U: \scA/\cring \to \scA\textrm{-}\Mod$.
Suppose further that $\tau:\scM\to \scN$ is a surjective map of $\scA$-modules with kernel $K$ and $\langle \eta_\scM (K)\rangle $ is the ideal in the $\scA$-algebra $\scC$ generated by $\eta_\scM(K)$. Then the map
\begin{equation}
    \overline{\eta}:\scN\to U(\scC/\langle \eta_\scM (K)\rangle)
\end{equation}
induced by $\eta_\scM$ is a universal arrow from the module $\scN$ to the functor $U$.
\end{lemma}

\begin{proof}
Let $\varphi_\scB:\scA\to \scB$ be an $\scA$-algebra and $\psi: \scN \to U(\scB)$ a map of $\scA$-modules.   We argue that there is a unique map $\hpsi: \scC/\langle \eta_\scM (K)\rangle \to \scB$ so that 
\[
\psi = U(\,\hpsi\,) \circ \overline{\eta}.
\]
Since $\eta_\scM$ is a universal arrow, there is a unique map $(\psi \circ \tau)^\sim: \scC \to \scB$ so that the diagram
\[
\xy
(-10,10)*+{\scM }="1";
(15,10)*+{U(\scC)}="2";
 (-10,-8)*+{\scN}="3";
 (30,-20)*+{U(\scB)}="5";
{\ar@{->}^{U((\psi \circ \tau)^\sim)} "2";"5"};
{\ar@{->}_{\tau} "1";"3"};
{\ar@{->}^{\eta_\scM} "1";"2"};
{\ar@{->}_{\psi} "3";"5"};
\endxy
\] commutes. Since $\psi (\tau(K))=0$, $\eta_\scM(K)$ lies in the kernel of $U((\psi\circ\tau)^\sim)$.  Therefore, the ideal $\langle \eta_\scM (K)\rangle$ generated by $\eta_\scM(K)$ lies in the kernel of $(\psi \circ \tau)^\sim \colon \scC\to \scB$. Let 
\[
\pi: \scC\to \scC/\langle \eta_\scM (K)\rangle
\]
denote the quotient map. Then $\eta_\scM$ induces 
\[
\overline{\eta}: \scN \to U(\scC/\langle \eta_\scM (K)\rangle ) \simeq U(\scC)/\langle \eta_\scM (K)\rangle
\]
making the diagram 
\begin{equation} \label{eq:3.9.2}
\xy
(-30,10)*+{\scM }="1";
(15,10)*+{U(\scC)}="2";
 (-30,-8)*+{\scN}="3";
(15,-8)*+{U(\scC)/\langle \eta_\scM (K)\rangle )}="4"; 
{\ar@{->}_{\tau} "1";"3"};
{\ar@{->}_{\pi} "2";"4"};
{\ar@{->}^{\eta_\scM} "1";"2"};
{\ar@{->}_<<<<<<<<<<<<<<<<<{\overline{\eta}} "3";"4"};
\endxy
\end{equation}
commute.  
Hence, there is a unique map $\widetilde{\psi}:\scC/\langle \eta_\scM (K)\rangle \to \scB$ so that 
\begin{equation} \label{eq:3.9.3}
    \widetilde{\psi} \circ \pi = (\psi \circ \tau)^\sim.
\end{equation}
We next argue that 
\begin{equation} \label{eq:3.9.4}
    U(\widetilde{\psi}) \circ\overline{\eta} = \psi.
\end{equation}
This is because 
\[
\psi \circ \tau = U((\psi \circ \tau)^\sim)\circ \eta_\scM \stackrel{\quad \eqref{eq:3.9.3}\quad }{=} U(\widetilde{\psi})\circ \pi \circ \eta_\scM \stackrel{\eqref{eq:3.9.2}\textrm{ commutes}}{=} U(\widetilde{\psi})\circ \overline{\eta} \circ \tau.
\]
Since $\tau$ is surjective, \eqref{eq:3.9.4} holds.  

It remains to prove that if $\overline{\psi}: \scC/\langle \eta_\scM(K)\rangle \to \scB$ is another map of $\scA$-algebras with 
\begin{equation}
    U(\overline{\psi}) \circ \overline{\eta} = \psi,
\end{equation}
then
\begin{equation} \label{eq:3.9.6}
    \overline{\psi} = \widetilde{\psi}.
\end{equation}
On the one hand 
\[
\psi \circ \tau = U(\overline{\psi})\circ \overline{\eta} \circ \tau = U(\overline{\psi})\circ  U(\pi) \circ \eta_\scM = U(\overline{\psi} \circ \pi)\circ \eta_\scM.
\]
On the other hand
\[
\psi \circ \tau = U((\psi \circ \tau)^\sim) \circ \eta_\scM
\]
as well.  Since $\eta_\scM$ is a universal arrow,
\[
\overline{\psi} \circ \pi = (\psi \circ \tau)^\sim.
\]
Since $(\psi \circ \tau)^\sim = \widetilde{\psi} \circ \pi$,
\[
\overline{\psi} \circ \pi = \widetilde{\psi} \circ \pi.
\]
Since the map $\pi$ is surjective, $\overline{\psi}= \widetilde{\psi}$ and we are done.
\end{proof}

\begin{corollary}\label{cor:3.150}
    For any $\cin$-ring $\scA$ and for any $\scA$-module $\scN$ there is a universal arrow from $\scN$ to the functor $U: \scA/\cring \to \scA\textrm{-}\Mod$.      
\end{corollary}

\begin{proof}
    For any $\scA$-module $\scN$ there is a free module $\scM$ and a surjective map $\tau: \scM\to \scN$.  Choose a basis $\{e_y\}_{y\in Y}$ of $\scM$ (which exists since $\scM$ is a free module). By Lemma~\ref{lem:3.08} there is a universal arrow 
\[   
\eta_\scM \to  \scA \otimes_\infty \bbF(Y), \quad {\eta}_\scM \, (\sum _{y\in Y} a_y f_y) = \sum_{y\in Y} a_y \cdot y,
\] 
c.f. \eqref{eq:eta_free}.
Here as before $\bbF(Y)$ is the $\cin$-ring freely generated by the set $Y$. Let $K := \ker(\tau: \scM\to \scN)$.
By Lemma~\ref{lem:3.09} the map
    \[
    \overline{\eta}:\scN \simeq \scM/K \to U( (\scA \otimes_\infty \bbF(Y))/\langle \eta_\scM(K)\rangle),
    \]
induced by $\eta_\scM$ is a universal arrow from the module $\scN$ to the functor $U$.
\end{proof}
\begin{remark}
We note that there is an explicit formula for the map $\overline{\eta}:\scN\to U( (\scA \otimes_\infty \bbF(Y))/\langle \eta_\scM(K)\rangle)$ constructed in the course of the proof of Corollary~\ref{cor:3.150}.  Namely, the set $\{f_y:= \tau(e_y)\}_{y\in Y}$ generates the module $\scN$.  Therefore, any $n\in \scN$ can be written as $n = \sum_{y\in Y} a_y f_y$ for some (not necessarily unique) choice of coefficients $a_y\in \scA$ with all but finitely many $a_y$'s being 0.  And then
\begin{equation} \label{eq:eta_formula}
\overline{\eta} \, (\sum _{y\in Y} a_y f_y) = \sum_{y\in Y} a_y \cdot y + \langle \eta(K)\rangle.
\end{equation}
\end{remark}
The proof of Corollary~\ref{cor:3.150} concludes our construction of the functor $\cF: \scA\textrm{-}\Mod \to \scA/\cring$ left adjoint to the forgetful functor $U: \scA/\cring \to \scA\textrm{-}\Mod$.   \\

We will see in the next section (Section~\ref{sec:free2}) that in the case where the $\cin$-ring $\scA$ is the $\cin$-ring $\cin(M)$ of smooth functions on a manifold $M$ and the $\cin(M)$-module $\scM$ is the module of sections $\Gamma(E)$ of a vector bundle $E\to M$, the free algebra $\cF(\Gamma(E))$ is the algebra $\cin(E^\svee)$ of smooth functions on the vector bundle $E^\svee \to M$ dual to $E\to M$. See Theorem~\ref{thm:4.01}. So in this particular case the $\cin$-ring $\cF(\scM)$ is easy to compute. For instance, we have: 

\begin{example}[Tangent and cotangent bundles of manifolds]\label{ex:3.160}
    Fix a manifold $M$.  
      Then the module $\scM=\CDer(\cin(M))$ of derivations of $\scA=\cin(M)$ is isomorphic to the module of sections $\Gamma(TM)$ of its tangent bundle $TM\to M$.  Hence by Theorem~\ref{thm:4.01}, the $\cin$-ring $\cF(\CDer(\cin(M)))$ is (isomorphic to) the ring of smooth functions $\cin(T^\svee M)$ on the cotangent bundle $T^\svee M$ of $M$.

    The module $\Omega^1(M)$ of de Rham 1-forms on the manifold $M$ is the module of sections of the cotangent bundle $T^\svee M$.  Hence, the $\cin$-ring $\cF(\Omega^1(M))$ is (isomorphic to) the ring of smooth functions on the tangent bundle of $M$.
\end{example}  
\begin{remark} \label{rmrk:cotangent}
Motivated by Example~\ref{ex:3.160},  
 given a $\cin$-ring $\scA$ we view the $\scA$-algebra 
$\cF(\CDer(\scA))$ as the ``cotangent bundle" of the $\cin$-ring $\scA$.  Since the identity map $\id: \CDer(\scA)\to \CDer(\scA)$ makes the $\scA$-module $\CDer(\scA)$ into a Lie-Rinehart algebra, Theorem~\ref{thm:main} below implies that the $\cin$-ring $\cF(\CDer(\scA))$ comes equipped with a natural Poisson bracket.  If $\scA = \cin(M)$, the ring of smooth functions on a manifold $M$, then $\cF(\CDer(\cin(M))$ is $\cin(T^\svee M)$ with its natural Poisson bracket.
\end{remark}

\begin{remark}
Proposition~\ref{prop1.1} also allows us to define  the $\cin$-ring analogue of
``tangent bundles'' of commutative rings.  Recall that for a
commutative ring $R$ its tangent bundle is the symmetric power
$S^\bullet (\Omega^1_R)$, where $\Omega^1_R$ denotes the $R$-module of
(commutative algebraic) K\"ahler differentials.    Consequently for a
$\cin$-ring $\scA$ the $\cin$-ring $\cF(\Omega^1_\scA)$ should be its
tangent bundle.  Here $\Omega^1_\scA$ is the Dubuc-Kock module of
$\cin$-algebraic K\"ahler differentials (Definition~\ref{def:Kahler} and Theorem~\ref{thm:DK}).  In fact more is true: there is a  tangent category structure on the category $\op{\cring}$, the category opposite to the category $\cring$ of $\cin$-rings; see \cite{Ivan}.   Tangent categories were  introduced by Rosicky \cite{Ros} in 1984 and later generalized by Cockett and Cruttwell in 2014 \cite{CC}.
\end{remark}

Computing the free algebra $\cF(\scM)$ is not  hard.  We give several examples.
\begin{example}\label{ex:single_vf}
    Let \( X = f\,\partial_x \) be a vector field on \( \mathbb{R} \), where \( f \in C^\infty(\mathbb{R}) \) satisfies \( f(x) = 0 \) for all \( x \le 0 \). Let \( \mathscr{M} = \mathrm{Span}_{C^\infty(\mathbb{R})}(X) \subset \mathfrak X(\mathbb R)\) be the module over the \( C^\infty \)-ring \(\scA= C^\infty(\mathbb{R}) \) generated by \( X \). Consider the \(\scA\)-linear surjection
\[
\tau : \scA \longrightarrow \scM, \,1\to X.
\]We have \[\ker\tau =\{a\in \scA\mid af=0\}=:\mathrm{Ann}(f).\] Thus, \[\scM  \stackrel{\eta_\scM}{\longrightarrow}\cF(\scM)\simeq C^\infty(\mathbb R^2)/\langle a(x)\xi, a\in \mathrm{Ann}(f) \rangle\stackrel{\varphi_\scM\;\;\;}{\longleftarrow \scA},\]where $x$ and  $\xi$ denote the standard coordinates on $\mathbb{R}^2$. Here 
\[
\eta_\scM(X)=\xi+ \langle a(x)\xi, a\in \mathrm{Ann}(f) \rangle
\]
and $\varphi_\scM$  is the quotient map induced by the pullback of the projection on the first factor $\pr_1:\mathbb R^2\to \mathbb R$, $\pr_1(x,\xi) = x$. If the function $f$ vanishes exactly on $(-\infty,0]$, for example if
\[
f(x)=
\begin{cases}
e^{-1/x^2}, & x>0,\\
0, & x\le 0,
\end{cases}
\] 
then 
\[
\cF(\scM)\simeq C^\infty(\mathbb R^2)/\langle a(x)\xi, a|_{(0, \infty)}\equiv 0 \rangle.
\]
\end{example}

\begin{example}\label{ex:order2_vanishing}We consider the following example of vector fields on $\mathbb{R}^2$ vanishing to order two. Let  $\scM\subset \mathfrak X(\mathbb R^2)$ denote the module over $\scA=C^\infty(\mathbb R^2)$ generated by them. Explicitly, $\scM$ is globally generated by the vector fields \[x^2\partial_x,\, y^2\partial_x,\, xy\partial_x,\, x^2\partial_y,\, y^2\partial_y,\, xy\partial_y.\]
Define the \(\scA\)-linear surjection
\[
\tau :  \scA^{\oplus 6} \longrightarrow \scM
\]
by sending the standard basis \(e_1, \dots, e_6\) to
\[
\tau(e_1) = x^2 \partial_x, \quad
\tau(e_2) = xy \partial_x, \quad
\tau(e_3) = y^2 \partial_x,
\]
\[
\tau(e_4) = x^2 \partial_y, \quad
\tau(e_5) = xy \partial_y, \quad
\tau(e_6) = y^2 \partial_y.
\]
The  kernel of $\tau$ is given by
\[
\ker(\tau) =\operatorname{Span}_\scA \{ x e_2-y e_1,\; x e_3-y e_2,\; x e_5-y e_4 ,\; x e_6-y e_5 \}.
\]    We have
\[
\cF(\scM) \simeq \cF(\scA^{\oplus 6})\big/ \langle \eta_{\scA^{\oplus 6}}(\ker \tau) \rangle,
\]
where
\[
\cF(\scA^{\oplus 6}) \simeq \scA \otimes_{\infty} C^\infty(\mathbb{R}^6),
\]
and, since ${\scA^{\oplus 6}}$ is a free module,
\[
\eta_{\scA^{\oplus 6}}\left(\sum_{i=1}^6 a_i e_i\right) = \sum_{i=1}^6 a_i \,\xi_i,\quad\text{and}\quad \varphi_{\scA^{\oplus 6}}=\iota_{\scA^{\oplus 6}}
\]
with $\xi_1,\dots,\xi_6$ denoting the standard coordinates on $\mathbb{R}^6$. More precisely, \[\cF(\scM)\simeq C^\infty(\mathbb R^2\times \mathbb R^6)/\langle\, x \xi_2-y \xi_1,\; x \xi_3-y \xi_2,\; x \xi_5-y \xi_4 ,\; x \xi_6-y \xi_5 \,\rangle.\]    
\end{example}

In the next  two examples  \( C^\infty \)-rings are not algebras of functions on smooth manifolds and the modules  are neither free nor projective. 

\begin{example}[The cotangent bundle of dual numbers, cf.\ Remark~\ref{rmrk:cotangent}]\label{ex:dual_numbers}
Consider the $\cin$-ring $\scA := \cin(\R)/(x^2)$. Hadamard's lemma implies that
\[
\scA = \{a + b \varepsilon \mid \varepsilon^2 = 0, a,b\in \R\}.
\]      
Since $x\frac{d}{dx} (x^2) = 2x^2\in \langle x^2\rangle$, the vector field $x\frac{d}{dx}$ induces a derivation $v:\scA\to \scA$.  Note that 
\[
v (a + b x +\langle x^2\rangle) = bx +\langle x^2\rangle.
\]
That is, $v(\varepsilon) = \varepsilon$    and $v(1) = 0$ (as one should expect).  We thus write
\[
v = \varepsilon \frac{d}{d\varepsilon}.
\]
We next argue that any other derivation $w\in \CDer(\scA)$ is of the form $w = c v$ for some $c\in \R$.
Indeed, $w(1) = 0$ and $w(\varepsilon) = \alpha + \beta \varepsilon$ for some $\alpha, \beta\in \R$.  But 
\[
0 = w(0) = w(\varepsilon^2) = 2 \varepsilon w(\varepsilon) = 2\varepsilon \cdot (\alpha + \beta \varepsilon).
\]
Hence $\alpha = 0$. 
Since 
\[
((a+ b \varepsilon) v )\,(\varepsilon) = (a+ b\varepsilon) \varepsilon = a \varepsilon = (av) (\varepsilon),
\]
\[
(a+ b \varepsilon) v  = a v.
\]
It follows that $\CDer(\scA)$ is {\em not} a free module over $\scA$. Since $\scA$ is a local ring, a theorem of Kaplansky implies that the module $\CDer(\scA)$  is not projective.  In fact, it is not even torsion free. 
We have a short exact sequence of $\scA$-modules
\[
0\to \R\varepsilon \to \scA \to \CDer(\scA) \to 0.
\]
Since $\scA$ is a free rank 1 $\scA$-module, $\cF(\scA) = \scA \otimes_\infty\cin(\R^1)$ and $\eta:\scA \to \cF(\scA)$ is 
\[
\eta(a) = \imath_\scA (a) \imath_{\cin(\R)} (y)
\]
where $y:\R\to \R$ is the coordinate function.   It follows that $\cF(\CDer(\scA))$ is the quotient of $\cF(\scA)$ by the ideal generated by $\eta(\R \varepsilon)$ (see proof of Corollary~\ref{cor:3.150}). If we abuse notation and identify  
\[
\scA \otimes_\infty\cin(\R^1) = \left(\cin(\R)/\langle x^2 \rangle\right)\otimes_\infty\cin(\R^1) \simeq \cin(\R^2)/\langle x^2 \rangle,
\]
then 
\[
\eta (a+ b \varepsilon) = \eta( a + bx +(x^2)) = (a + b x)y +\langle x^2\rangle.
\]
In particular 
\[
\eta (\varepsilon) = xy +\langle x^2 \rangle.
\]
It follows that 
\[
\cF (\CDer(\scA) )\simeq \cin(\R^2))/\langle x^2, xy \rangle
\]
and that $\eta_{\CDer(\scA)}: \CDer(\scA) \to \cF(\CDer(\scA))$ is determined by 
\[
\eta_{\CDer(\scA)} (v) = y +\langle x^2, xy\rangle.
\]
\end{example}

\begin{example}\label{ex:the_cross}
    Consider the ideal $\langle xy \rangle $ in the $\cin$-ring $\cin(\R^2)$ generated by the product of the two standard coordinate functions $x$ and $y$.   Let $\scA = \cin(\R^2)/\langle xy\rangle$ be the quotient $\cin$-ring. Let  $\scM=\CDer(\scA)$. A direct computation using Hadamard’s theorem twice shows that $\scM$ is generated by the classes of the elements $x\partial_x, y\partial_y$. The kernel of the $\scA$-linear map $\scA^2\to \CDer(\scA)$ that sends the standard basis elements $e_1$ and $e_2$ to $x\partial_x$ and $y\partial_y$, respectively, is generated by $(y,x)$.  The corresponding $\cin$-ring $\cF(\CDer(\scA))$ is 
\[
C^{\infty}(\mathbb R^2\times \mathbb R^2)/\langle xy,  y\xi, x\nu \rangle\simeq  \scA/\langle y+ \langle xy \rangle\rangle\oplus \scA/\langle x+ \langle xy \rangle \rangle
\]
    $x,y, \xi, \nu$ are the standard coordinates for $\mathbb R^4=\mathbb R^2\times \mathbb R^2$. The structure maps are given as follows:\begin{align*}
        \eta_\scM(x\partial_x)=\xi\quad&\text{and}\quad \eta_\scM(y\partial_y)=\nu\\\\\varphi_\scM(x + \langle xy\rangle)= x +\langle xy,  y\xi, x\nu \rangle \quad&\text{and}\quad  \varphi_\scM(y + \langle xy\rangle)= y +\langle xy,  y\xi, x\nu \rangle.
    \end{align*}
\end{example} 

\section{The ``symmetric algebra" functor and projective modules over manifolds} \label{sec:free2}

The goal of this section is to prove that when the $\cin$-ring $\scA$ is the ring of smooth functions on a manifold $M$ and $\scM= \Gamma(E)$ is the module of sections of a vector bundle $E\to M$ then the ``symmetric algebra" $\cF(\scM)$ is the $\cin$-ring $\cin(E^{\svee})$ of smooth functions on the total space of the dual bundle $E^{\svee}\xrightarrow{p} M$.  The $\cin(M)$-algebra structure on $\cin(E^{\svee})$ is given by the pullback map
\[
p^*: \cin(M)\to \cin(E^{\svee}).
\]
More precisely we  prove
\begin{theorem}\label{thm:4.01}
 Let $p:E\to M$ be a vector bundle over a manifold $M$, $p^{\svee}:E^{\svee}\to M$ the dual bundle.  Then the map
\begin{equation}    \label{eq:4.2.01}
\eta_E: \Gamma(E) \to \cin (E^{\svee}), \qquad \eta_E (s) (q, \ell) = \langle \ell, s(q)\rangle
\end{equation}
(where $q\in M$, $\ell \in E^{\svee}_q$ and $\langle \cdot\,, \cdot \rangle$ the canonical pairing between the fibers of the bundles $E\to M$ and $E^{\svee}\to M$) is the universal arrow from the $\cin(M)$-module $\Gamma(E)$ of sections of $E$ to the functor $U: \cin(M)/\cring \to \cin(M)\textrm{-}\Mod$.
\end{theorem}

\begin{example}[$b$-Tangent and $b$-cotangent bundles]\label{ex:4.3}
 Let $(M,Z)$ be a $b$-manifold, that is, $Z \subset M$ is a codimension-one submanifold (see \cite{VMR}). Let $\scM={}^b\mathfrak X(M)$ be the $\cin(M)$-module of vector fields on $M$ that are tangent to $Z$. Note that ${}^b\mathfrak X(M)$ is a locally free $\cin(M)$-module, generated locally by
\[
x_1 \partial_{x_1}, \partial_{x_2}, \dots, \partial_{x_n}
\]
in a coordinate chart $(x_1,\dots,x_n)$ adapted to $Z=\{x_1=0\}$.

By the Serre--Swan theorem \cite{Jet, Swan} the finitely generated projective module ${}^b\mathfrak X(M)$ is isomorphic to the module of sections of a vector bundle ${}^bTM \to M$, called the $b$-tangent bundle of $(M,Z)$.

By Theorem~\ref{thm:4.01}, the $\cin$-ring $\cF({}^b\mathfrak X(M))$ is isomorphic to the ring of smooth functions $\cin({}^bT^\svee M)$ on the $b$-cotangent bundle ${}^bT^\svee M$, that is, on the dual bundle of ${}^bTM$.
\end{example}

Our proof of Theorem~\ref{thm:4.01} is fairly long.
We start with a check that the maps $\eta_E: \Gamma(E) \to \cin (E^{\svee})$ of Theorem~\ref{thm:4.01} are natural in the vector bundle $E$.

\begin{lemma} \label{lem:4.02}
Let $\tau_i:H_i\to M$, $i=1,2$ be two vector bundles over a manifold $M$ and $f:H_1\to H_2$  a  vector bundle map (covering $\id_M$). Let $f^{\svee}:H_2^{\svee} \to H^{\svee}_1$ be the dual map, $(f^{\svee})^*: \cin(H^{\svee}_1) \to \cin(H_2^{\svee})$ the induced pullback map and 
\[
f_* :\Gamma(H_1) \to \Gamma(H_2), \quad f_* (s):= f\circ s
\] 
the induced map on sections. Then the diagram 
\begin{equation} \label{eq:4.2.02}
\xy
(-30,10)*+{\Gamma(H_1 )}="1";
(15,10)*+{\cin(H_1^{\svee})}="2";
 (-30,-8)*+{\Gamma(H_2)}="3";
(15,-8)*+{\cin(H_2^{\svee})}="4"; 
{\ar@{->}_{f_*} "1";"3"};
{\ar@{->}^{(f^{\svee})^*} "2";"4"};
{\ar@{->}^{\eta_{H_1}} "1";"2"};
{\ar@{->}_{\eta_{H_2}} "3";"4"};
\endxy
\end{equation}
commutes.  The maps $\eta_{H_i}$, $i=1,2$, are given by \eqref{eq:4.2.01}, {\em mutatis mutandis}.
\end{lemma}
\begin{proof}
The proof is a computation. Given a point $q\in M$ we denote the linear map on fibers by $f_q$:  $f_q: (H_1)_q\to (H_2)_q$.  Then the dual of $f_q$ is $(f^{\svee})_q: (H_2)_q^{\svee} \to (H_2)_q^{\svee}$, i.e., $(f_q)^{\svee} = (f^{\svee})_q$.  Now given a section $s\in \Gamma(H_1)$ and a vector $\ell \in (H_2^{\svee})_q$ we have
\[
\begin{split}
((f^{\svee})^* \eta_{H_1}(s))(q, \ell) &= \eta_{H_1}(s)\left((f^{\svee})(q,\ell)\right) = \langle (f^{\svee})_q \ell, s(q)\rangle = \langle (f_q)^{\svee} \ell, s(q)\rangle \\
&= \langle \ell, f_q (s(q))\rangle = \langle \ell, (f_* s)(q)\rangle = \eta_{H_2}(f_*s) (q,\ell).
\end{split}
\]
Hence $(f^{\svee})^* \eta_{H_1}(s) = \eta_{H_2}(f_*s)$ for all sections $s$ of $H_1\to M$ and we are done.    
\end{proof}
Next we prove a special case of Theorem~\ref{thm:4.01}, which may also work as a sanity check.
\begin{lemma}  \label{lem:4.03}
  Let $p:H\to M$ be a {\em trivial} vector bundle over a manifold $M$, $p^{\svee}:H^{\svee}\to M$ the dual bundle.  Then the map
\[
\eta_H: \Gamma(H) \to \cin (H^{\svee}), \qquad \eta_H (s) (q, \ell) = \langle \ell, s(q)\rangle
\] 
(where as in Theorem~\ref{thm:4.01} $q\in M$, $\ell \in H^{\svee}_q$ and $\langle \cdot\,, \cdot \rangle$ the canonical pairing between the fibers of the bundles $H\to M$ and $H^{\svee}\to M$) is the universal arrow from the $\cin(M)$-module $\Gamma(H)$ of sections of $H$ to the functor $U: \cin(M)/\cring \to \cin(M)\textrm{-}\Mod$.  
\end{lemma}
\begin{proof}
Since the bundle $H$ is trivial, there is a finite dimensional vector space $V$ and an isomorphism $f:H\to M\times V$ of vector bundles over $M$. By Lemma~\ref{lem:4.02} the diagram
\[
\xy
(-30,10)*+{\Gamma(H)}="1";
(15,10)*+{\cin(H^{\svee})}="2";
 (-30,-8)*+{\Gamma(M\times V)}="3";
(15,-8)*+{\cin(M \times V^{\svee})}="4"; 
{\ar@{->}_{f_*} "1";"3"};
{\ar@{->}^{(f^{\svee})^*} "2";"4"};
{\ar@{->}^{\eta_{H}} "1";"2"};
{\ar@{->}_{\eta_{M\times V}} "3";"4"};
\endxy
\]

By Remark~\ref{rmrk:pedantic_arrow} it is enough to argue that $\eta_{M\times V}$ is a universal arrow.  Recall that for any two manifolds $N_1$, $N_2$ and the projections $p_i:N_1\times N_2 \to N_i$ the diagram
\[
\cin(N_1) \xrightarrow{p_1^*} \cin(N_1\times N_2) \xleftarrow{p_2^*} \cin(N_2)
\]
is the coproduct diagram (see \cite{MR}).  Thus 
\[
\cin(M) \xrightarrow{p_M^*} \cin(M\times V^{\svee}) \xleftarrow{p_{V^{\svee}}^*} \cin(V^{\svee})
\]
makes $\cin(M\times V^\svee)$ into the coproduct $\cin(M)\otimes_{\infty} \cin(V^{\svee})$ (where $p_M, p_{V^{\svee}}$ are the evident projections).  Choose a basis $\{f_1, \ldots, f_k\}$ of the vector space $V$.   If we view $f_1,\ldots, f_k$ as constant sections of the product vector bundle $M\times V\to V$, then $\{f_1, \ldots, f_k\}$ is a basis of the $\cin(M)$-module of sections $\Gamma(M\times V)$ of the product bundle $M\times V\to V$.  Thus, an arbitrary section $s\in \Gamma(M\times V)$ is of the form 
\[
s = \sum a_i f_i
\]
for some (unique) functions $a_1, \ldots, a_k \in \cin(M)$. 
If we view $f_1,\ldots, f_k$ as real-valued functions on $V^{\svee}$ then the set $\{f_1, \ldots, f_k\}$ freely generates the $\cin$-ring $\cin(V^{\svee})$.  In the notation of Lemma~\ref{lem:3.08}
\[
\cin(V^{\svee}) = \bbF (\{f_1, \ldots, f_k\}).
\]
By Lemma~\ref{lem:3.08} and in particular by \eqref{eq:eta_free},  the map
\[
\begin{split}
\eta_{\Gamma(M\times V)} :\Gamma(M\times V) &\xrightarrow{\quad } \cin(M)\otimes_\infty \bbF(\{f_1, \ldots, f_k\}) = \cin(M)\otimes_\infty \cin(V^{\svee}) = \cin(M\times V^{\svee}),\\ 
\eta_{\Gamma(M\times V)}(\sum a_i f_i) &=  \sum a_i f_i 
\end{split}
\]
is a universal arrow.  On the other hand for any point $(q,\ell) \in M\times V^\svee$ and any section $s = \sum _i a_i f_i$
\[
\begin{split}
\eta_{M\times V} (s)\, (q, \ell) &= \langle \ell, (\sum a_i f_i)\, (q)\rangle = \langle \ell, \sum a_i(q) f_i \rangle \\ &= \sum_i a_i(q) \langle \ell, f_i\rangle = \sum a_i (q, \ell)\cdot f_i (q, \ell) = (\sum a_i f_i) (q, \ell).
\end{split}
\]
Thus 
\[
\eta_{M\times V} (\sum a_i f_i) = \sum a_i f_i = \eta_{\Gamma(M\times V)}(\sum a_i f_i).
\]
Therefore $\eta_{M\times V}$ is a universal arrow, and we are done.
\end{proof}

Next we need a version of Hadamard's lemma.
\begin{theorem} \label{thm:4.4}
Let $H\xrightarrow{\tau} B$ be a vector bundle, $B\xrightarrow{z}H$ the zero section and $I \subset \cin(H)$ the ideal of functions that vanish on the zero section $z(B)$.  For any set $\{\nu_1,\ldots \nu_k\}$ of generators of the $\cin(B)$-module $\Gamma(H^{\svee})$ of sections of the dual bundle $H^{\svee} \xrightarrow{\tau^{\svee}} B$ the set of functions $\{\eta(\nu_1), \ldots, \eta(\nu_k)\}$ generates the ideal $I$. Here 
\[
    \eta:  \Gamma(H^{\svee}) \to \cin(H), \qquad \eta(s) (q,v):= \langle s(q), v\rangle
\]
is the canonical embedding of the set of sections of the dual bundle into smooth functions on the total space of the bundle $H\to B$.
\end{theorem}

\begin{proof}
Let $f\in I$ be a function in the ideal $I$.  Then for any point $q\in B$, $f(q,0) =0$.  Therefore, for any vector $v\in H_q$, the fiber of $H$ above $q\in B$ we have
\begin{equation}
f(q,v) = f(q,v) -f(q,0) = \int_0^1 \left. \frac{d}{ds}\right|_{s=t} f(q,sv)\, dt.
\end{equation}
Recall next that for any function $f\in \cin(H)$ we have the Legendre transform 
\[
\cL_f:H\to H^{\svee}, \qquad \langle \cL_f(q,v), w\rangle = \left. \frac{d}{dt}\right|_{t=0} f(q,v + tw)
\]
The Legendre transform $\cL_f$ is a smooth map between vector bundles covering the identity map on the base.  It is not, in general, linear on the fibers.
However, if $f = \eta(s)$ for a section $s\in \Gamma(H)$ 
then
\[
\left. \frac{d}{dt}\right|_{t=0} \eta(s)(q,v + tw) = \left. \frac{d}{dt}\right|_{t=0} \langle s(q), v + tw\rangle = \langle s(q), w\rangle.
\]
That is,
\begin{equation}
    \cL_{\eta(s)}  = s\circ \tau,
\end{equation}
where as before $\tau:H\to B$ is the canonical projection.
It follows from the above discussion that for any function $f\in \cin(H)$
\[
\left. \frac{d}{ds}\right|_{s=t} f(q,sv) = \left. \frac{d}{ds}\right|_{s=0} f(q,tv + sv) = \langle \cL_f (q,tv), v\rangle
\]
For any fixed point $q\in B$ the path
\[
t\mapsto \cL_f(q, tv)
\]
is a smooth path in a finite dimensional vector space $H^{\svee}_q$.   Therefore, the integral $\int_0^1 \cL_f (q,tv)\,dt$ makes sense.  Moreover, 
\[
\int_0^1 \langle \cL_f (q,tv),v\rangle \, dt = \langle  \int_0^1 \cL_f (q,tv)\, dt , v\rangle.
\]
We conclude that for any $f\in I$
\begin{equation}
    f(q,v) = \langle \int_0^1 \cL_f (q,tv)\, dt , v\rangle.
\end{equation}
Therefore, given a function $f\in I$ we define $\phi_f:H\to H^{\svee}$ by
\[
\phi_f (q,v) := \int_0^1 \cL_f (q,tv)\, dt .
\]
To finish the proof it is enough to show that there exist functions $h_1,\ldots, h_k \in \cin(H)$ so that 
\begin{equation} \label{eq:D.1.4}
\phi_f (q,v) = \sum_i h_i(q,v ) \nu_i (q).
\end{equation}
For then
\[
f(q,v) = \langle \sum h_i (q,v ) \,\nu_i(q), v\rangle = \sum h_i(q,v)\cdot \langle \nu_i(q),v\rangle
\]
for all $(q,v)\in H$, i.e.
\[
f = \sum h_i \cdot \eta (\nu_i).
\]
\mbox{}\\
The trick to proving \eqref{eq:D.1.4} is to view the set
\[
\{\phi:H\to H^{\svee} \mid \tau^{\svee} \circ \phi = \tau\}
\]
of all (not necessarily linear) maps between the two vector bundles as a module  $\Gamma_\tau (H^{\svee})$ (over $\cin(H)$) of sections of the vector bundle $H^{\svee}\to B$ along the map $\tau:H\to B$.  The desired result then follows from Lemma~\ref{lem:below} below.
\end{proof}

\begin{lemma} \label{lem:below}
    Let $E\to B$ be a vector bundle and $F:N\to B$ a smooth map.  If a set $\{s_1,\ldots, s_k\}$ generates the $\cin(B)$-module $\Gamma(E)$ of sections of $E$, then the set $\{s_1\circ F, \ldots, s_k \circ F\}$ generates the $\cin(N)$-module $\Gamma_F(E)$ of sections of $E$ along $F$.
\end{lemma}
\begin{proof}
On one hand the module $\Gamma_F(E)$  is canonically isomorphic to the module $\Gamma(F^* E)$ of the pullback $F^*E\to N$ of the vector bundle $E$ along the map $F$.  On the other hand there is an isomorphism 
\[
 \cin(N)\otimes_{\cin(B)} \Gamma(E)\xrightarrow{\quad \simeq\quad } \Gamma(F^* E), \qquad h\otimes s \mapsto h\cdot F^*s
\] 
of $\cin(N)$-modules. 
Tracing through the two isomorphisms we see that for any $\phi\in \Gamma_F (E)$ there exist $h_1,\ldots, h_k \in \cin(N)$ so that 
\[
\phi = \sum_i h_i \cdot (s_i \circ F) .
\]
\end{proof}

\begin{proof}[Proof of Theorem~\ref{thm:4.01}]  There is a vector bundle $F\to M$ so that the Whitney sum $H:= E\oplus F\to M$ is a trivial bundle.   We then have vector bundle surjections $p_E:H\to E$, $p_F:H\to F$ and inclusions $s_E:E\to H$, $s_F:F\to H$ with 
\[
\ker (p_E) = \im (s_F).
\]
By Lemma~\ref{lem:4.02} the diagram
\begin{equation} \label{eq:4.5.01}
\xy
(-30,10)*+{\Gamma(H )}="1";
(15,10)*+{\cin(H^{\svee})}="2";
 (-30,-8)*+{\Gamma(E)}="3";
(15,-8)*+{\cin(E^{\svee})}="4"; 
{\ar@{->}_{ (p_E)_*} "1";"3"};
{\ar@{->}^{(p_E^{\svee})^*} "2";"4"};
{\ar@{->}^{\eta_{H}} "1";"2"};
{\ar@{->}_{\eta_{E}} "3";"4"};
\endxy
\end{equation}
commutes. By Lemma~\ref{lem:4.03} the map $\eta_H$ is a universal arrow from $\Gamma(H)$ to the functor $U$.  By Lemma~\ref{lem:3.09} the map
\[
\overline{\eta}:\Gamma(E) \to \cin(H^\svee)/\langle \eta_H (\ker ( (p_E)_*)\rangle )
\]
is a universal arrow.  Since $\ker (p_E) = \im (s_F)$,
\[
\ker ( (p_E)_*) = \im ( (s_F)_*).
\]
By Remark~\ref{rmrk:pedantic_arrow}, if 
\begin{equation} \label{eq:4.5.02}
    \ker (( p_E^{\svee})^*)= \langle \eta_H (\ker (p_E)_*) \rangle = \langle \eta_H (\im (s_F)_*) \rangle,
\end{equation}
then $\eta_E: \Gamma(E) \to \cin(E^\svee)$ is a universal arrow.  So to finish our proof of Theorem~\ref{thm:4.01} it is enough to show that \eqref{eq:4.5.02} holds.
By  Lemma~\ref{lem:4.02},
\[
(s_F)_* w  = (s_F^\svee)^*\eta_F(w)
\]
for all sections $w\in \Gamma(F)$. Since the vector bundle $H$ is the Whitney sum of the bundles $E$ and $F$, the dual bundle $H^\svee$ is the Whitney sum of the dual bundles $\tau_{E^\svee} :E^\svee \to M$ and $\tau_{F^\svee} :F^\svee \to M$ with projections $s_E^\svee: H^\svee \to E^\svee$, $s_F^\svee: H^\svee \to F^\svee$ and inclusions $p_{E^\svee}: E^\svee \to H^\svee$ and $p_{F^\svee}: F^\svee \to H^\svee$.  Since $H^\svee$ is a Whitney sum of $E^\svee$ and $F^\svee$ the diagram
\[
\xy
(-15,10)*+{H ^\svee}="1";
(15,10)*+{F^{\svee}}="2";
 (-15,-8)*+{E^\svee}="3";
(15,-8)*+{M}="4"; 
{\ar@{->}_{ s_{E}^\svee} "1";"3"};
{\ar@{->}^{\tau_{F^{\svee}}} "2";"4"};
{\ar@{->}^{s_{F}^\svee} "1";"2"};
{\ar@{->}_{\tau_{E^\svee}} "3";"4"};
\endxy
\]
is a pullback diagram.   In particular $s_{E^\svee}:H^\svee \to E^\svee$ is a vector bundle; it's the pullback of $F^\svee\xrightarrow{\tau_{F^{\svee}}} M$ along $\tau_{E^\svee}$.  The zero section of $H^\svee \xrightarrow{s_{E^\svee}} E^\svee$ is $p_E^\svee$.  It follows that the kernel of the map $(p_E^\svee)^*$ is the ideal 
\[
I = \{f\in \cin(H^\svee)\mid f|_{p_E^\svee (E^\svee)} = 0\}
\]
 of  functions that vanish on the submanifold $p_E^\svee (E^\svee)$ of $H^\svee$.   It remains to prove that 
\[ 
I= \langle \{\eta_F(w) \circ s_F^\svee \mid w\in \Gamma(F)\}\rangle.
\]
Since $p_E\circ s_F = 0$, 
\[
\langle \{\eta_F(w) \circ s_F^\svee \mid w\in \Gamma(F)\}\rangle \subseteq I.
\]
To get the reverse inclusion we apply Theorem~\ref{thm:4.4} to the vector bundle $H^\svee \xrightarrow{s_E^\svee} E^\svee$.  Then for any choice $\{u_1, \ldots, u_k\}$ of generators of the module of sections of the dual bundle $(H^\svee \xrightarrow{s_E^\svee} E^\svee)^\svee$ the set 
$\{\eta(u_1), \ldots, \eta(u_k)\} \subset \cin(H^\svee)$ generates the ideal $I$. Here as in Theorem~\ref{thm:4.4}
\begin{equation} \label{eq:4.5.3}
    \eta : \Gamma ((H^\svee \xrightarrow{s_E^\svee} E^\svee)^\svee) \to \cin(H^\svee)
\end{equation}
is the canonical embedding.  

Since $H^\svee \xrightarrow{s_E^\svee} E^\svee$ is the pullback ${\tau_{E^\svee}}^* (F^\svee \to  M) $, the dual bundle 
${(H^\svee \xrightarrow{s_E^\svee} E^\svee )}^\svee$ 
is the pullback 
${\tau_{E^\svee}}^* (F^{\svee \svee}\to  M) = {\tau_{E^\svee}}^*(F\to M)$.  By Lemma~\ref{lem:below} we have a map 
\[
(\tau_{E^\svee})^*: \Gamma(F)\to \Gamma ({\tau_{E^\svee}}^* F)
\]
which sends a choice of generators of the $\cin(M)$-module $\Gamma(F)$ to a choice of generators of the $\cin(E^\svee)$-module $\Gamma (\tau_{E^\svee}^* F)$.  And then the image of {\em those } generators under the map \eqref{eq:4.5.3} generates the ideal $I$.

We now trace through various identifications, i.e., we compute the composite map
\[
\Gamma(F)\xrightarrow{(\tau_{E^\svee})^*} \Gamma ((H^\svee \to E^\svee)^\svee) \xrightarrow{\eta} \cin(H^\svee).
\]
If we think of total space of the vector bundle $H^\svee \to M$ as 
\[
H^\svee = \{(q, (p_E)^*\ell_1 + (p_F)^* \ell_2 )\mid q\in M, \ell_1\in E_q^\svee, \ell_2\in F_q^\svee\},
\]
then the projection $s_E^\svee :H^\svee \to E^\svee$ is 
\[
s_E^\svee (q, (p_E)^*\ell_1 + (p_F)^* \ell_2 ) = (q, \ell_1)
\]
and consequently the dual bundle $(H^\svee \to E^\svee)^\svee$ is 
\begin{equation} \label{eq:4.5.4}
(H^\svee \to E^\svee)^\svee = \{((q, \ell_1), v) \mid q\in M, \ell_1 \in E_q^\svee, v\in F_q\}. 
\end{equation}
Given \eqref{eq:4.5.4} the map $(\tau_{E^\svee})^*$ is 
\[
(\tau_{E^\svee})^* w) (q, \ell_1) = ((q, \ell_1), w(q))
\]
for all $q\in M$, $\ell_1\in E_q^\svee$.  On the other hand $\eta: (H^\svee \to E^\svee)^\svee) \to \cin(H^\svee)$ is 
\[
\eta(u)( q, (p_E)^*\ell_1 + (p_F)^* \ell_2) = \langle \ell_2, u(q, \ell_1)\rangle.
\]
Therefore 
\[
((\eta\circ \tau_{E^\svee})^*)(w)) (q, (p_E)^*\ell_1 + (p_F)^* \ell_2) = \langle \ell_2, w(q)\rangle = (\eta_F(w) \circ s_F^\svee)  (q, (p_E)^*\ell_1 + (p_F)^* \ell_2)
\]
for all $w\in \Gamma(F)$.  
Thus 
\[
I \subseteq \langle \{\eta_F(w) \circ s_F^\svee \mid w\in \Gamma(F)\}\rangle
\]
and we are done.
\end{proof}

\section{From Poisson $C^\infty$-rings to Lie-Rinehart algebras} \label{sec:5}

Recall that if $P$ is a manifold with a Poisson bivector field $\pi\in \Gamma(\Lambda^2TP)$, then the vector bundle map 
\[
\pi^\#:T^*P\to TP, \quad \pi^\# (q, \alpha) = \imath(\alpha)\pi_q 
\]
given by sending a covector $\alpha\in T^*_q P$ to its contraction with the bivector $\pi_q$ makes the cotangent bundle $T^*P$ of $P$ into a Lie algebroid (see \cite{CFM}, for example).  Similarly, given a commutative algebra $R$ over a field $k$ and a Poisson bracket $\{\cdot\,, \cdot\}: R\times R\to R$ there is an $R$-module map 
\[
\rho: \Omega^1_R \to \Der_k(R)
\]
from the module $\Omega_R^1$ of ordinary K\"ahler differentials to the module of derivations and a $k$-Lie algebra bracket $\lb_{\Omega^1_R}$ on $\Omega^1_R$ that make $((\Omega^1_R, [\cdot\,, \cdot]_{\Omega^1_R}), 
\rho: \Omega^1_R \to \Der_k(R))$ into a Lie-Rinehart algebra. In this section we prove the analogous result for $\cin$-rings: Theorem~\ref{thm:4.1} below.  We then show that the construction is functorial. 

\begin{theorem} \label{thm:4.1}
Let $(\scA, \{\cdot\,, \cdot\})$ be a Poisson $\cin$-ring.  Then the $\scA$-module $\Omega_\scA^1$ of $\cin$-K\"ahler differentials is a real Lie algebra with a bracket $[\cdot\,, \cdot ]_{\Omega^1_\scA}$. The $\scA$-module map 
\[
\rho: \Omega^1_\scA \to \CDer(\scA), \quad \rho (\sum a_i db_i) = \sum a_i \{b_i, \cdot \}
\] 
of Lemma~\ref{lem:2.19}  makes $((\Omega^1_\scA, [\cdot\,, \cdot]_{\Omega^1_\scA}),  \rho: \Omega^1_\scA \to \CDer(\scA))$ into a Lie-Rinehart algebra. 
\end{theorem}

Before proving the theorem, we present an example.
\begin{example}
    Let $\phi\in C^\infty(\mathbb R^3)$. Consider the Poisson structure $\Pb_\phi$ on $\mathbb R^3$ given by:
\[
\{x,y\}_\phi=\frac{\partial \phi}{\partial z},\qquad
\{y,z\}_\phi=\frac{\partial \phi}{\partial x},\qquad
\{z,x\}_\phi=\frac{\partial \phi}{\partial y}.
\]
This Poisson structure was studied in \cite{Pi, LGPV}. The Poisson structure $\Pb_\phi$ admits
$\phi$ as a Casimir function. In particular, the ideal $\langle \phi \rangle$ is a Poisson ideal of the Poisson algebra $(C^\infty(\mathbb R^3),\Pb_\phi)$. This implies that $\Pb_\phi$ descends to the quotient algebra $(\scA=C^\infty(\mathbb R^3)/\langle \phi\rangle, \Pb)$. Viewing $C^\infty(\mathbb{R}^3)$ as a $C^\infty$-ring, the pair $(\scA,\Pb)$ defines a $C^\infty$-Poisson structure \cite{L-poisson}. By Theorem~\ref{thm:4.1}, the module of Kähler differentials \[\Omega^1_{\scA}=\Omega^1(\mathbb R^3)/\langle \phi\Omega^1(\mathbb R^3), d\phi\rangle\] is naturally endowed with the structure of a Lie--Rinehart algebra $(\Omega^1_\scA, [\cdot\,, \cdot]_{\Omega^1_\scA}, \rho: \Omega^1_\scA \to \CDer(\scA))$ over the $C^\infty$-ring $\scA$. The latter can be described as follows: 
\begin{itemize}
   
\item the anchor map of the cotangent Lie algebroid associated with $\Pb_\phi$
\[
\pi^\#\colon \Omega^1(\mathbb R^3)\to \CDer(C^\infty(\mathbb R^3))
\]
$$
\pi^\#(dx)=\frac{\partial\phi}{\partial z} \frac{\partial}{\partial y}- \frac{\partial\phi}{\partial y} \frac{\partial}{\partial z},\;\; \pi^\#(dy)=\frac{\partial\phi}{\partial z} \frac{\partial}{\partial x} - \frac{\partial\phi}{\partial x} \frac{\partial}{\partial z},\;\;\pi^\#(dz)=\frac{\partial\phi}{\partial y}  \frac{\partial}{\partial x} - \frac{\partial\phi}{\partial x} \frac{\partial}{\partial y} 
$$ 
descents to  a well-defined map $\rho\colon \Omega^1_\scA\to \CDer(\scA)$ on the quotient.
 \item The  bracket $\lb_{\phi}$ of the cotangent Lie algebroid associated with $\Pb_\phi$ is given on generators by 
 \[
[dx,dy]_{\phi}=d\!\left(\frac{\partial \phi}{\partial z}\right),\qquad
[dy,dz]_{\phi}=d\!\left(\frac{\partial \phi}{\partial x}\right),\qquad
[dz,dx]_{\phi}=d\!\left(\frac{\partial \phi}{\partial y}\right).
\]
Since \(\phi\) is a Casimir of the Poisson structure $\Pb_\phi$, we have 
\(
\{\phi,f\}_\phi=0\) {for all} \(f\in C^\infty(\mathbb{R}^3)\). This implies that
 for every \(f\in C^\infty(\mathbb{R}^3)\),
\[
[d\phi,df]_{\phi}=d\{\phi,f\}_\phi=0.
\]
Thus, the bracket $\lb_\phi$ descends  a Lie bracket $\lb_{\Omega^1_\scA}$ on the quotient \(\Omega^1_{\scA}\).
\end{itemize}
\end{example}

\begin{proof}[Proof of Theorem \ref{thm:4.1}]
We mimic the argument for Poisson manifolds (see \cite{CFM}).   By Lemma~\ref{lem:2.19}, there are maps of $\scA$-modules 
\[
\rho: \Omega_\scA^1\to \CDer(\scA) \simeq \Hom(\Omega^1_\scA, \scA)
\]
and 
\[
B: \Lambda^2 \Omega^1_\scA \to \scA
\]
with 
\begin{eqnarray}
B(\alpha \wedge \beta)  &= & \imath(\rho(\alpha))\beta\qquad \textrm{ for all } \alpha, \beta\in \Omega^1_\scA \quad \textrm{and} \\
B(da\wedge db) &= &\{a,b\}\qquad \quad \textrm{ for all } a,b\in \scA.
\end{eqnarray}
We define the $\R$-bilinear map
\[
[\cdot\,, \cdot]_{\Omega^1_\scA} : \Omega^1_\scA \times \Omega^1_\scA \to \Omega^1_\scA
\]
by 
\begin{equation}
    [\alpha, \beta]_{\Omega^1_\scA} := \cL(\rho(\alpha))\beta - \cL(\rho(\beta)) \alpha - d (B(\alpha \wedge \beta)).
\end{equation}
Here $\cL: \CDer(\scA)\to \Der^{0} (\Lambda^\bullet\Omega^1_\scA)$ is the Lie derivative (see Definition~\ref{def:Lie_der}).\\[2pt]
\noindent
{\bf Claim 1. } $[da, db]_{\Omega^1_\scA} = d\{a,b\}$ for all $a,b\in \scA$. 
\begin{proof} For any derivation $X\in \CDer(\scA)$ and any element $b\in \scA$
\[
\cL(X) db = d X(b).
\]
Hence
\[
\cL(\rho(da))db  = d \left(\rho(da) b \right) = d\{a,b\}.
\]
Similarly,
\[
\cL(\rho(db))da  =   d\{b,a\}.
\]
Since $B(da\wedge db) = d\{a,b\}$,
\[
[da,db]_{\Omega^1_\scA} = d\{a,b\}- d\{b, a\} - d\{a,b\} = d\{a,b\}
\]
which is Claim~1. \end{proof}
\noindent
{\bf Claim 2} For all K\"ahler differentials $\alpha, \beta\in \Omega^1_\scA$ and for all $f\in \scA$
\begin{equation} \label{eq:3.1.4}
    [\alpha, f\beta]_{\Omega^1_\scA} = (\rho(\alpha)f) \beta + f[\alpha, \beta]_{\Omega^1_\scA},
\end{equation}
the Leibniz rule for the map $\rho$.
\begin{proof}
For any $f\in \scA$, any derivation $X\in \CDer(\scA)$ and any 1-form $\gamma\in \Omega^1_\scA$
\[
\begin{split} 
\cL(fX)\gamma &= \imath(fX) d\gamma + d \imath(fX)\gamma = f \imath(X) d\gamma + d( f \imath(X)\gamma)\\
&= f \imath(X) d \gamma + f d(\imath(X)\gamma) + \imath(X) \gamma\, df \\
&= f \cL(X)\gamma + \imath(X) \gamma\, df.
\end{split}
\]
Therefore 
\[
\begin{split} 
[\alpha, f\beta]_{\Omega^1_\scA} &= \cL(\rho(\alpha)) (f\beta) - \cL(\rho(f\beta)) \alpha - d (B(\alpha \wedge f\beta) \\
&=(\rho(\alpha)f) \,\beta + f\,\cL(\rho(\alpha)) \beta - (f \cL(\rho(\beta))\alpha + \imath(\rho(\beta) \alpha \, df) - d (f B(\alpha \wedge \beta))\\
&= (\rho(\alpha)f) \,\beta + f\,\cL(\rho(\alpha)) \beta - f \cL(\rho(\beta))\alpha - B(\alpha \wedge \beta) df \\
&\qquad - B(\alpha \wedge \beta) \, df - f d (B(\alpha \wedge \beta)\qquad \qquad \qquad (\textrm{since } \imath (\rho(\beta))\alpha = B (\beta \wedge \alpha))\\
&= (\rho(\alpha) f) \beta + f[\alpha , \beta]_{\Omega^1_\scA}.
\end{split}
\] 
\end{proof}
\noindent
{\bf Claim 3. }   The map 
\[
U: \Omega^1 _\scA \times \Omega^1_\scA \to \CDer(\scA),\qquad U(\alpha, \beta) = [\rho(\alpha), \rho(\beta)] - \rho ([\alpha, \beta]_{\Omega^1_\scA})
\]
is $\scA$-linear in each slot.

\begin{proof}
This is a computation that uses Claim~2.  Since $U$ is skew-symmetric, it is enough to check that 
\[
U(\alpha, f\beta ) = f\, U(\alpha, \beta)
\]
for all $f\in \scA$, all $\alpha, \beta\in \Omega^1_\scA$.  We compute

\[
\begin{split} 
U(\alpha, f\beta ) &= [\rho(\alpha), \rho(f\beta)] - \rho ([\alpha, f\beta]_{\Omega^1_\scA})\\
&= [\rho(\alpha), f\rho(\beta)] - \rho( (\rho(\alpha)f)\beta + f \,[\alpha, \beta]_{\Omega^1_\scA})\\
&= (\rho(\alpha) f)\, \rho(\beta )+ f\, [\rho(\alpha), \rho(\beta)] -  \rho(\alpha)(f)\, \rho(\beta) - f \rho([\alpha, \beta]_{\Omega^1_\scA})\\
&= f U(\alpha, \beta).
\end{split}
\]
\end{proof}
\noindent{\bf Claim 4. } The map $\rho: \Omega^1_\scA \to \CDer(\scA)$ preserves brackets.
\begin{proof}
By Claim~3 and the fact that the $\scA$-module $\Omega^1_\scA$ of K\"ahler differentials is generated by the set $\{da\}_{a\in\scA}$ it is enough to check that  $U(da, db) = 0$ for all $a,b\in \scA$, that is, 
\[
\rho([da, db]_{\Omega^1_\scA}) = [\rho(da), \rho(db)]
\]  
for all $a,b\in \scA$.  Recall that for all $a,b, e\in \scA$, 
\[
\rho(da) e = \{a, e\} \quad \textrm{and } \rho([da, db]_{\Omega^1_\scA}) e = \{\{a,b\}, e\}.
\]
Since the Poisson bracket $\{\cdot\,, \cdot\}$ satisfies the Jacobi identity,
\[
\begin{split} 
[\rho(da), \rho(db)] e - \rho ([da, db]_{\Omega^1_\scA})e &= \rho(da) (\rho (db) e) - \rho(db)(\rho(da) e) - \{\{a,b\}, e\}\\
& = \{a, \{b, c\}\} - \{b, \{a, e\}\} - \{\{a,b\}, e\} = 0.
\end{split}
\]
\end{proof}
\noindent{ \bf Claim 5.}  The Jacobiator map 
\[
J: \Omega^1_\scA \times \Omega^1_\scA \times \Omega^1_\scA \to \Omega^1_\scA,\]
\[J(\alpha, \beta, \gamma)  := [\alpha, [\beta, \gamma]_{\Omega^1_\scA}]_{\Omega^1_\scA} - [[\alpha, \beta]_{\Omega^1_\scA}, \gamma]_{\Omega^1_\scA} - [\beta, [\alpha, \gamma]_{\Omega^1_\scA}]_{\Omega^1_\scA}
\]
is $\scA$-linear in each slot.
\begin{proof}
    This is a computation using \eqref{eq:3.1.4}, the Leibniz rule for $\rho$.
\end{proof}
\noindent {\bf Claim 6. } For all $a,b, c\in \scA$ the Jacobiator $J(da, db, dc)$ is 0.
\begin{proof}
Recall that for all $f,g\in \scA$
\[
[df, dg]_{\Omega^1_\scA} = d\{f, g\}.
\]
Therefore,
\[
\begin{split}
  J(da, db, dc)&= [da, d\{b,c\}]_{\Omega^1_\scA} -[d\{a, b\}, dc]_{\Omega^1_\scA} - [db, d\{a, c\}]_{\Omega^1_\scA}\\
 &= d (\{a, \{b, c\} \}- \{\{a,b\}, c\} - \{ b, \{a, c\}\}) = d (0) = 0.
\end{split}
\]
\end{proof}
The fact that the exact differentials generate $\Omega^1_\scA$ together with the Claims 5 and 6 imply that the Jacobiator $J$ vanishes on all K\"ahler differentials.  Thus, the bracket $[\cdot\,, \cdot]_{\Omega^1_\scA}$ on $\Omega^1_\scA$ is a Lie bracket.  Since the anchor map $\rho$ preserves this bracket, satisfies the Leibniz rule and is $\scA$-linear, as we showed above, we are done.
\end{proof}
In Theorem~\ref{thm:4.1}  we constructed a Lie-Rinehart algebra out of a Poisson $\cin$-ring.  It is natural to wonder if this construction is functorial.  
Recall that by Remark~\ref{rmrk:Omega_functor} given a map $\varphi:\scA\to \scB$ of $\cin$-rings we get a map $\Omega^1_\varphi:\Omega^1_\scA \to \Omega^1_\scB$ of modules so that \eqref{eq:2.12.1} commutes.  In particular 
\[
\Omega^1_\varphi (a_1 da_2) =\varphi(a_1) d \varphi(a_2)
\]
for all $a_1, a_2\in \scA$. It is straightforward to verify that if $\varphi:\scA\to \scB$ and $\psi:\scB\to \scC$ are morphisms of $\cin$-rings, then
\[
\Omega^1_{\psi \circ \varphi} = \Omega^1_\psi \circ \Omega^1_\varphi\quad\text{and}\quad \Omega^1_{\mathrm{Id}_\scA}=\mathrm{Id}_{\Omega^1_\scA}.
\]
In other words $\Omega^1$ is a functor from $\cin$-rings to modules.
\begin{lemma} \label{lem:5.02}
    If $\varphi:\scA\to \scB$ is a map of {Poisson} $\cin$-rings (Definition~\ref{def:Poisson_map}), then $(\varphi, \Phi:= \Omega^1_\varphi): (\Omega^1_\scA \xrightarrow{\rho_\scA}\cin\mathrm{Der}(\scA)) \to (\Omega^1_\scB \xrightarrow{\rho_\scB}\cin\mathrm{Der}(\scB))$ is a map of Lie-Rinehart algebras (Definition~\ref{def:mor_LR}).
\end{lemma}

\begin{proof}
It is enough to check that (1) $\Phi: \Omega^1_\scA\to \Omega^1_\scB$ 
preserves Lie brackets and that (2) for any K\"ahler differential  $\alpha\in \Omega^1_\scA$ the derivations $\rho_\scB (\Phi(\alpha))$ and $\rho_\scA(\alpha)$ are $\varphi$-related.

Since the module $\Omega^1_\scA$ is generated by the set $\{da \}_{a\in \scA}$ of exact differentials (we dropped the subscript $_\scA$ on the universal derivations $d$), to prove (1) it is enough to check that 
\[
\Phi([da_1, a_2da_3]_{\Omega^1_\scA}) = [\Phi(da_1), \Phi(a_2da_3)]_{\Omega^1_\scB}
\]
for all $a_1, a_2, a_3 \in \scA$.   Note that 
\[
\begin{split}
    \Phi(da_1) &= d \varphi(a_1)\\
    \Phi(a_2 da_3) &= \varphi(a_2) d \varphi(a_3)\qquad \textrm{ and }\\
    [da_1, a_2 da_3]_{\Omega^1_\scA}&= \{a_1, a_2\} da_3 + a_2 d\{a_1, a_3\}.
\end{split}
\]
Therefore
\[
\begin{split}
\Phi([da_1, a_2da_3]_{\Omega^1_\scA}) &= \varphi(\{a_1, a_2\} ) d\varphi(a_3) + \varphi(a_2) d\varphi(\{a_1, a_3\})\\
&=\{\varphi(a_1), \varphi(a_2)\} d\varphi(a_3)  + \varphi(a_2) d\{\varphi(a_1), \varphi(a_3)\}\\
&= [\Phi(da_1), \Phi(a_2da_3)]_{\Omega^1_\scB}), 
\end{split}
\]
which proves (1).  

Similarly to prove (2) we may assume that $\alpha =a_1 da_2$.    Then for any $a_3 \in \scA$
\[
\begin{split}
\varphi (\rho_\scA (\alpha) a_3) &= \varphi(a_1 \{a_2, a_3\})  = \varphi(a_1) \{\varphi(a_2), \varphi(a_3)\}\\
&= \rho_\scB (\varphi(a_1) d\varphi(a_2))(\varphi (a_3))\\
&= \left( \rho_\scB (\Phi(\alpha) ) \circ \varphi \right) (a_3).
\end{split}
\]
Hence 
\[
\rho_\scB (\Phi(\alpha) ) \circ \varphi  = \varphi \circ \rho_\scA (\alpha),
\]
which proves (2).
\end{proof}
\begin{remark}
It follows from Lemma~\ref{lem:5.02} that the map 
\[
(\scA, \{\cdot\,, \cdot\}) \mapsto (\Omega^1_\scA \xrightarrow{\rho_\scA} \CDer(\scA))
\]
extends to a functor from the category $\pcring $ of Poisson $\cin$-rings (Definition~\ref{def:pscring})  to the category $\clr$ of Lie-Rinehart algebras (Remark~\ref{rmrk:CLR}).
\end{remark}

\section{Lie-Rinehart algebras $\simeq$ ``linear" Poisson $C^\infty$-rings} \label{sec:6}

In this section we generalize a theorem of Ted Courant \cite[Theorem 2.1.4]{courant} who proved that for any Lie algebroid $\rho: E\to M$ over a manifold $M$ there is a linear Poisson bracket on the total space of the dual vector bundle $E^{\svee}\to M$.  The analogous result for a Lie-Rinehart algebra $((\scM, [\cdot\,, \cdot]_\scM), \rho:\scM\to \Der (A))$ with $A$ a commutative ring is that the symmetric algebra $S^\bullet(\scM)$ of the $A$-module $\scM$ is a Poisson algebra \cite[Example 3.16]{Hueb}. 

We start with a definition of a linear Poisson bracket on the free (``symmetric") $\scA$-algebra $\cF(\scM)$ of an $\scA$-module $\scM$ (cf.\ Definition~\ref{def:sym-alg}).   
While there are various analogues of vector bundles in fairly abstract category theoretic settings (for example they occur in the theory of tangent categories of Rosicky \cite{Ros}), the algebra $\cF(\scM)$ does not appear to have such a structure.  So the word ``linear" in Definition~\ref{def:lin_Poisson} should be taken with a grain of salt.

\begin{definition} \label{def:lin_Poisson}
    Let $\scA$ be a $\cin$-ring, $\scM$ an $\scA$-module and $\cF(\scM)$ the corresponding free $\scA$-algebra (cf.\ Definition~\ref{def:sym-alg}) with the structure maps  $\varphi_\scM:\scA\to \cF(\scM)$ and $\eta_\scM:\scM \to U(\cF(\scM))$ (Notation~\ref{nota:phi_eta}).  A Poisson bracket $\{\cdot\,, \cdot\}$ on the $\cin$-ring $\cF(\scM)$ is {\sf linear} if the following three conditions hold:
\[
\{\eta_\scM(m_1), \eta_\scM(m_2 )\}
\in  \eta_\scM(\scM)
\]
for all $m_1, m_2\in \scM$,
\[
\{\eta(m), \varphi(a)\}\in \varphi(\scA)
\]
for all $m\in \scM$ and $a\in \scA$ and
\[
\{\varphi_\scM(a), \varphi_\scM(b)\}=0
\]
for all $a, b\in \scA$.

\end{definition}
\begin{theorem}\label{thm:main}
Let $\scA$ be a $C^\infty$-ring and let $\scM$ be an $\scA$-module. 
The following statements are equivalent:

\begin{enumerate}
\item \label{thm:main1} The module $\scM$ admits a Lie-Rinehart algebra structure 
$((\scM, [\cdot\,,\cdot]_{\scM}), \rho:\scM \to \CDer(\scA))$.

\item \label{thm:main2} The $\cin$-ring $\mathcal F(\scM)$ carries a  linear Poisson  $C^\infty$-ring structure $\left(\mathcal F(\scM),  \Pb_{\mathcal F(\scM)}\right)$ (Definition~\ref{def:lin_Poisson}).
\end{enumerate}
Furthermore, the Lie–Rinehart structure on the $\scA$-module $\scM$ and the linear Poisson structure on the $\cin$-ring $\mathcal F(\scM)$ (Definition~\ref{def:lin_Poisson})  uniquely determine each other by way of the following equations:

\begin{align}\label{eq:Rinehart-Poisson1}
    \left\{\eta_{\scM}(m_1), \eta_{\scM}(m_2)\right\}_{\cF(\scM)}
&=
\eta_{\scM}\left([m_1, m_2]_{\scM}\right)\\
\nonumber&\\\label{eq:Rinehart-Poisson2}\{\eta_\scM(m), \varphi_\scM(a)\}_{\mathcal{F}(\scM)}&=\varphi_\scM(\rho(m)a)\\\nonumber&\\
\label{eq:Rinehart-Poisson3}\{\varphi_\scM(a), \varphi_\scM(b)\}_{\mathcal{F}(\scM)}&=0
\end{align}
for all $m_1, m_2,m \in \scM$ and $a, b\in \scA$, where $\scA\xrightarrow{\varphi_\scM}\cF(\scM)$ and  $\eta_\scM\colon \scM \to \mathcal{F}(\scM)$ are the structure maps (Notation~\ref{nota:phi_eta}). 
\end{theorem}

\begin{remark}
    Although the bracket $\Pb_{\mathcal F(\scM)}$ is completely well-defined on the $\scA$-modules $\eta_\scM(\scM)$ and $ \varphi_\scM(\scA)\subset \mathcal F(\scM)$, this does not guarantee that it extends to a well-defined bracket on $\mathcal F(\scM)$, even though $\eta_\scM(\scM)$ and $\varphi_\scM(\scA)$ generate $\mathcal F(\scM)$. The issue is that, for example,  an element $u\in \mathcal F(\scM)$ may admit different realizations of the form
\[
u = f_{\mathcal F(\scM)}(\eta_{\scM}(m_1), \ldots, \eta_{\scM}(m_n))
   = g_{\mathcal F(\scM)}(\eta_{\scM}(m'_1), \ldots, \eta_{\scM}(m'_d)),
\]
for different smooth functions $f\colon \mathbb R^n\to \mathbb R$, $g\colon \mathbb R^d\to \mathbb R$, and elements $m_i,m'_j\in \scM$. One must therefore verify that the extension is compatible with these relations.
\end{remark}
\begin{remark}[Lie-Poisson structure for an arbitrary Lie algebra $\fg$]
Let $\scA$ be a $\cin$-ring and $((\scM,[\cdot\,,\cdot]_{\scM}),\rho:\scM\to \CDer_{\mathbb R}(\scA))$  be a Lie--Rinehart algebra. If $\rho=0$, then $\scM$ reduces to a Lie $\mathbb R$-algebra whose bracket
\(
[\cdot\,,\cdot]_{\scM}:\scM\times\scM\to\scM
\)
is $\scA$-linear in both arguments. Consequently, the associated Poisson bracket on $\mathcal{F}(\scM)$ is likewise $\scA$-bilinear. 

In particular, when $\scA=\mathbb R$ (cf.\ Example~\ref{ex:smooth-functions}) then $\CDer(\scA) = \CDer(\R) = 0$, the anchor map $\rho$ is forced to be 0, 
and the Lie-Rinehart algebra $((\scM,[\cdot\,,\cdot]_{\scM}),\rho:\scM\to \CDer(\R))$ is simply a (possibly infinite-dimensional) Lie $\mathbb R$-algebra $(\scM,[\cdot\,,\cdot]_{\scM})$. 

If the Lie algebra $\scM$ is finite-dimensional, one recovers the classical Lie--Poisson structure on the algebra of smooth functions on the dual space $\scM^\svee$. More generally, by Corollary~\ref{cor:3pi}, 
for any real Lie algebra $\scM$, finite dimensional or not, the free (``symmetric") $\R$-algebra $\cF(\scM)$ is the algebra $\cin(\scM^\svee)$ of smooth functions on the dual. Thus $\cin(\scM^\svee)$ is a Poisson $\cin$-ring, and we obtain a Lie-Poisson structure on the algebra of smooth functions on the dual space $\scM^\svee$. 

In other words
the linear Poisson $\cin$-ring structure on $\cF(\scM) = \cin(\scM^\svee)$ remains well-defined without any additional assumptions, even when the Lie algebra $\scM$ is infinite-dimensional.  Note that the use of $\cin$-rings obviates the need to define topologies on the module $\scM$ and on its dual $\scM^\svee$.
\end{remark}

Before proving Theorem~\ref{thm:main} we make an observation:
in the special case where the $\scA$-module $\scM$ is faithfully generated in the sense of Definition~\ref{def:tor-free}, equation \eqref{eq:Rinehart-Poisson1} implies \eqref{eq:Rinehart-Poisson2} and, if we assume that  $\{\varphi_\scM(a), \varphi_\scM(b)\} \in \varphi(\scA)$, implies \eqref{eq:Rinehart-Poisson3} as well --- see Lemma~\ref{lem:6.05} below.
\begin{definition}\label{def:tor-free}
Let $\scM$ be a module over a $\cin$-ring $\scA$. The $\scA$-module $\scM$ is {\sf faithfully generated} if it admits a set of generators $\{m_j\}_{j \in J}$ such that, for every $a \in \scA$, if $a m_j = 0$ for all $j \in J$, then $a = 0$. 
\end{definition}
\begin{remark}%
There are many examples of faithfully generated modules. For instance
any free  module is faithfully generated. 
Since the module of sections of a vector bundle is locally free it is faithfully generated. Torsion free modules over an integral domain are also faithfully generated. 

On the other hand, for $\scA=C^\infty(\mathbb R)$, the $\scA$-module $\scM=C^\infty(\mathbb R)/\langle x \rangle \simeq \mathbb R$ is not faithfully generated.
\end{remark}
\begin{remark}
Since the underlying $\mathbb R$-algebra of a $C^\infty$-ring  $\scA$ typically contains zero-divisors (for instance, in $C^\infty(\mathbb R)$ any function that vanish on an nonempty open set is a zero divisor), the fact that a module $\scM$ is torsion-free (meaning  there is no non zero-divisor $a \in\scA$ such that $am=0$) does not, in general, imply that it is  faithfully generated in the sense of Definition \ref{def:tor-free}.
\end{remark}

\begin{lemma} \label{lem:6.05}
Let $((\scM, [\cdot\,, \cdot]_\scM),  \rho: \scM\to \CDer(\scA))$ be a Lie-Rinehart algebra and 
\[
 \scM \xrightarrow{\eta_\scM} \cF(\scM) \xleftarrow{\varphi_\scM} \scA
\]
the corresponding free $\scA$-algebra with a $\cin$-ring Poisson bracket $\Pb$.  Suppose 
\begin{equation} \label{eq:6.8.1}
 \left\{\eta_{\scM}(m_1), \eta_{\scM}(m_2)\right\}
=\eta_{\scM}\left([m_1, m_2]_{\scM}\right)\
\end{equation}
for all $m_1, m_2\in \scM$ and
\begin{equation}\label{eq:6.8.2}
\{\varphi_\scM(a), \varphi_\scM(b)\} \in \varphi(\scA)
\    
\end{equation}
for all $a,b\in \scA$.
Then
\begin{align}
  \label{eq:6.8.3}  \{\eta_\scM(m), \varphi_\scM(a)\}&=\varphi_\scM(\rho(m)a) \\\nonumber \textrm{and}\\
  \label{eq:6.8.4}    \{\varphi_\scM(a), \varphi_\scM(b)\} &= 0
\end{align}
for all $m\in\scM$ and $a,b\in \scA$.
\end{lemma}
\begin{proof} This is a consequence of the Leibniz rule for the Poisson bracket $\Pb$, the Leibniz rule  the Lie bracket $\lb_\scM$, and of injectivity of  the structure maps $\eta_\scM: \scM\to \cF(\scM)$ and $\varphi_\scM: \scA\to \cF(\scM)$.
We now drop the subscript $_\scM$ to reduce the clutter and provide the details.

We can think of the $\scA$-algebra $\cF(\scM)$ as an $\scA$-module. Then $\eta(\scM)$ is an $\scA$-submodule of $\cF(\scM)$.  The injectivity of $\eta$ implies that if $\scM$ is faithfully generated, then $\eta(\scM)$ is also faithfully generated.

Since $\eta$ is a map of $\scA$-modules,  
\[
\eta (bm') = \varphi(b)\, \eta(m')
\]
for any $b\in \scA$ and any $m'\in \scM$. By \eqref{eq:6.8.1}  
\begin{equation} \label{eq:6.8.5}
 \left\{\eta(m), \eta(bm')\right\}
{=}
\eta\left([m, bm']\right).
\end{equation}
Since $\scM\xrightarrow{\rho} \CDer(\scA) $ is a Lie-Rinehart algebra, $[m, bm'] = (\rho(m)b)m' + b [m,m']$.  Therefore \eqref{eq:6.8.5} becomes
\begin{equation} \label{eq:6.8.6}
   \left\{\eta(m), \varphi(b)\eta(m')\right\} = \eta( (\rho(m)b)m' + b [m,m']).
\end{equation}
Since $\{\cdot \,,\cdot\}$ is a biderivation and since \eqref{eq:6.8.1} is assume to hold, \eqref{eq:6.8.6} becomes
\[
\left\{\eta(m), \varphi(b)\right\}\eta(m')+ {\varphi(b){\left\{\eta(m), \eta(m')\right\}}}=
\varphi(\rho(m)b)\eta(m') + \varphi(b)\{\eta(m), \eta(m') \}.
\]
Hence 
\begin{equation*}\label{eq:zero-div}
    \left(\left\{\eta(m), \varphi(b)\right\}-\varphi(\rho(m)b)\right)\eta(m')=0.
\end{equation*}
Since the module $\eta(\scM)$ is faithfully generated and since 
and  $m'$ is arbitrary, \eqref{eq:6.8.3} holds.

Finally, for any $m \in \scM$ and any $a,b \in \scA$
\[
\begin{split}
    \varphi(b) \{ \eta(m), \varphi(a)\} &\stackrel{\eqref{eq:6.8.3}}{=} \varphi(b) \, \varphi(\rho(m)a) = \varphi(b \rho(m)a) = \varphi(\rho(bm)a) \\
    &\stackrel{\eqref{eq:6.8.3}}{=} \{\varphi(b)\eta(m), \varphi(a)\} = 
    \{\varphi(b), \varphi(a)\} \eta(m) + {\varphi(b) \{\eta(m), \varphi(a)\}}.
\end{split}
\]
Hence 
\[
\{\varphi(b), \varphi(a)\}\, \eta(m) = 0.
\]
Since $\{\varphi(b), \varphi(a)\}\in \varphi(\scA)$, since the $\scA$-module $\eta(\scM)$ is faithfully generated 
and since $m$ is arbitrary, this forces  $\{\varphi(b), \varphi(a)\} =0$ which is 
 \eqref{eq:6.8.4}.
\end{proof}

\begin{proposition} \label{prop:6.7}
If 
\[
(\id, \Phi): (\scM_1 \xrightarrow{\rho_1}\cin\mathrm{Der}(\scA)) \to (\scM_2 \xrightarrow{\rho_1}\cin\mathrm{Der}(\scA))
\]
is a map of Lie-Rinehart algebras (Definition~\ref{def:mor_LR}), then  
\[
\mathcal{F}(\Phi)\colon \mathcal   F(\scM_1)\to \mathcal F(\scM_2)
\]    
is a map of Poisson $\cin$-rings. 
\end{proposition}

\begin{proof}
For the sake of simplicity of notation, we denote the $C^\infty$-Poisson brackets on the $\cin$-rings $\mathcal
F(\scM_1)$ and $\mathcal F(\scM_2)$ by $\Pb_1$ and $\Pb_2$, respectively.
    
Since $\mathcal F$ is a functor, it preserves composition of morphisms. In addition,  the following diagrams 
\begin{align}\label{diag:FunctorLRPoiss}
        \xymatrix{
\scM_1 \ar[r]^{\eta_{\scM_1}} \ar[d]_\Phi & \cF(\scM_1) \ar[d]^{\mathcal{F}(\Phi)} \\
\scM_2 \ar[r]_{\eta_{\scM_2}} & \cF(\scM_2),
}
\qquad
\xymatrix{
\scA \ar[r]^{\varphi_{\scM_1}} \ar@{=}[d]_{\mathrm{id}} & \cF(\scM_1) \ar[d]^{\mathcal{F}(\Phi)} \\
\scA \ar[r]_{\varphi_{\scM_2}} & \cF(\scM_2)
}
    \end{align} 
commute by construction of $\cF(\Phi)$. We temporarily denote $\eta_{\scM_1}$ and $\eta_{\scM_2}$ by $\eta_1$ and $\eta_2$. Likewise, we denote $\varphi_{\scM_1}$ and $\varphi_{\scM_2}$ by  $\varphi_{1}$ and $\varphi_{2}$. We need to show that
\[
\cF(\Phi)(\{u,v\}_1)=\{\cF(\Phi)(u),\cF(\Phi)(v)\}_2
\qquad \forall\, u,v\in \cF(\scM_1).
\]
It is enough to verify this on the generators $\eta_1(\scM_1)$ and $\varphi_{1}(\scA)$. For $m_1,m_2\in \scM_1$, using \eqref{diag:FunctorLRPoiss} and the fact that $\Phi$ is a morphism of Lie algebras,
we get\begin{align*}
\cF(\Phi)\bigl(\{\eta_1(m_1),\eta_1(m_2)\}_1\bigr)
&=\cF(\Phi)\bigl(\eta_1([m_1,m_2]_{{\scM_1}})\bigr) \\
&=\eta_2\bigl(\Phi([m_1,m_2]_{{\scM_1}})\bigr) \\
&=\eta_2\bigl([\Phi(m_1),\Phi(m_2)]_{{\scM_2}}\bigr) \\
&=\{\eta_2(\Phi(m_1)),\eta_2(\Phi(m_2))\}_2 \\
&=\left\{\cF(\Phi)(\eta_1(m_1)),\cF(\Phi)(\eta_1(m_2))\right\}_2.
\end{align*}
For $m\in \scM_1$ and $a\in \scA$, using \eqref{diag:FunctorLRPoiss} and the fact that $\rho_2(\Phi(m))=\rho_1(m)$, we obtain
\begin{align*}
\cF(\Phi)\bigl(\{\eta_1(m),\varphi_{1}(a)\}_1\bigr)
&=\cF(\Phi)\bigl(\varphi_{1}(\rho_1(m)a)\bigr) \\
&=\varphi_{2}\bigl(\rho_1(m)a\bigr) \\
&=\varphi_{2}\bigl(\rho_2(\Phi(m))(a)\bigr) \\
&=\{\eta_2(\Phi(m)),\varphi_{2}(a)\}_2 \\
&=\{\cF(\Phi)(\eta_1(m)),\cF(\Phi)(\varphi_{1}(a))\}_2.
\end{align*}
Finally, for $a,b\in \scA$, by \eqref{diag:FunctorLRPoiss},
\begin{align*}
\cF(\Phi)\bigl(\{\phi_{1}(a),\varphi_{1}(b)\}_1\bigr)
&=\cF(\Phi)(0)=0 \\
&=\{\varphi_{2}(a),\varphi_{2}(b)\}_2 \\
&=\{\cF(\Phi)(\varphi_{1}(a)),\cF(\Phi)(\varphi_{1}(b))\}_2 .
\end{align*}
\end{proof}

\begin{example}
    Let $E\to M$ be a vector bundle over a manifold $M$ and  $\Gamma(E)$ the space of sections of $E$ viewed as a module over the $\cin$-ring $\cin(M)$ of smooth functions on $M$. By Theorem~\ref{thm:4.01}, one has
\[
\cF(\Gamma(E)) \simeq C^\infty(E^\svee)
\]
with the map $\eta_E: \Gamma(E) \to \cin (E^{\svee})$  given by \eqref{eq:4.2.01}.
This recovers the classical correspondence between Lie algebroid structures on \(E\) and linear Poisson structures \(\{\cdot\,,\cdot\}\) on the dual bundle \(E^\svee\). Here, ``linear'' means that
\[
\{C^\infty_{\mathrm{lin}}(E^\svee), C^\infty_{\mathrm{lin}}(E^\svee)\} \subset C^\infty_{\mathrm{lin}}(E^\svee),
\]
where \(C^\infty_{\mathrm{lin}}(E^\svee)\) denotes the space of smooth functions on \(E^\svee\) whose restriction to each fiber is linear. For $\alpha\in \Gamma(E)$,  \(\eta_E(\alpha)\) is the fiberwise linear function on \(E^\svee\) associated to \(\alpha\).

The Poisson bracket \(\{\cdot\,,\cdot\}\) on \(\cF(\Gamma(E))\) coincides with the well-known linear Poisson structure on \(E^\svee\). Since $\Gamma(E)$ is faithfully generated, Lemma~\ref{lem:6.05} implies that the bracket $\Pb$ is determined by the relation
\[
\{\eta_E(\alpha), \eta_E(\beta)\} = \eta_E({[\alpha,\beta]_E}),
\]
for all \(\alpha,\beta \in \Gamma(E)\), together with the fact that, for any $f,g\in \cin(M)$ the brackets $\{\pi^*f, \pi^*g\}$ belongs to $\pi^*\cin(M)$, where $\pi:E^\svee\to M$.
\end{example}
\begin{example}
Let $E\stackrel{\rho}{\longrightarrow} TM$ be a Lie algebroid, and $\scM\hookrightarrow \mathfrak X(M)$ its basic Lie-Rinehart algebra (Definition~\ref{rmrk:basic}). By Proposition \ref{prop:6.7}, there exist Poisson morphisms: 
\begin{align*}
       \phi: C^\infty(E^\svee) \to \cF(\scM) \quad\text{and}\quad \cF(\scM)\to C^\infty(T^\svee M).
\end{align*}
Moreover,  $\mathcal{F(\scM)}=\mathcal{F}(\Gamma(E))/\langle \eta_E(\ker\rho )\rangle=C^\infty(E^\svee)/\langle \eta_E(\ker\rho )\rangle$ (cf. Lemma \ref{lem:3.09}). The latter formula remains valid when \(E \to TM\) is merely an anchored vector bundle over a Lie-Rinehart algebra \(\scM\), that is, when \(\rho(\Gamma(E))=\scM\), without assuming that \(E\) carries a Lie algebroid structure\footnote{Such an anchored bundle exists for any finitely generated singular foliation \cite{LLS, LLL1}. Similar assertion holds for singular Lie subalgebroids.}. In particular, \(\mathcal F(\scM)=C^\infty(E^\svee)/\langle \eta_E(\ker\rho )\rangle\)  carries a linear Poisson structure.

The surjective map $ \phi: C^\infty(E^\svee) \to \cF(\scM)$ has a geometric interpretation.   First of all recall that the global sections functor
\[
\Gamma: \LCRS \to \op{\cring}
\]
from the category $\LCRS$ of local $\cin$-ringed spaces to the category  (opposite to)  $\cin$-rings has a left adjoint $\Spec$ (see \cite{Dubuc, Joy}), Dubuc's spectrum functor. By definition $\Spec(\scA)$ is an affine scheme for any $\cin$-ring $\scA$.   Note that if $M$ is a manifold then $\Spec(\cin(M)) = (M,\cin_M)$, where $\cin_M$ is the sheaf of smooth functions on the manifold $M$.  Since $\Spec$ is a functor, for any map $\varphi: \scA\to \scB$ of $\cin$-rings we have a map $\Spec(\varphi): \Spec(\scB) \to \Spec(\scA)$ of affine $\cin$-schemes.  Furthermore, if the map $\varphi$ is surjective then by \cite[Lemma~5.5]{L-poisson} the map $\Spec(\varphi):  \Spec(\scB) \to \Spec(\scA)$ is a closed embedding.  Therefore,
\[
\Spec(\phi): \Spec( \cF(\scM)) \to \Spec(C^\infty(E^\svee) ) = (E^\svee, \cin_{E^\svee})
\]
is a closed embedding.
\end{example}
\begin{example}
Let $\scA = C^{\infty}(\mathbb{R})$ and $\scM = \mathrm{Span}_{C^{\infty}(\mathbb{R})}(x\partial_x)\hookrightarrow \CDer(\cin(\R))$ be the Lie-Rinehart algebra generated by the vector field $x\partial_x$. Since $\scM$ is a free $\cin(\R)$-module of rank 1, the corresponding free $\cin(\R)$ algebra is $\cF(\scM) = \cin(\R) \otimes _\infty \cin(\R)\cong C^{\infty}(\mathbb{R}^2)$. The Poisson bracket is determined by $\{x, y\}=x$, where $x$ and $y$ are the standard coordinates on $\R^2$. 

{Note that for $\scM=\operatorname{Span}_{C^{\infty}(\mathbb{R})}(x^2\partial_x)\hookrightarrow \CDer(\cin(\R))$, the associated $C^\infty$-ring $\cF(\scM)$ is again $C^\infty(\mathbb{R}^2)$. However, the corresponding Poisson structure is different and is determined by
\(
\{x,y\}=x^2.
\)}
\end{example}
\begin{example}We return to Example~\ref{ex:dual_numbers}. Consider the $\cin$-ring $\scA := \cin(\R)/\langle x^2 \rangle $.  We compute 
the Poisson ring $\cF(\CDer(\scA))$ corresponding to the trivial/tautologial Lie-Rinehart algebras 
$\id:\CDer(\scA)\to \CDer(\scA)$. We have
\[
\cF (\CDer(\scA) )\simeq \cin(\R^2))/\langle x^2, xy \rangle 
.\]
The Poisson bracket is given on representatives $F, G\in C^\infty(\mathbb R^2)$ by
\[
\{F,G\}
=
x\left(
\frac{\partial F}{\partial y}\frac{\partial G}{\partial x}
-
\frac{\partial F}{\partial x}\frac{\partial G}{\partial y}
\right),
\]
and this descends to the quotient \( C^\infty(\mathbb{R}^2)/\langle x^2, xy \rangle  \), since the ideal \(\langle x^2,xy \rangle \) is stable under the bracket.
 
\end{example}
\begin{example}We return to Example \ref{ex:the_cross}.
    Consider the ideal $\langle xy \rangle $ in the $\cin$-ring $\cin(\R^2)$ generated by the product of the two standard coordinate functions $x$ and $y$.   Let $\scA = \cin(\R^2)/\langle xy\rangle$ be the quotient $\cin$-ring.  Then the identity map $\id: \CDer(\scA)\to \CDer(\scA)$ makes $((\CDer(\scA), [\cdot, \cdot ]), \id:\CDer(\scA)\to \CDer(\scA))$ into a Lie-Rinehart algebra over the $\cin$-ring $\scA$.  The corresponding Poisson $\cin$-ring is \[\cF(\CDer(\scA))=C^{\infty}(\mathbb R^4)/\langle xy,  y\xi, x\nu \rangle.\]The Poisson bracket of $\cF(\CDer(\scA))$ is determined on the ambient algebra \( C^\infty(\mathbb{R}^4) \) by
\[
\{\xi,x\} = x, \qquad
\{\nu,y\} = y, \qquad
\{\xi,y\} = 0, \qquad\{\nu,x\} = \{\xi,\nu\} = \{x, y\}=0.
\]
 This bracket is Poisson, and it is given by the bivector field
\[
\pi = x\,\partial_\xi \wedge \partial_x \;+\; y\,\partial_\nu \wedge \partial_y.
\]
That is, for \(F,G \in C^\infty(\mathbb{R}^4)\),
\[
\{F,G\}
=
x\left(
\frac{\partial F}{\partial \xi}\frac{\partial G}{\partial x}
-
\frac{\partial F}{\partial x}\frac{\partial G}{\partial \xi}
\right)
+
y\left(
\frac{\partial F}{\partial \nu}\frac{\partial G}{\partial y}
-
\frac{\partial F}{\partial y}\frac{\partial G}{\partial \nu}
\right).
\]
The ideal
\[
\langle xy,\; y\xi,\; x\nu \rangle
\]
is stable under this bracket, and therefore the Poisson structure descends to the quotient.
\end{example}
\begin{example}\label{ex:single_vf2}
    We now reconsider Example~\ref{ex:single_vf}, which describes the  (non-projective) Lie-Rinehart algebra generated by the vector field \( X = f\,\partial_x \), where \( f \in C^\infty(\mathbb{R}) \) vanishes exactly on $(-\infty,0]$. We have \[\cF(\scM)\simeq C^\infty(\mathbb R^2)/\langle a(x)\xi, a|_{(0, \infty)}\equiv 0 \rangle.\] We now describe the associated Poisson structure. The Poisson bracket is determined by
\[
\{\xi, x\} = f(x).
\]
More generally, for smooth functions \(F(x,\xi), G(x,\xi) \in C^\infty(\mathbb{R}^2)\), one has
\[
\{F,G\}
=
f(x)\left(
\frac{\partial F}{\partial \xi}\frac{\partial G}{\partial x}
-
\frac{\partial F}{\partial x}\frac{\partial G}{\partial \xi}
\right).
\]
This bracket descends to the quotient \( \cF(\scM) \), since the ideal generated by elements of the form \(a(x)\xi\), with \( a|_{(0, \infty)}\equiv 0 \), is stable under $\Pb$.

\end{example}

\begin{example}\label{ex:order2_vanishing2} We return to Example~\ref{ex:order2_vanishing}, in which $\scM\subset \mathfrak X(\mathbb R^2)$ is the Lie-Rinehart algebra over $\scA=C^\infty(\mathbb R^2)$ generated by vector fields  vanishing to order \(2\). In that case,  \[\cF(\scM)\simeq C^\infty(\mathbb R^2\times \mathbb R^6)/\langle\, x \xi_2-y \xi_1,\; x \xi_3-y \xi_2,\; x \xi_5-y \xi_4 ,\; x \xi_6-y \xi_5 \,\rangle,\]
the variables \(x, y, \xi_1, \ldots, \xi_6\) being the standard coordinates on \(\mathbb{R}^2 \times \mathbb{R}^6\). We now describe the corresponding Poisson structure.

The Poisson bracket is defined as follows. We continue to denote by $x, y, \xi_1, \ldots, \xi_6$ the generators of $\cF(\scM)$.\\

\noindent
\textbf{Brackets with base coordinates}:

\[
\{\xi_1,x\}=\{\xi_4,y\}=x^2,\quad \{\xi_4,x\}=\{\xi_1,y\}=0,
\]
\[
\{\xi_2,x\}=\{\xi_5,y\}=xy,\quad \{\xi_5,x\}=\{\xi_2,y\}=0,
\]
\[
\{\xi_3,x\}=\{\xi_6,y\}=y^2,\quad \{\xi_6,x\}=\{\xi_3,y\}=0.\]

\vspace{0.5cm}
\noindent
\textbf{Brackets among the  generators $(\xi_i)_{i=1}^6$}:

\[
\{\xi_1,\xi_2\}=-x\,\xi_2,\quad
\{\xi_1,\xi_3\}=-2x\,\xi_3,\quad
\{\xi_2,\xi_3\}=-y\,\xi_3,
\]

\[
\{\xi_4,\xi_5\}=x\,\xi_4,\quad
\{\xi_4,\xi_6\}=2y\,\xi_4,\quad
\{\xi_5,\xi_6\}=x\,\xi_6,
\]

\[
\{\xi_1,\xi_4\}=2x\,\xi_4,\quad
\{\xi_1,\xi_5\}=y\,\xi_4=x\,\xi_5,\quad
\{\xi_1,\xi_6\}=0,
\]

\[
\{\xi_2,\xi_4\}=2x\,\xi_5-y\,\xi_1,\quad
\{\xi_2,\xi_5\}=x\,\xi_6-y\,\xi_1,\quad
\{\xi_2,\xi_6\}=-x\,\xi_3,
\]

\[
\{\xi_3,\xi_4\}=2x\,\xi_6-2y\,\xi_1,\quad
\{\xi_3,\xi_5\}=y\,\xi_6-2x\,\xi_3,\quad
\{\xi_3,\xi_6\}=-2y\,\xi_3.
\]
\end{example}
\begin{remark} In Example \ref{ex:single_vf2}, the Poisson structure on $\cF(\scM)$ admits a linear Poisson extension to the ambient algebra \( C^\infty(\mathbb{R}^2) \), since every bivector field on a $2$-dimensional manifold automatically defines a Poisson structure for dimensional reasons. However, in Example \ref{ex:order2_vanishing2}, it is an open question  whether the linear Poisson structure on
\[
\cF(\scM)\simeq C^\infty(\mathbb R^2\times \mathbb R^6)\big/\langle\, x \xi_2-y \xi_1,\; x \xi_3-y \xi_2,\; x \xi_5-y \xi_4 ,\; x \xi_6-y \xi_5 \,\rangle
\]
extends to a linear Poisson structure on \(C^\infty(\mathbb R^2\times \mathbb R^6)\). Equivalently, this amounts to asking whether there exists a Lie algebroid structure on a vector bundle \(E \to \mathbb R^2\), with \(\operatorname{rank}(E)=6\), whose image of the anchor map is \(\mathscr{M}\) \cite{LLS}.
\end{remark}

Here is an example for which the module $\scM$ is not  finitely generated \cite{LLL2}.
\begin{example}\label{ex:infinte}
Let \(\scA=C^\infty(\mathbb{R})\), and let \(\left((\scM=\oplus_{i\in \mathbb N}\scA, \lb_\scM), \rho:\scM \to \CDer(\scA)\right)\) be the Lie-Rinehart algebra over $\scA$ so that
\begin{itemize}
        \item anchor map is \(\rho(e_i)=f_i\frac{\partial}{\partial x}\), for $i\in \mathbb N$;
        \item Lie bracket is \[
[e_i,e_j]_\scM = (i-j)f_0\,e_{i+j+1}.
\]
\end{itemize} 
where for all $i\in \mathbb N$ 
\[
f_i(x)=\begin{cases}
    \frac{1}{x^{i}} e^{-\frac{1}{x^{2}}}, &\; x\neq 0,\\\\0,&\; x=0
\end{cases}
\]
and $(e_i)_{i\in \mathbb N}$ is the standard basis for $\scM$. Let $(\xi_i)_{i\in \mathbb N}$ be the standard coordinates on $\mathbb R^\mathbb N$. We have   
\[
\cF(\scM)\simeq \scA \otimes_\infty C^\infty(\mathbb{R}^{\mathbb{N}})\simeq C^\infty(\mathbb R\times \mathbb R^\mathbb N),
\]
with  
\[
\eta_\scM\left(\sum_{i=1}^n a_i(x)e_i\right)=\sum_{i=1}^n a_i(x)\xi_i\qquad\text{and}\qquad \varphi_\scM(a)= \pr_1^*a 
\quad \text{for}\; i\in\mathbb N,\; a\in \scA =\cin(\R),
\]
where $\pr_1:\R\times \R^{\mathbb N} \to \R$ is the projection. The Poisson bracket is defined on generators by 
\[
\{x,x\}=0,\qquad
\{\xi_i,x\}=f_i(x),\qquad
\{\xi_i,\xi_j\}=(i-j)f_0(x)\,\xi_{i+j+1},
\]
and then extended to the whole $\cin$-ring $C^\infty(\mathbb R\times \mathbb R^\mathbb N)$ as a \(C^\infty\)-biderivation.  Here $\cin(\R^\N)$ is the free $\cin$-ring generated by the set $\N$, see Appendix~\ref{app:free}.
\end{example}

\subsection{Proof of Theorem~\ref{thm:main}} 
\mbox{}

\begin{proof}[Proof of Theorem~\ref{thm:main}] 
We first prove that if $((\scM, [\cdot\,,\cdot]_{\scM}), \scM\xrightarrow{\rho} \CDer(\scA))$ is  a  Lie-Rinehart algebra then the $\scA$-algebra $\cF(\scM)$ freely generated by the module $\scM$ carries a linear Poisson structure in the sense of Definition~\ref{def:lin_Poisson}. 
Choose a set of generators $\{v_i\}_{i\in I}$ of the module $\scM$.  Consider the free $\scA$-module $\scA^{\oplus I}$ generated by the set $I$.   Let $\{e_i\}_{i\in I}$ denote its standard basis.  There is a surjective map of modules
\[
\tau: \scA^{\oplus I} \to \scM
\]
uniquely determined by the condition that 
\[
\tau(e_i) = v_i
\]
for all $i\in I$.  By Lemma~\ref{lem:3.08}
\[
\cF(\scA^{\oplus I}) = \scA \otimes_\infty \bbF(I)
\]
with the $\scA$-algebra structure given by 
\[
\imath_\scA :\scA \hookrightarrow \scA \otimes_\infty \bbF(I).
\]
The structure map $\eta_{\scA^{\oplus I}}: \scA^{\oplus I} \to \cF(\scA^{\oplus I})$ is given by \eqref{eq:eta_free}:
\[
\eta_{\scA^{\oplus I}} (\sum_{i\in I} a_i e_i) = \sum a_i \cdot i. 
\]
By Lemma~\ref{lem:3.09}
\[
\cF(\scM) = \scA \otimes_\infty \bbF(I)/ \langle \eta_{\scA^{\oplus I}} (\ker \tau )\rangle
\]
and the structure map $\eta_\scM: \scM \to \scA \otimes_\infty \bbF(I)/ \langle \eta_{\scA^{\oplus I}} (\ker \tau )\rangle$ is given by
\[
\eta_\scM (\sum a_i v_i ) = \sum a_i \cdot  i + \langle \eta_{\scA^{\oplus I}} (\ker \tau )\rangle.
\]
We now proceed to construct a linear Poisson bracket on the $\cin$-ring $\cF(\scM)$.   Since neither $\cF(\scM)$ nor $\cF(\scA^{\oplus I}) = \scA \otimes_\infty \bbF(I)$ are free as $\cin$-rings, we cannot simply define the Poisson bracket on generators.  For this reason  we will first construct a bracket (in the sense of Definition~\ref{def:1}, i.e., a skew-symmetric biderivation) on an appropriately chosen free $\cin$-ring and then show that it descends to a well-defined Poisson bracket on \(\cF(\scM)\). The construction proceeds in several steps.

\vspace{0.2cm}

\noindent
\textbf{Step 1. A bracket $\llb\cdot\,, \cdot \rrb$ on $\scA^{\oplus I}$.}\quad 
Choose a linear order $<$ on the indexing set $I$. For  each pair of indices $i< j$ in $I$ there exist   elements $\{c^\ell_{ij}\}_{\ell\in I} \subset \scA$ with $c^\ell_{ij}= 0$ for all but finitely many $\ell$ so that 
\[
[v_i, v_j]_\scM = \sum _{\ell\in I} c^\ell_{ij} v_\ell.
\]
Since the set $\{v_i\}_{i\in I}$ of generators need not be a basis of the module $\scM$, the elements $c^\ell_{ij}$ are not necessarily unique.  We fix our choice of $c^\ell_{ij}$'s and then set 
\[
c^\ell_{ii} = 0 \qquad \textrm{ for all } i,\ell \in I
\]
and 
\[
c^\ell_{ji} := -c^\ell_{ij} \qquad \textrm{ for all } i<j \textrm{ and all }\ell.
\]
We also have a map
\[
\tilde{\rho}:= \rho\circ \tau: \scA^{\oplus I} \to \CDer(\scA)
\]
of $\scA$-modules
with
\[
\tilde{\rho} (e_i) = \rho(v_i)\qquad \textrm{ for all } i\in I.
\]
Since $\scA^{\oplus I}$ is a free $\scA$-module, there is a unique skew-symmetric $\R$-bilinear  $\scA$-biderivation (a bracket)
\[
\llb\cdot\,,\cdot\rrb\colon \scA^{\oplus I}\times \scA^{\oplus I}\to \scA^{\oplus I}
\]
so that 
\[
\llb e_i,e_j\rrb:= \sum_{\ell\in I} c_{ij}^\ell\, e_\ell,
\quad \text{for all}\; i,j\in I,
\]
and 
\[
\llb e_i, ae_j\rrb =(\widetilde{\rho}(e_i)a)e_j+a\llb e_i, e_j\rrb.
\]
for all $i, j\in I$ and $a\in \scA$.   By construction
\[
\tau (\llb e_i,e_j\rrb ) = [v_i, v_j]_\scM
\]
for all $i, j \in I$. It follows that 
\[
\tau(\llb\alpha, \beta\rrb) =[\tau(\alpha), \tau(\beta)]_\scM
\]
for all $\alpha, \beta\in \scA^{\oplus I}$.
In general there is no reason for the bracket  \(\llb\cdot\,,\cdot\rrb \) to satisfy the Jacobi identity.    So $( (\scA^{\oplus I}, \llb \cdot, \cdot \rrb), \scA^{\oplus I} \xrightarrow{\tilde{\rho}} \CDer(\scA))$ is an almost Lie-Rinehart algebra (Definition~\ref{def:cin-LR}).

However, since $\tau$ preserves brackets and since $[\cdot\,, \cdot]_\scM$ does satisfy the Jacobi identity, $\tau$ sends the Jacobiator $\operatorname{Jac}_{\llb\cdot\,,\cdot\rrb }(\alpha,\beta,\gamma)$ to zero for all $\alpha, \beta, \gamma \in \scA^{\oplus I}$.
It follows that $\operatorname{Jac}_{\llb\cdot\,,\cdot\rrb }$ is a $\ker\tau$-valued $\scA$-trilinear map on $\scA^{\oplus I}$ and  \(K:=\ker \tau\) is an almost Lie-Rinehart algebra ideal in a natural sense: for any $k\in K$ and for any $\alpha \in \scA^{\oplus I}$
\[
\llb \alpha, k\rrb \in K.
\]
\mbox{}\\[0pt]
\noindent
\textbf{Step 2. A well-defined bracket   $\lbr \cdot\,, \cdot\rbr$ on \(\cF(\scA^{\oplus I}) = \scA\otimes_\infty \bbF(I)\)}.   Since the $\cin$-ring $\scA\otimes_\infty \bbF(I)$ is not necessarily free, we first choose a surjective map $\Pi:\bbF(Y) \twoheadrightarrow \scA$ of $\cin$-rings.  That is, we choose a set $Y$ of generators of $\scA$.   Note that, as was mentioned earlier in the paper (see \ref{item:3.13.1}), 
\[
\bbF(Y)\otimes_\infty \bbF(I)  = \bbF(Y\sqcup I).
\]  
Note also that the diagram
\[
\xymatrix{
\bbF(Y) \ar@{^{(}->}[r]^{\imath_{\bbF(Y)} } \ar@{->>}[d]_{\Pi} &
\bbF (Y\sqcup I) \ar@{->}[d]^{\Pi\otimes_\infty \mathrm{id}} \\
\scA \ar[r]_<<<<<<<{\imath_\scA} &
\scA\otimes_\infty \bbF (I)
}
\]
commutes and that the map $\Pi\otimes_\infty \mathrm{id}: \bbF(Y\sqcup I) \to \scA\otimes_\infty \bbF (I)$ is surjective with 
\[
\ker (\Pi\otimes_\infty \id) = \langle \imath_{\bbF(Y)} (\ker \Pi ) \rangle.
\]
(cf.\ \ref{item:3.013.3}). Since the map $\imath_{\bbF(Y)}: \bbF(Y) \to \bbF(Y\sqcup I)$ is injective, we may suppres it and view $\bbF(Y)$ as a $\cin$-subring of $\bbF(Y\sqcup I)$. 

We define a skew-symmetric biderivation $B$ on the free $\cin$-ring $\bbF (Y\sqcup I)$ by specifying its values on the generators, which we are free to do by Theorem~\ref{thm:bider_gen}.
Since $\Pi: \bbF(Y)\to \scA$ is surjective, for every $y\in Y$ and $i\in I$ there is $b_{iy}\in \bbF(Y)$ so that 
\[
 \Pi(b_{iy})=\widetilde{\rho}(e_i)\Pi(y).
\]
Since the map $\Pi\otimes_\infty \mathrm{id}$  is surjective, for  all $i,j\in I$ there is $b_{ij}\in \bbF(Y\sqcup I) $ so that 
\[
 (\Pi\otimes_\infty \mathrm{id})\, (b_{ij})=\eta_{\scA^{\oplus I}}(\llb e_i, e_j\rrb ).
\]
Note that $b_{ji} = - b_{ij}$ for all $i,j\in I$.
Since the map $\Pi$ is surjective, for all $y\in Y$, $i\in I$ there is $b_{iy}\in \bbF(Y) \subset \bbF(Y\sqcup I)$ so that 
\[
 \Pi(b_{iy})=\widetilde{\rho}(e_i)\Pi(y).
\]
We set 
\[
b_{yi}:= - b_{iy}.
\]
By Theorem~\ref{thm:bider_gen} there exists a unique skew-symmetric biderivation
\[
B: \bbF(Y\sqcup I) \times  \bbF(Y\sqcup I) \to  \bbF(Y\sqcup I)
\]
with 
\[
B(i,j) = b_{ij} \qquad B(i,y) = b_{iy} \qquad \textrm{ and } \qquad B(x,y) = 0
\]
for all $x,y\in Y$ and all $i,j\in I$.\\

\noindent
\textbf{Step 3.} $K:= \ker (\Pi\otimes_\infty \id) $ is a $B$-ideal in $\bbF(Y\sqcup I)$. 

We argue that 
\begin{equation}
    B(f, h) \in K
\end{equation}
for any $f\in \bbF(Y\sqcup I)$ and any $h\in K$.   
Since $B$ is a biderivation, we may assume that $f\in Y\sqcup I$.  Since $K$ is generated by 
\[
\imath_{\bbF(Y)}(\ker \Pi) = \ker \Pi
\]
we may assume $h\in \ker \Pi$.  Since $\bbF(Y)$ is freely generated, as a $\cin$-ring, by the set $Y$, Remark~\ref{rmrk:A9} tells us that there is  $d\geq 0$, $g\in \cin(\R^d)$ and $y_1,\ldots, y_d\in Y\subset \bbF(Y)$ so that 
\[
h = g_{\bbF(Y)}(y_1,\ldots, y_d).
\]
Since $B$ is a $\cin$-ring biderivation,
\begin{equation}
B(f,h) = B(f, g_{\bbF(Y)}(y_1,\ldots, y_d) = \sum_s (\partial_s g)_{\bbF} (y_1, \ldots, y_d)\cdot B(f, y_s).
\end{equation}

If $f\in Y$, then $B(f, y_s) = 0$ for all $s$. Hence $B(f,h) = 0 \in \ker \Pi$.
If $f\in I$ then 
\[
(\Pi\otimes_\infty \id) (B(f,y_s)) = \Pi(B(f, y_s)) = \widetilde{\rho}(e_f) (\Pi (y_s)).
\]
Hence 
\[
\begin{split}
(\Pi\otimes_\infty \id) (B(f, h)) &= \sum_s \Pi ((\partial_s g)_{\bbF} (y_1, \ldots, y_d))\cdot \widetilde{\rho}(e_f) (\Pi (y_s))\\ 
&= \sum_s \Pi ((\partial_s g)_{\scA} (\Pi(y_1), \ldots, \Pi(y_d)))\cdot \widetilde{\rho}(e_f) (\Pi (y_s)) \qquad \textrm{ (since $\Pi$ is a homomorphism)}.
\end{split}
\]
Since $\widetilde{\rho}(e_f)$ is a $\cin$-ring derivation
\[
\begin{split}
\sum_s \Pi ((\partial_s g)_{\scA} (\Pi(y_1), \ldots, \Pi(y_d)))&\cdot \widetilde{\rho}(e_f) (\Pi (y_s)) = \widetilde{\rho}(e_f) (g_\scA (\Pi(y_1), \ldots \Pi(y_d)) \\
&= \widetilde{\rho}(e_f)  \Pi (g_{\bbF(Y)} (y_1,\ldots, y_d)) \quad \textrm{(since $\Pi$ is a homomorphism)}\\
&= \widetilde{\rho}(e_f) (\Pi(h)) \stackrel{\, \textrm{$\Pi(h) = 0$ by assumption\,}}{=} \widetilde{\rho}(e_f) (0) =0,
\end{split}
\]
since $\Pi(h) = 0$ by assumption.  Hence, $B(f,h) \in \ker \Pi$ in this case as well.

Since $\langle \ker \Pi \rangle$ is a $B$-ideal, the biderivation $B$ descends to a well-defined biderivation $\{\cdot\,,\cdot\}$ on 
\[ 
\scA\otimes_\infty \bbF(I) = 
\bbF (Y\sqcup I)/\langle \ker \Pi \rangle 
\]
In addition, for $a,b\in \scA$ and $i\in I$, the bracket $\lbr \cdot\,, \cdot\rbr$ satisfies the following relations:
\begin{align}\label{eq1:{}}
   \lbr\eta_{\scA^{\oplus I}}(e_i), \eta_{\scA^{\oplus I}}(e_j)\rbr &=\eta_{\scA^{\oplus I}}(\llb e_i, e_j\rrb )\\
\label{eq2:{}}\lbr\eta_{\scA^{\oplus I}}(e_i), \iota_\scA(a)\rbr&=\iota_\scA(\widetilde{\rho}(e_i)a)\\\label{eq3:{}}\lbr\iota_\scA(a), \iota_\scA(b)\rbr&=0\end{align}
\\[12pt]
\noindent
\textbf{Step 4.} A Poisson bracket $\Pb_{\mathcal{F}(\scM)}$ on $\cF(\scM)$.

By Step~3 the biderivation $B$ descends to a well-defined biderivation $\{\cdot\,,\cdot\}$ on 
\[
\scA\otimes_\infty \bbF(I) = \bbF(Y\sqcup I)/\langle \ker \Pi \rangle. 
\]
We argue that  $\{\cdot\,,\cdot\}$ descends to a well-defined biderivation $\Pb_{\mathcal{F}(\scM)}$, on 
\[
\mathcal{F}(\scM)=\left(\scA\otimes_\infty \bbF(I)\right)/\langle \eta_{\scA^{\oplus I}}(K)\rangle. 
\]
To prove this, it suffices to show that $\langle \eta_{\scA^{\oplus I}}(K)\rangle$ is a $\{\cdot\, ,\cdot\}$-ideal in $\scA\otimes_\infty \bbF(I)$.  We then argue that $\Pb_{\mathcal{F}(\scM)}$ satisfies the Jacobi identity, and hence $\Pb_{\mathcal{F}(\scM)}$ is a Poisson bracket.

Using the Leibniz rule for the brackets  $\Pb$,  $\lb_\scM$, together with Equations \eqref{eq1:{}}, \eqref{eq2:{}} and \eqref{eq3:{}}, we obtain
\begin{align*}
    \{\eta_{\scA^{\oplus I}}(e_i), \eta_{\scA^{\oplus I}}(\alpha)\}&=\eta_{\scA^{\oplus I}}(\llb e_i, \alpha\rrb )\\
    \{\eta_{\scA^{\oplus I}}(\alpha), \iota_\scA(a)\}&=\widetilde{\rho}(\alpha)a
\end{align*}for all $\alpha\in \scA^{\oplus I}$ and $a\in \scA$.
Consequently, if $\alpha\in \ker \tau$ and $a\in \scA$, then 
\[
\{\eta_{\scA^{\oplus I}}(\alpha), \iota_\scA(a)\}=\widetilde{\rho}(\alpha)a={\rho}(\tau(\alpha))a=0. 
\]
Moreover, since \[\tau(\llb e_i, \alpha\rrb )=[v_i, \tau(\alpha)]_\scM=0,\] it follows that\begin{align*}
    \{\eta_{\scA^{\oplus I}}(e_i), \eta_{\scA^{\oplus I}}(\alpha)\}&=\eta_{\scA^{\oplus I}}(\llb e_i, \alpha\rrb )\in \eta_{\scA^{\oplus I}}(\ker \tau).
\end{align*}
Therefore,  $\langle \eta_{\scA^{\oplus I}}(K)\rangle\subset \scA\otimes_\infty C^\infty(\mathbb R^I)$ is a $\{\cdot\,,\cdot\}$-ideal. Hence, the bracket descends to a well-defined biderivation on $\cF(\scM)=\scA\otimes_\infty C^\infty(\mathbb R^I)/\langle \eta_{\scA^{\oplus I}}(K)\rangle$. By construction, the bracket $\Pb_\scM$ satisfies Equations \eqref{eq:Rinehart-Poisson1}, \eqref{eq:Rinehart-Poisson2} and \eqref{eq:Rinehart-Poisson3}.
We now argue that $\Pb_{\mathcal{F}(\scM)}$ satisfies the Jacobi identity, and hence $\Pb_{\mathcal{F}(\scM)}$ is a Poisson bracket. By Theorem \ref{thm:Jacobiator}, it suffices to verify the Jacobi identity on generators. 
To simplify notation, for every $i\in I$ and every $a\in \mathscr A$, set
\[
\eta_\scA^{\oplus I}(e_i)=x_i
\quad \text{and} \quad
\iota_{\mathscr A}(a)=a.
\]For $i,j,k\in I$ and $a, b, c\in \scA$, it is immediate that $\operatorname{Jac}_{\{\cdot\,,\,\cdot\}}(a, b, c)=0$. We compute (as usual $\circlearrowleft (x_i, x_j, x_k)$ stands for cyclic permutations): 
\begin{align*}
    \operatorname{Jac}_{\{\cdot\,,\,\cdot\}}(x_i, x_j, x_j)&=\{\eta_{\scA^{\oplus I}}(\llb e_i, e_j\rrb ), x_k\} +\circlearrowleft (x_i, x_j, x_k)\\
    &=\left\{\sum_{r\in I}c_{ij}^rx_r, x_k\right\} +\circlearrowleft (x_i, x_j, x_k)\\
    &= \sum_{r\in I}\left(\left\{c_{ij}^r, x_k\right\}x_r + c_{ij}^r\left\{x_r, x_k\right\} + \circlearrowleft (x_i, x_j, x_k)\right)\\
    &= \sum_{r\in I}\left(-\widetilde{\rho}(e_k)(c_{ij}^r)x_r + c_{ij}^r\llb e_r, e_k\rrb  + \circlearrowleft (x_i, x_j, x_k)\right)\\
    &=\eta_{\scA^{\oplus I}}\left(\left[\!\!\left[\sum_{r\in I}c_{ij}^re_r, e_k\right]\!\!\right] +\circlearrowleft (x_i, x_j, x_k)\right)\\&=\eta_{\scA^{\oplus I}}\left(\llb \llb e_i, e_j\rrb , e_k\rrb  +\circlearrowleft (x_i, x_j, x_k)\right)\\&=\eta_{\scA^{\oplus I}}\left(\underbrace{\operatorname{Jac}_{\llb \cdot\,,\cdot\rrb }(e_i, e_j, e_k)}_{\in \ker \tau}\right)\in \langle \eta_{\scA^{\oplus I}}(K)\rangle.
\end{align*}
Next, we have 
\begin{align*}
    \operatorname{Jac}_{\{\cdot\,,\,\cdot\}}(x_i, x_j, a)&=\{\{x_i, x_j\}, a\}+ \{\{x_i, x_j\}, a\}+\{\{x_i, x_j\}, a\}\\&=\{\eta_{\scA^{\oplus I}}(\llb e_i, e_j\rrb ), a\}-\{\widetilde{\rho}(e_i)a, x_j\}+\{\widetilde{\rho}(e_j)a,x_i\}&\\&=\sum_{r\in I}\{c_{ij}^rx_r,a\}+\widetilde{\rho}(e_j)\circ\widetilde{\rho}(e_i)a -\widetilde{\rho}(e_i)\circ\widetilde{\rho}(e_j)a\\&= \sum_{r\in I}\left(\underbrace{\{c_{ij}^r,a\}}_{0}x_r+c_{ij}^r\underbrace{\{x_r,a\}}_{\widetilde{\rho}(e_r)a}\right) -[\widetilde{\rho}(e_i),\widetilde{\rho}(e_j)]a\\&=\widetilde{\rho}(\llb e_i, e_j\rrb )a- [\widetilde{\rho}(e_i),\widetilde{\rho}(e_j)]a=0
\end{align*}
Finally, 
\begin{align*}
    \operatorname{Jac}_{\{\cdot\,,\,\cdot\}}(x_i, a, b)&=\{\{x_i, a\}, b\}+ \{\{b, x_i\}, a\} + \underbrace{\{\{a, b\}, x_i\}}_{=0}\\& = \underbrace{\{\{\widetilde{\rho}(e_i)a, b\}}_{=0}-\underbrace{\{\{\widetilde{\rho}(e_i)b, a\}}_{=0}=0.
\end{align*}
Therefore $\Pb_{\mathcal{F}(\scM)}$ is a Poisson bracket on $\cF(\scM)=\scA\otimes_\infty C^\infty(\mathbb R^I)/\langle \eta_{\scA^{\oplus I}}(K)\rangle$.\\

\noindent
We now prove the converse: given a linear Poisson  $C^\infty$-ring structure $\left(\mathcal F(\scM),  \Pb_{\mathcal F(\scM)}\right)$, we construct a Lie-Rinehart algebra $((\scM, [\cdot\,,\cdot]_{\scM}),  \rho:\scM\to \CDer(\scA))$.

Since the  Poisson  $C^\infty$-ring bracket $\Pb_{\mathcal F(\scM)}$ is linear, it follows from Definition~\ref{def:lin_Poisson} that 
\begin{align*}\left\{\eta_{\scM}(\scM), \eta_{\scM}(\scM)\right\}_{\cF(\scM)}&\subset \eta_\scM(\scM)
    \\\{\eta_\scM(\scM), \varphi_\scM(\scA)\}_{\mathcal{F}(\scM)}&\subset \varphi_\scA(\scA)\nonumber&\\\{\varphi_\scM(\scA), \varphi_\scM(\scA)\}_{\mathcal{F}(\scM)}&=0
\end{align*}

\noindent
Let $\{v_i\}_{i\in I}$ be a set of generators of $\scM$. Then, as we have seen before, there exist elements \(c_{ij}^\ell \in \scA\) satisfying the skew-symmetry condition
\(
c_{ij}^\ell = -c_{ji}^\ell
\)
for all $i,j,\ell\in I$
together with
\begin{align*}
    \{\eta_\scM(v_i),\eta_\scM(v_j)\}_{\mathcal{F}(\scM)} &= \sum_{\ell\in I} c_{ij}^\ell\, \eta_\scM(v_\ell)
\quad \textrm{ for all \, }  i,j \in I,
\end{align*}
where all but finitely many of the coefficients $c_{ij}^\ell$ vanish. Moreover, since the structure map $\varphi: \scA \to \cF(\scM)$ is injective (Lemma~\ref{lem:3.06}), for every $a\in \scA$, there exists a unique element $X_i(a)\in \scA$ such that 
\[
\{\eta_{\scM}(v_i), \varphi_\scA(a)\}_{\mathcal{F}(\scM)}=\varphi_\scA(X_i(a)).
\] 
Consider the surjective $\scA$-linear map 
\[
\tau\colon \scA^{\oplus I}
\twoheadrightarrow \scM,
\]
that takes the standard basis  $(e_i)_{i\in I}$ of the  $\scA$-module $\scA^{\oplus I}$ to the generators $\{v_i\}_{i\in I}$ of $\scM$: $\tau(e_i) = v_i$ for all $i\in I$.\\

\noindent
\textbf{Claim 1.}\\
$(i)$ For all $i\in I$, the assignment $\scA\ni a\mapsto X_i(a)\in \scA$ is a $C^\infty$-ring derivation of $\scA$;\\
$(ii)$ There is an almost Lie algebroid structure \(((\scA^{\oplus I}, \llb \cdot\,,\cdot\rrb ), \widetilde{\rho}: \scA^{\oplus I} \to \CDer(\scA))\) such that 
\begin{itemize}
\item the anchor map $\widetilde{\rho}\colon \scA^{\oplus I}\to C^\infty\mathrm{Der}(\scA)$ is given by
\[
\widetilde{\rho}(e_i):= X_i,
\qquad \text{for all } i\in I,
\]

\item the skew-symmetric bracket \(\llb \cdot\,,\cdot\rrb \colon \scA^{\oplus I}\times \scA^{\oplus I}\to \scA^{\oplus I}\) is defined by

\[
\llb e_i,e_j\rrb 
:=
\sum_{\ell\in I} c_{ij}^\ell\, e_\ell,
\quad \text{for all}\; i,j\in I.
\]
\end{itemize}
\begin{proof}[Proof of Claim 1]
$(i)$ For any $n$, any $f\in
\cin(\R^n)$ and any $a_1,\ldots, a_n\in \scA$, one has
\begin{align*}
    \varphi_\scA \Bigr(X_i\left(f_\scA(a_1,\ldots,a_n)\right)\Bigr)&= \left\{\eta_{\scM}(v_i), \varphi_\scA(f_\scA(a_1,\ldots,a_n))\right\}_{\mathcal{F}(\scM)}\\&=\sum_{j=1}^n\varphi_\scA\Bigr((\partial_j
  f)_\scA(a_1,\ldots, a_n)\Bigr)\cdot \varphi_\scA(X_i(a_j))\\&=\varphi_\scA\left(\sum_{j=1}^n(\partial_j
  f)_\scA(a_1,\ldots, a_n)\cdot X_i(a_j)\right).
\end{align*}
The conclusion in $(i)$ follows from the injectivity of $\varphi_\scA$.\\

\noindent
$(ii)$ Let $i, j\in I$, and $a\in \scA$. Application of the Jacobi identity for $\{\cdot\,,\cdot\}_{\cF(\scM)}$ to the elements $\eta_\scM(e_i),  \eta_\scM(e_j)$ and $ \varphi_\scA(a)$, yields  
\begin{align*}
    0=\operatorname{Jac}_{\{\cdot\,,\,\cdot\}_\cF(\scM)}(\eta_\scM(e_i), \eta_\scM(e_j), \varphi_\scA(a))&=\widetilde{\rho}(\llb e_i, e_j\rrb )a- [\widetilde{\rho}(e_i),\widetilde{\rho}(e_j)]a.
\end{align*}
This follows from a computation similar to that in the previous direction.
\end{proof}\mbox{}\\
\noindent
\textbf{Claim 2.} The almost Lie algebroid structure \(((\scA^{\oplus I}, \llb \cdot\,,\cdot\rrb ), \widetilde{\rho}: \scA^{\oplus I} \to \CDer(\scA))\) satisfies
\begin{align}
 \{\eta_\scM(\tau(\alpha)),\eta_\scM(\tau(\beta))\}_{\mathcal{F}(\scM)} &= \eta_\scM\left(\tau(\llb \alpha,\beta\rrb )\right),
\quad \forall \alpha,\beta \in \scA^{\oplus I}.\\\{\eta_{\scM}(\tau(\alpha)), \varphi_\scA(a)\}_{\mathcal{F}(\scM)}&=\varphi_\scA(\widetilde\rho(\alpha)a), \quad \forall \alpha \in \scA^{\oplus I}, \; a\in \scA.
\end{align}
In particular, one has $\ker\tau\subset \ker\widetilde{\rho}$, and for all \(\alpha,\beta,\gamma \in \scA^{\oplus I}\), 
\[
\operatorname{Jac}_{\llb \cdot\,,\,\cdot\rrb }(\alpha,\beta,\gamma)\subset \ker\tau.\]
\begin{proof}[Proof of Claim 2] This follows from a direct computation and from the fact that the structure maps $\varphi_{\scA}$ and $\eta_{\scM}$ are injective (Lemma~\ref{lem:3.06}).
\end{proof}

Therefore, the bracket $\llb \cdot\,,\cdot\rrb $ and the map $\widetilde{\rho}$ descend to the quotient $\scA^{\oplus I}/\ker\tau$ and induce a well-defined Lie--Rinehart algebra $\left((\scA^{\oplus I}/\ker\tau \simeq \scM, \lb_\scM), \rho:\scM\to \CDer(\scA)\right)$  satisfying the required relations.
\end{proof}

\section{Cotangent bundles of orbit spaces} \label{sec:7}

In this short section we touch upon cotangent bundles of quotients of manifolds by group actions.   The reader may wish to compare our approach with the approach in \cite{HPR}.  Given an action of a compact Lie group $G$ on a manifold $M$ we describe three, in generally different,  Poisson $\cin$-rings that may reasonably  be called $\cin(T^\svee (M/G))$, the algebra of ``functions" on the ``cotangent bundle" of the orbit space $M/G$.

If the action of the Lie group $G$ on the manifold $M$ is free, then the orbit space
\[
B = M/G
\]
is a manifold.  Moreover, the cotangent bundle $T^\svee B$ of $B$ is the symplectic quotient  at 0 of the cotangent bundle of $M$ by the Hamiltonian lifted action of $G$:
\begin{equation} \label{eq:701}
T^\svee B = T^\svee M/\!/_0 G.
\end{equation}
In general the quotient $B$ is a singular object and there are many ways to define its cotangent bundle.

One approach is to {\em define} $T^\svee B$ to be the symplectic quotient at 0, that is, to take \eqref{eq:701} as a definition of $T^\svee B$.  Then since the right-hand side of \eqref{eq:701} is a symplectic stratified space in the sense of \cite{SL} there is  $\cin$-ring $\cin (T^\svee M/\!/_0 G)$ of ``smooth" functions on the symplectic quotient $T^\svee M/\!/_0 G$, which is a Poisson $\cin$-ring (see  \cite{L-poisson}).  

Thanks to the results proved in the present paper we have two more ways to define $\cin(T^\svee (M/G))$ and give it the structure of a Poisson $\cin$-ring.  First of all recall that given a $\cin$-ring $\scA$ we have the tautological Lie-Rinehart algebra
\[
\CDer{\scA} \xrightarrow{\, \id\, } \CDer{\scA}.
\]
By Theorem~\ref{thm:main} the $\scA$-algebra $\cF(\CDer{\scA})$ is a Poisson $\cin$-ring.   Therefore, since one usually defines $\cin(M/G)$ to be the $\cin$-ring of $G$-invariant functions on the manifold $M$, we can define $\cin(T^\svee(M/G))$ to be the $\cin$-ring $\cF (\CDer (\cin(M)^G))$:
\[
\cin(T^\svee(M/G)) := \cF (\CDer (\cin(M)^G)). 
\]
We will see in Example~\ref{ex:7.1} that in general the  $\cin$-rings $\cin (T^\svee M/\!/_0 G)$ and $\cF (\CDer (\cin(M)^G))$ are not isomorphic.
Note that applying the Dubuc spectrum functor $\Spec$ to any Poisson $\cin$-ring $\scA$ produces an affine  Poisson $\cin$-scheme (see \cite{L-poisson}).   So $\Spec( \cF (\CDer (\cin(M)^G)))$ is an affine  Poisson $\cin$-scheme.

There is  another Lie-Rinehart algebra that we may consider:
\[
(\CDer (\cin(M)))^G \xrightarrow{\,\, \rho\,\, } \CDer(\cin(M)^G),
\]
where $(\CDer (\cin(M)))^G$ is the $\cin(M)^G$-module of $G$-invariant vector fields/derivations of the $\cin$-ring $\cin(M)$ of smooth functions on the manifold $M$.  The map $\rho $ is, essentially, a restriction: for any $G$-invariant vector field $v$ on $M$ and any $G$-invariant function $f$ the function $v(f)$ is again $G$-invariant.  Thus, an invariant vector field $v\in (\CDer (\cin(M)))^G$ induces a derivation
\[
\rho(v): \cin(M)^G\to \cin(M)^G,\qquad \rho(v)(f): = v(f).
\]
By Theorem~\ref{thm:main} the $\cin(M)^G$-algebra $\cF\left((\CDer (\cin(M)))^G  \right)$ is a Poisson $\cin$-ring. Again, in general, the three $\cin$-rings $\cin (T^\svee M/\!/_0 G)$ with the Poisson $\cin$-ring structure coming from  $\cF (\CDer (\cin(M)^G))$ and $\cF\left((\CDer (\cin(M)))^G  \right)$ are different from each other as one can see in Example~\ref{ex:7.1}.

\begin{example} \label{ex:7.1}
Consider the action of the group $G= \{\pm 1\}$ on the real line $\R$:
\[
(-1)\cdot x = -x
\]
for all $x\in \R$.    This action lifts to the Hamiltonian action of $G$ on the cotangent bundle $T^\svee \R$ with zero moment map.  
We now compute the three associated Poisson $\cin$-rings: $\cin(T^\svee \R/\!/_0 G)$, $\cF (\CDer (\cin(\R)^G))$ and $\cF\left((\CDer (\cin(\R)))^G  \right)$.

Since the moment map is identically zero,  the  symplectic quotient $T^\svee \R/\!/_0 G$ is the symplectic orbifold $(T^\svee \R)/G$.  Thus,
\[
\cin(T^\svee \R/\!/_0 G) = \cin(T^\svee \R)^G.
\]
By a theorem of G.\ Schwarz \cite{Sch} the $\cin$-ring $\cin(T^\svee \R)^G$ is generated by the set 
\[
\{u:=x^2, v:= xy, w:= y^2\} \subset \cin(T^\svee \R)^G.  
\]These generators are subject to relations
\[
uw = v^2 \quad\textrm{and}\quad  u,w\geq 0.
\]
It follows that the $\cin$-ring $\cin(T^\svee \R)^G$ is isomorphic to the ring of Whitney smooth functions on the semi-algebraic set
\begin{equation} \label{eq:7.0.1}
C= \{(u,v,w)\in \R^3 \mid uw = v^2, u,w\geq 0\}.
\end{equation}
That is,
\[
\cin(T^\svee \R/\!/_0 G) = \cin(T^\svee \R)^G \simeq \{f\in \cin(\R^3) \mid f|_C =0\}.
\]
The ring $\cin(T^\svee \R)^G$ inherits a Poisson bracket from $\cin(T^\svee \R)$. 
For example since $\{x^2, xy\} = 2x^2$, $\{u, v\} = 2u $ and so on.

By Schwarz's theorem 
\[
\cin(\R)^G = \{f\in \cin(\R) \mid f(x) = f(-x)\} = \{f \mid f(x) = g(x^2), g\in \cin([0, \infty) \} \simeq \cin([0, \infty)).
\]
The module $\CDer(\cin([0,\infty))$ is the module of vector fields on the manifold with boundary $[0,\infty)$.  This module is freely generated by the vector field $\frac{d}{dx}$, That is, the map 
\[
\cin([0,\infty) \to \CDer(\cin([0,\infty)), \qquad f\mapsto f\frac{d}{dx}
\]
is an isomorphism of $\cin([0,\infty))$-modules.  Recall from the proof of Lemma~\ref{lem:3.08} that for any $\cin$-ring $\scA$ and for any set $Y$
\[
\cF(\scA^{\oplus Y}) = \scA\otimes_\infty \bbF(Y).
\]
Hence
\[
\begin{split}
\cF (\CDer (\cin(\R)^G)) &\simeq  \cF(\CDer (\cin[0,\infty))) \simeq \cF(\cin([0,\infty)) \\&\simeq \cin([0,\infty))\otimes_\infty \bbF({1}) = \cin([0,\infty))\otimes_\infty \cin(\R).
\end{split}
\]
To compute the coproduct $\cin([0,\infty))\otimes_\infty \cin(\R)$ we first observe that 
\[
\cin([0,\infty)) = \{f\in \cin(\R)\mid f|_{[0,\infty)} = 0\} = \cin(\R)/I
\]
where $I$ is the ideal of functions in $\cin(\R)$ that vanish identically on the ray $[0,\infty)$.
Recall that if $I$ is an ideal in a $\cin$-ring $\scA$ and $\scB$ is another $\cin$-ring then
\begin{equation} \label{eq:7.1.2}
    (\scA/I)\otimes_\infty \scB = (\scA \otimes_\infty \scB)/\langle \iota_1(I)\rangle
\end{equation}
where $\iota_1:\scA\to \scA \otimes_\infty \scB$ is one of the two structure maps and $\langle \iota_1(I)\rangle\subset \scA \otimes_\infty \scB $ is the ideal generated by the image of $I$ under the map $\iota_1$.
It follows that 
\[
\begin{split}
\cin([0,\infty))\otimes_\infty \cin(\R) &= \cin(\R)/I\otimes_\infty \cin(\R) \stackrel{\eqref{eq:7.1.2}}{=} (\cin(\R)\otimes _\infty \cin(\R))/ 
\langle \iota_1 (I)\rangle\\& \simeq \cin(\R^2) / \langle (\pr_1)^* I\rangle,
\end{split}
\]
where $\pr_1 :\R^2\to \R$ is the projection on the first factor.  One can show with a bit of effort that the ideal $ \langle (\pr_1)^* I\rangle$ is the ideal of functions in $\cin(\R^2)$ that vanish on the halfspace $[0,\infty) \times \R$:
\[
\langle (\pr_1)^* I\rangle = \{f\in \cin(\R^2) \mid f|_{[0,\infty) \times \R} =0\}.
\]
We conclude that 
\[
\cF (\CDer (\cin(\R)^G))  \simeq \cin([0,\infty) \times \R).
\]
Denote the standard coordinates on $[0,\infty) \times \R\subset \R^2$ by $x$ and $y$. Tracing throught the identificaitons we see that the map 
\[
\eta: \CDer (\cin(\R)^G) \to \cF (\CDer (\cin(\R)^G))  \simeq \cin([0,\infty) \times \R)
\]
is determined by 
\[
\eta( \frac{d}{dx}) = y.
\]
Consequently the Poisson bracket on $ \cF (\CDer (\cin(\R)^G))  \simeq \cin([0,\infty) \times \R)$ is given on the generators $x,y$ by 
\[
\{y, x\} = \{ \eta(\frac{d}{dx}), x\} = \frac{d}{dx} (x) = 1.
\]

The computation of the module $(\CDer (\cin(\R))^G = \mathfrak X(\R)^G$ of invariant vector fields is a bit more involved.  Define $\tau:\R\to \R$ by
\[ 
\tau(x) = -x.
\]
Then $v\in \mathfrak X(\R)^G$  if and only if 
\[
T\tau \circ v = v\circ \tau.
\]
Thus $v = f(x)\frac{d}{dx} \in \mathfrak X(\R)^G$ if and only if
\[
(-1) \left( f(x)\frac{d}{dx}\right) = f(-x)\frac{d}{dx}.
\]
That is,
\begin{equation}\label{eq:7.2.3}
    f(-x) = - f(x).
\end{equation}
Equation \eqref{eq:7.2.3} implies that $f(0) = 0$. Hence by Hadamard's lemma, there is a  function $h\in \cin(\R)$ so that 
\[
f(x) = x h(x).
\]
Equation \eqref{eq:7.2.3} then implies that 
\[
(-x)h(-x) = - (xh(x)).
\]
Hence 
\[
h(x) = h(-x),
\]
i.e., the function $h$ is even.  It follows from  a theorem of Whitney (which is a special case of Schwarz's theorem) that there is a smooth function $g\in \cin([0,\infty))$ so that 
\[
h(x) = g(x^2).
\]
Thus 
\[
(\CDer (\cin(\R))^G = \left\{ g(x^2) \cdot x \frac{d}{dx} \,\, \bigg| \,\, g \in \cin([0, \infty))\right\}.
\]
So again we have an isomorphism of $\cin(\R)^G = \cin([0,\infty))$-modules
\[
\cin([0,\infty))\to (\CDer (\cin(\R))^G \qquad g \mapsto x g(x^2) \frac{d}{dx}.
\]
We conclude that 
\[
\cF((\CDer (\cin(\R)))^G ) \simeq \cin([0,\infty) \times \R)
\]
with $\eta(x  \frac{d}{dx}) = y$. It remains to compute the Poisson bracket on the generators $x,y \in \cin([0,\infty) \times \R)$.  By Theorem~\ref{thm:main},
\[
\{y, x\} = \{ \eta(x  \frac{d}{dx}), x\} = \rho(x\frac{d}{dx}) x,
\]
where as before 
\[
\rho : \CDer (\cin(\R)))^G \to \CDer (\cin(\R)^G) \simeq \CDer(\cin([0,\infty))
\]
is the anchor map.  Recall that the identification of $\cin([0,\infty))$ with $\cin(\R)^G$ is
\[
h \mapsto h(u^2).
\]
Since 
\[
(u\frac{d}{du}) \,h(u^2) = 2 u^2 h'(u^2),
\]
and since in particular
\[
(u\frac{d}{du}) \,(u^2) = 2 u^2,
\]
the map $\rho:\CDer (\cin(\R)))^G \to \CDer(\cin([0,\infty))$ is determined by 
\[
\rho(x\frac{d}{dx}) x = 2x.
\]
It follow that
\[
\{y, x\} = \{ \eta(x  \frac{d}{dx}), x\} = 2x.
\]
There are many reasons why the $\cin$-rings $\cin([0,\infty)\times \R)= \cF\left((\CDer (\cin(\R)))^G  \right)$ and $(\cin(T^\svee \R))^G$ could not possibly be isomorphic, and we invite the reader to pick their favorite.   For one thing the semi-algebraic $C$ given by \eqref{eq:7.0.1}  has a singularity at the origin and the ring of Whitney smooth functions detects the singularity.
The $\cin$-rings $\cF(\CDer (\cin(\R)^G))$ and $\cF\left((\CDer (\cin(\R)))^G  \right)$ are isomorphic, but their Poisson brackets are different since in one case we have $\{y,x\} = 1$ and in the other $\{y,x\} = 2x$. 

Note that the first bracket corresponds to the standard symplectic form $dy\wedge dx$ on the cotangent bunlde $T^\svee [0,\infty)$, while the second corresponds to $\frac{1}{2x}dy\wedge dx$, which is a $b$-symplectic form.
\end{example}

\mbox{}

\appendix 

\section{Free $\cin$-rings and their biderivations}  \label{app:free}

This appendix contains  a number of results regarding the properties of free $\cin$-rings, their derivation and biderivations.
\begin{itemize}
\item We  review a well-known construction of the free $\cin$-ring functor $\bbF: \Set \to \cring$ (the left adjoint to the forgetful functor $\bbU: \cring \to \Set$) and show that the universal arrow $\delta_X:X\to \bbU(\bbF(X))$ is injective for every set $X$.  Thus, we can suppress $\delta_X$ (and the forgetful functor $\bbU$) and view the set $X$ as a subset of the corresponding free $\cin$-ring $\bbF(X)$.
\item 
We  check that a free $\cin$-ring $\bbF(X)$ is indeed generated by the set $X$ in the sense that any element of $\bbF(X)$ can be obtained uniquely from the set $X$ by applying the $\cin$-ring operations. 
\item We show  that if $i:X\hookrightarrow Y$ is an injective map of sets then $\bbF(i): \bbF(X) \to \bbF(Y)$ in an injective map of $\cin$-rings.

\item We show that the $\bbF(X)$-module $\Omega^1_{\bbF(X)}$ of K\"ahler differentials of a free $\cin$-ring $\bbF(X)$ is freely generated (as a module) by the set of symbols $\{dx\}_{x\in X}$ (see Remark~\ref{rmrk:free}): 
\[
\Omega^1_{\bbF(X)} = \Free_{\bbF(X)} (\{dx\}_{x\in X}).
\]
See Lemma~\ref{lem:Kahler_free}.
\item We prove that a skew-symmetric biderivation on a  free $\cin$-ring $\bbF(X)$ (cf.\ Remark~\ref{rmrk:3.12a}) can be defined and is uniquely determined by this choice by specifying its value on generators\footnote{This assertion is not immediate, since the underlying $\mathbb{R}$-algebra of a free $C^\infty$-ring is larger than the polynomial algebra. See Remark~\ref{rmk:free-generators}.}.  See Theorem~\ref{thm:bider_gen}.
\end{itemize}

The standard construction of the $\cin$-ring $\bbF(Y)$ freely generated by a set $Y$ realizes it as the $\cin$-ring $\cin(\R^Y)$ of functions on the vector space $\R^Y$ (the vector space of all maps from the set $Y$ to the reals) that ``depend smoothly on finitely many variables" (cf.\ \cite[Example~4.31]{Joy}). Since we want to make sure that a number of things are true about the functor $\bbF$, we carry out a construction of the functor in detail.
We start by considering the case of finite sets.
\begin{lemma} \label{lem:A1}
 The $\cin$-ring $\bbF(X)$  freely generated  by a {\em finite} set $X$  is the $\cin$-ring of smooth functions $\cin(\R^X)$ on the finite dimensional real vector space $\R^X$ and the map 
 \begin{equation} \label{eq:A1}
 \delta_X: X\to \cin(\R^X), \qquad \delta_X(x) \, (p): = p(x) \quad \textrm{ for all }p\in \R^X
 \end{equation}
 is a universal arrow from the set $X$ to the forgetful functor $\bbU:\cring \to \Set$.

 Moreover, if we chose an order $\{x_1,\ldots, x_n\}$ of elements of $X$, then for any $a\in \cin(\R^X)$ there is a unique function $f\in \cin(\R^n)$ so that
 \[
 a = f_{\cin(\R^X)}(\delta_X(x_1), \ldots, \delta_X(x_n)).
 \]
\end{lemma}
\begin{proof}
Since the set $X$ if finite, $\R^X$ is a finite dimensional real vector space and therefore a $\cin$-manifold.  It follows that it makes sense to talk about the $\cin$-ring $\cin(\R^X)$ of smooth functions on $\R^X$. Recall that for manifolds the $\cin$-ring operations are given by composition: for any $k\geq 0$, any $a_1,\ldots, a_k\in \cin(\R^X) $ and any $f\in \cin(\R^k)$
\[
f_{\cin(\R^X)}(a_1,\ldots, a_k) = f\circ (a_1, \ldots, a_k).
\]

Since for any point $x\in X$ the function $
\delta_X(x):\R^X \to \R$ defined by \eqref{eq:A1} is linear, it is $\cin$.  So the image of $\delta_X$ does land in $ \cin(\R^X)$.

Given $x\in X$ there is a function $\Delta_x: X\to \R$, the Kronecker delta,  so that
\[
\Delta_x(x') = 
\begin{cases}
    1 & x' = x\\
    0 & x'\not = x 
\end{cases} \quad .
\]
The existence of such a function implies that the map $\delta_X$ is injective.  From now on we identify $x\in X$ with the smooth function $\delta_X(x) \in \cin(\R^X)$:
\[
x\equiv \delta_X(x).
\]
And then  $X\subset \cin(\R^X)$.

An ordering $\{x_1,\ldots, x_n\}$ of elements of $X$ gives rise to a linear isomorphism
\[
\tau: \R^X \to \R^n, \qquad \tau(p) = (x_1(p), \ldots, x_n(p)).
\]
The map $\tau$ is also a diffemorphism and the pullback map
\[
\tau^*: \cin(\R^n) \to \cin(\R^X), \quad \tau^*f := f\circ \tau = f\circ (x_1,\ldots, x_n)
\]
is an isomorphism of $\cin$-rings.  It follows that for any $a\in \cin(\R^X)$ there exists a unique smooth function $f\in \cin(\R^n)$ so that
\[
a = f\circ (x_1,\ldots, x_n) = f_{\cin(\R^X)} (x_1,\ldots, x_n).
\]
In particular if $a = \delta_X(x_i)\equiv x_i$ then the corresponding function $f:\R^n\to \R$ is the projection $\pr_i$ on the $i$th factor: $\pr_i (r_1,\ldots, r_n) = r_i$.  Thus,
\[
\delta_X(x_i) = (\pr_i)_{\cin(\R^X)} (x_1,\ldots, x_n).
\]
We now check that $\delta_X$ is a universal arrow.  Suppose $\scB$ is a $\cin$-ring  and 
\[
\mu: \{x_1,\ldots, x_n\} \to \scB
\]
is a map of sets. We  define a map 
\[
\widetilde{\mu} :\cin(\R^X) \to \scB
\]
by 
\begin{equation} \label{eq:A.1.2}
\widetilde{\mu} (f_{\cin(\R^X)} (x_1,\ldots, x_n) ):= f_\scB (\mu(x_1), \ldots, \mu(x_n)).
\end{equation}
It is not hard to check that $\widetilde{\mu}$ {\em is} a map of $\cin$-rings.  Note that if $f= \pr_i$ then \eqref{eq:A.1.2} says that
\[
\widetilde{\mu} (\delta_X(x_i)) = (\pr_i)_\scB (\mu(x_1), \ldots, \mu(x_n)) \stackrel{\, \eqref{eq:201}\,}{=} \mu(x_i)
\]
It follows that 
\[
\widetilde{\mu} \circ \delta_X = \mu.
\]
Finally, if $\mu': \cin(\R^X) \to \scB$ is another map of $\cin$-rings with $\mu'\circ \delta_X = \mu$ then
\[
\mu' (f_{\cin(\R^X)} (x_1,\ldots, x_n)) = f_\scB (\mu(x_1), \ldots, \mu(x_n)) {=} \widetilde{\mu} (f_{\cin(\R^X)} (x_1,\ldots, x_n)).
\]    
Thus $\widetilde{\mu}$ is a unique map of $\cin$-rings with the  property that $\widetilde{\mu} \circ \delta_X = \mu$.  We conclude that $\delta_X$ is a universal arrow.
\end{proof}
\begin{remark}
    Note that the map $\widetilde{\mu}$ constructed in the course of the proof of Lemma~\ref{lem:A1} {\em does not} depend on the choice of ordering of elements of the set $X$.
\end{remark}
\begin{corollary}  \label{cor:A}
We have a functor $\bbF:\FinSet \to \cring$, where $\FinSet$ is the category of finite sets, so that 
\begin{enumerate}
    \item $\bbF(X) = \cin(\R^X)$ for any finite set $X$ and
    \item for any map $f:X\to Z$ between finite sets the diagram
\begin{equation} \label{eq:A.3.1}
\xy
(-10,10)*+{X}="1";
(15,10)*+{\cin(\R^X)}="2";
 (-10,-8)*+{Z}="3";
(15,-8)*+{\cin(\R^Z)}="4"; 
{\ar@{->}_{f} "1";"3"};
{\ar@{->}^{\bbF(f)} "2";"4"};
{\ar@{->}^{\delta_X} "1";"2"};
{\ar@{->}_{\delta_Z} "3";"4"};
\endxy
\end{equation}   
commutes.  Here as be before $\delta_X, \delta_Z$ are the universal arrow.
\end{enumerate}
Moreover, for any $a\in \cin(\R^X)$
\[
\bbF(f)\, (a) = (f^*)^* \,a
\]
where $f^*: \R^Z\to \R^X$ is $f^*(p) = p\circ f$ and 
\[
(f^*)^* \,a = a \circ f^*.
\]
\end{corollary}
\begin{proof}
The corollary follows quickly from the fact that $\delta_X$ is a universal arrow.
\end{proof}
\begin{notation} \label{nota:A4}
Given an inclusion $X\stackrel{\imath}{\hooklongrightarrow}Z$ of sets we denote the induced map $\imath^*: \R^Z \to \R^X$ by $\pi^Z_X$:
\[
 \pi^Z_X:= \imath^*: \R^Z\to \R^X.   
\]
The map $\pi^Z_X$ is a surjective linear map.
Note that for $p\in \R^Z$, 
\[
\pi^Z_X (p) = p|_X,
\]
so we can also think of $\pi^Z_X$ as a restriction map.
If the set $Z$ is understood, we may simply write $\pi_X$ for $\pi^Z_X$.  With this notation and with $X$, $Z$ finite, 
\[
\bbF(X\stackrel{\imath}{\hooklongrightarrow}Z ) = (\pi^Z_X)^*.
\]
\end{notation}

\begin{remark} \label{rmrk:1.2}
Given an inclusion $X\subseteq Z$ of finite sets,  a $\cin$-ring $\scA$ and a function $\nu:Z\to \scA$ the universal property of $\delta_X:X\to \cin(\R^X)$ implies that the diagram
\[
 \xy
(-14,10)*+{\cin(\R^X) }="1";
(14,10)*+{\cin(\R^Z)}="2";
(0, -6)*+{\scA}="3";
{\ar@{->} ^{(\pi^Z_X)^*} "1";"2"};
{\ar@{->}_{(\nu|_X)^\sim} "1";"3"};
{\ar@{->}^{\widetilde{\nu}} "2";"3"};
\endxy
 \] 
commutes (cf.\ Notation~\ref{nota:A4}).
\end{remark}

\begin{definition} \label{def:A6}
Let $Y$ be an arbitrary set.     Define the {\sf set of smooth function $\cin(\R^Y)$} on the vector space $\R^Y$ to be 
\[
 \begin{split} 
\cin(\R^Y)&:= \\&\left\{f:\R^Y\to \R  \mid \textrm{ there exists a finite
  set }
X \subset Y \textrm{ and }\tilde{f}\in \cin(\R^X) \textrm{ so that } f
= \tilde{f}\circ \pi_X\right\},
\end{split}
\]  
where the maps $\pi_X: \R^Y\to \R^X$ are restriction/projection maps (see Notation~\ref{nota:A4}).
\end{definition}

\begin{lemma} \label{lem:A.7}
For any set $Y$ the set $\cin(\R^Y)$ (Definition~\ref{def:A6}) is a $\cin$-subring of the
$\cin$-ring $\R^{(\R^Y)}$ of {\em all } $\R$-valued functions on $\R^Y$.  (The $\cin$-ring operations on $\R^{(\R^Y)}$ are given by composition: $f_{\R^{(\R^Y)}}(a_1,\ldots, a_n) := f\circ (a_1,\ldots, a_n)$.)
\end{lemma}  

\begin{proof}
Given $n$, $f_1,\ldots, f_n\in \cin(\R^Y)$ and $h\in \cin(\R^n)$, we
need to check that
\begin{equation}\label{eq:A71}
  h\circ (f_1, \ldots, f_n) = g\circ \pi_X
\end{equation}  for some
finite set $X \subset Y$ and some $g\in \cin(\R^X)$.
By definition of $\cin(\R^Y)$ for each $f_i\in \cin(\R^Y)$  there exists a finite
set $X_i \subset Y$ and $\tilde{f_i} \in \cin(\R^{X_i})$ so that $f_i
= \tilde{f_i}\circ \pi_{X_i}$.  Now let $X = \bigcup_{i=1}^n X_i$; let
$\varpi_i: \cin(\R^X) \to \cin(\R^{X_i})$ denote the corresponding
restrictions.  Define
\[
g = h\circ (f_1\circ \varpi_1, \ldots, f_n \circ \varpi_n).
\]  
Then $f_i\circ \varpi_i\in \cin(\R^X)$ for each $i$ and consequently
$g\in \cin (\R^X)$.   Associativity of compositions now implies that
\eqref{eq:A71} holds.
\end{proof}

\begin{remark} \label{rmrk:A8}
Just as in the case of finite sets for any set $Y$ we have an injective map 
 \[
 \delta_Y: X\to \R^{(\R^Y)}, \qquad \delta_Y(y) \, (p): = p(y) \quad \textrm{ for all }p\in \R^Y, 
 \]
 and the image of $\delta_Y$ lands in $\cin(\R^Y)$.  So we have 
 \begin{equation}
   \delta_Y: X\to \cin(\R^Y), \qquad \delta_Y(y) \, (p): = p(y) \quad \textrm{ for all }p\in \R^Y.  
 \end{equation}
Since the map $\delta_Y$ is injective, we will often suppress it and view the set $Y$ as a subset of $\cin(\R^Y)$ of all smooth functions on $\R^Y$.  We think of 
\[
y\equiv \delta_Y(y): \R^Y \to \R
\]
as a ``coordinate function."
\end{remark}

\begin{remark} \label{rmrk:A9}
It follows from Definition~\ref{def:A6}, Lemmas~\ref{lem:A.7} and \ref{lem:A1} that for any set $Y$ for any element $a\in \cin(\R^Y)$ there is $n\geq 0$, $f\in \cin(\R^n)$ and $y_1,\ldots, y_n \in Y \subset \cin(\R^Y)$ so that 
\[
a = f_{\cin(\R^Y)} (y_1, \ldots, y_n).
\]    
Thus the $\cin$-ring $\cin(\R^Y)$ is generated by $Y\subset \cin(\R^Y)$ in this sense: we can obtain any element in $\cin(\R^Y)$ by applying an appropriate operation to appropriate finitely many elements of $Y$.
\end{remark}
\begin{lemma} \label{lem:A.9}
The map
\[
\delta _Y:Y\to \cin(\R^Y)
\]  
of Remark~\ref{rmrk:A8} is a universal arrow from the set $Y$ to the forgetful functor $\bbU: \cring\to \Set$: given a $\cin$-ring
$\scA$ and a map of sets $\mu:Y\to \scA$ 
there is a unique map $\tilde{\mu}:\cin(\R^Y)\to \scA$ of $\cin$-rings
with
\[
\tilde{\mu}(\delta_Y(y) )  = \mu(y)
\]
for all $y\in Y$.
\end{lemma}

\begin{proof}
Given $f\in \cin(\R^Y)$ there is a finite set
$X\subset Y$ (which is not unique) and a function $f_X \in \cin(\R^X)$
(which is unique once $X$ is chosen) with $f = f_X\circ
\pi_X$.   The universal property of $\delta_X: X\to \cin(\R^X)$ implies
that there is a unique map $\tilde{\mu}_X: \cin(\R^X)\to \scA$ of $\cin$-rings with
$(\mu_X)^\sim \circ \eta_X  = \mu|_X$.  Set
\[
\widetilde{\mu}(f):= (\mu|_X)^\sim (f_X)
\]   
We check that the map $\widetilde{\mu}$ is well-defined. Suppose
$Z\subset Y$ is another finite set and $f_Z\in \cin(\R^Z)$ is another
function with $f = f_Z\circ \pi_Z$.   It is no loss of generality to
assume that $X\subset Z$.   Then $X\subset Z \subset Y$, hence
$\pi_X:\R^Y\to \R^X$ factors through $\R^Z$:
\[
\pi_X  = \pi^Z_X \circ \pi_Z.
\]  
And then
\[
f_Z\circ \pi_Z = f = f_X \circ \pi_X = f_X \circ \pi^Z_X \circ \pi_Z.
\]
Since $\pi_Z:\R^Y\to \R^Z$ is surjective,  $f_Z = f_X \circ \pi^Z_X$.  Note that $\mu|_X = (\mu|_Z)|_X$.
Hence 
\[
(\mu|_Z)^\sim (f_Z) = (\mu|_Z)^\sim ( f_X \circ \pi^Z_X) \stackrel{\mathrm{  Remark~\ref{rmrk:1.2} }}{=}
    (\mu|_X)^\sim (f_X).
  \]
Thus  the map $\tilde{\mu}$ is well-defined.   We leave to the reader to check
that
$\widetilde{\mu}$ is a map of $\cin$-rings.

The fact that $\widetilde{\mu}$ is unique and that $\widetilde{\mu}\circ \delta_Y = \mu$ follows from the analogous statements for $\delta_X$ when the set $X$ is finite; see Lemma~\ref{lem:A1}.
\end{proof}

\begin{remark}
Lemma~\ref{lem:A.9} (together with Lemma~\ref{lem:A.7}) imply the existence of the left adjoint $\bbF: \Set\to \cring$ to the forgetful functor $\bbU: \cring\to \Set$.  

Moreover, for any map of sets $f:X\to Y$ the diagram
\begin{equation} \label{eq:A.12.10}
\xy
(-10,10)*+{X}="1";
(15,10)*+{\cin(\R^X)}="2";
 (-10,-8)*+{Y}="3";
(15,-8)*+{\cin(\R^Y)}="4"; 
{\ar@{->}_{f} "1";"3"};
{\ar@{->}^{\bbF(f)} "2";"4"};
{\ar@{->}^{\delta_X} "1";"2"};
{\ar@{->}_{\delta_Y} "3";"4"};
\endxy
\end{equation}   
commutes.  Here as be before $\delta_X, \delta_Y$ are universal arrow.  This generalizes Corollary~\ref{cor:A}: we removed the assumption that the sets in question are finite.   
\end{remark}  

\begin{lemma}
    Let $X\xrightarrow{f} Y$ be a map of sets.  Then the $\cin$-ring map  $\bbF(f): \bbF(X) = \cin(\R^X) \to \cin(\R^Y) = \bbF(Y)$ is given by 
    \[
\bbF(f)\, (a) = (f^*)^* \,a
\]
where $f^*: \R^Y\to \R^X$ is $f^*(p) = p\circ f$ and 
\[
(f^*)^* \,a = a \circ f^*.
\]
Hence if $f:X\to Y$ is injective, then $f^*:\R^Y\to \R^X$ is surjective and $\bbF(f) = (f^*)^* $ is injective.
\end{lemma}

\begin{proof}
Since $\bbF(f)$ is the unique map of $\cin$-rings making \eqref{eq:A.12.10} commute,  it is enough to check that for any $x\in X$ 
\begin{equation} \label{eq:A.12.1}
  (f^*)^* (\delta_X(x))  = \delta_Y (f(x)).  
\end{equation}
Now, for any point $\varphi\in \R^Y$,
\[
\begin{split}
 (f^*)^* (\delta_X(x)) \, (\varphi)  &= (\delta_X(x) \circ f^*) (\varphi)  = \delta_X (\varphi \circ f) = (\varphi \circ f) (x) \\
 &= \varphi(f(x)) = \delta_Y(f(x))) \, (\varphi).  
\end{split}
\]
Thus \eqref{eq:A.12.1} holds and we are done.
\end{proof}
We note a corollary to the existence of the free functor $\bbF: \Set\to \cring$.  It will not be used in this paper.
\begin{corollary} Let $Y$ be a set and let $I$ denote the category whose
  objects are { finite } subsets of $Y$ and morphisms are
  inclusions. The $\cin$-ring $\cin(\R^Y)$ is the colimit of the
  functor $\bbF|_I:I\to \cring$.  Less formally
  \[
\cin(\R^Y) = \colim _{X \subseteq Y, X\textrm{ finite }}\cin(\R^X)
   \] 
\end{corollary}  

\begin{proof}  The set $Y$ is the colimit of its finite subsets.
 Since the functor $\bbF$ is left adjoint, it preserves colimits.  Thus,
\[
  \cin(\R^Y) = \bbF (Y) = \bbF (\colim_{X\subseteq Y, X\textrm{ finite }} X)
  =    \colim _{X \subseteq Y, X\textrm{ finite }} \bbF(X)  = \colim _{X \subseteq Y, X\textrm{ finite }} \cin(\R^X).
\]
\end{proof}

\begin{remark}\label{rmk:free-generators}
    A \( C^\infty \)-ring \( \mathbb F(X) \) freely generated by a set \( X \) is, in particular, an ordinary algebra, whose multiplication is induced by the smooth map
\[
h: \mathbb{R}^2 \to \mathbb{R}, \qquad h(x,y) :=  xy .
\]
However, \( \mathbb F(X) \) is not free as an ordinary \( \mathbb{R} \)-algebra. For example, if \( X=\{1,\ldots,.n\} \), then the corresponding free $\cin$-ring $\bbF(X)$ is the $\cin$-ring $C^\infty(\mathbb{R}^n) $ of smooth functions on $\R^n$, which is not free as an \( \mathbb{R} \)-algebra.

Moreover, in the setting of \( C^\infty \)-ring biderivations, one must verify compatibility with all smooth operations, not only with the multiplication map \((x,y)\mapsto xy\).
\end{remark}
\noindent
We next study derivations of free $\cin$-rings $\cin(\R^Y)$, $Y$ a set.

\begin{lemma} \label{lem:A13}
For an element $y$ in a set $Y$ there is a $\cin$-ring derivation
  \[
    \partial_y: \cin(\R^Y) \to \cin(\R^Y)
  \]
    given by
\begin{equation} \label{eq:A2.1}
  \partial_y (f)
  \,(p) =
  \left. \frac{d}{dt}\right|_0 f(p + t\Delta_y),
\end{equation}
for all $f\in \cin(\R^Y)$ and all  $p\in \R^Y$.  Here, as before,  $\Delta_y:
Y\to \R$ is the Kronecker delta function: $\Delta_y
(y') =1$ if $y=y'$ and is 0 otherwise.
\end{lemma}

\begin{proof}
  To get off the ground we need to check that the right-hand side of
  \eqref{eq:A2.1} defines an element of $\cin(\R^Y)$.  It is clearly a
  function from $\R^Y$ to $\R$ and the issue is whether this function
  is in $\cin(\R^Y)$. Given $f\in \cin(\R^Y)$ there exists a finite
  set $X\subseteq Y$ and $f_X \in \cin(\R^X)$ with
  $f = f_X\circ \pi_X$, where as before $\pi_X: \R^Y\to \R^X$ is the
  canonical projection.  Note that $\pi_X (\Delta_y) = \Delta_y$, with
  the obvious abuse of notation and the convention that
  $\Delta_y:\R^X\to \R$ is 0 if $y\not \in X$.  Therefore
\[
 \left. \frac{d}{dt}\right|_0 f(p + t\Delta_y) =
 \left. \frac{d}{dt}\right|_0 f_X (\pi_X(p) + t\pi_X(\Delta_y)) = \left. \frac{d}{dt}\right|_0 f_X (\pi_X(p) + t\Delta_y). 
\]  
Since $X$ is a finite set, $\R^X$ is a finite dimensional vector
space.  For any smooth function $h\in \cin(\R^X)$ and any $x\in X$ the function
\[
  \partial_xh: \R^X\to \R^X , \quad \partial_xh(\psi) :=
  \left. \frac{d}{dt}\right|_0 h (\psi + t\Delta_x)
\]  
is $\cin$ --- it is the directional derivative of $h$ in the direction
of $\Delta_x\in \R^X$.  Hence
\[
\left. \frac{d}{dt}\right|_0 f(p + t\Delta_y)  = (\partial
_{\pi_X(y)} f_X \circ \pi_X)\, (p)
\]
for all $p\in \R^Y$.  It follows that $\partial_y (f)$ is an element of $\cin(\R^Y)$.

Since directional derivatives of functions in $\cin(\R^n) $ are
$\cin$-ring derivations, the maps $  \partial_y:
\cin(\R^Y) \to \cin(\R^Y)$ are $\cin$-ring derivations as well.
\end{proof}

\begin{remark} \label{rmrk:A15}
It follows from \eqref{eq:A2.1} that if the element $f= \delta_Y(y') \equiv y'$ then
\[
\partial_y f = \partial_y y' = \begin{cases}
    1 & y=y'\\
    0 & y\not = y' .
\end{cases}
\]
\end{remark}

\begin{lemma}\label{lem:Kahler_free}
For a free $\cin$-ring $\bbF(Y)$ the module of K\"ahler differentials $\Omega_{\bbF(Y)}$ is a $\bbF$-module freely generated by the set of symbols $\{dy\}_{y\in Y}$ (see Remark~\ref{rmrk:free} That is,
\[
\Omega_{\bbF(Y)} = \Free_{\bbF(Y)} (\{dy\}_{y\in y})= \bigoplus_{y\in Y} \bbF(Y)\, dy.
\]   
Moreover the universal derivation $d: \bbF(Y) \to \Omega_{\bbF(Y)}$ is given by
\[
da = \sum_{y\in Y} \partial_y a \, dy
\]
for all $a\in \bbF(Y)$.
\end{lemma}

\begin{proof}
We {\em define}  a map $d: \bbF(Y)\to \bigoplus_{y\in Y} \bbF(Y)\, dy$ by 
\[
da = \sum_{y\in Y} \partial_y a \, dy.
\]
Since each $a\in \bbF(Y)$ depends only on finitely many variables, $\partial_y a = 0$ for all but finitely many $y\in Y$.  Hence the sum on the right is actually finite. Since each $\partial_y$ is a $\cin$-ring derivation, so is the map $d$.  

It remains to check the universal properties of the map $d$ constructed above.
Namely we need to show that for any $\bbF(Y)$  module $\scM$ and for any $\cin$-ring derivation $v:\bbF(Y)\to \scM$ there is a unique map $\imath_v:\bigoplus_{y\in Y} \bbF(Y)\, dy \to \scM$ so that
\[
\imath_v (da) = v(a)
\]
for all $a\in \bbF(Y)$.  Since the $\bbF(Y)$-module $\bigoplus_{y\in Y} \bbF(Y)\, dy $ is freely generated by $\{dy\}_{y\in Y}$ there is a unique map
\[
\imath_v: \bigoplus_{y\in Y} \bbF(Y)\, dy \to \scM
\]
with 
\[
\imath_v (dy) = v(y)
\]
for all $y\in Y$.  

By Remark~\ref{rmrk:A9} for any $a\in \bbF(Y)$ there is $n\geq 0$, $f\in \cin(\R^n)$ and $y_1,\ldots, y_n \in Y \subset \bbF(Y)$ so that 
\[
a = f_{\bbF(Y)} (y_1, \ldots, y_n).
\]    
Since the maps $\partial_y: \bbF(Y)\to \bbF(Y)$ of Lemma~\ref{lem:A13} are $\cin$-ring derivations
\[
\begin{split}
\partial_y (f_{\bbF(Y)} (y_1, \ldots, y_n)) &= \sum_{i=1}^n (\partial_i f)_{\bbF(Y)}(y_1, \ldots, y_n))\, \partial_y(y_i)\\& \stackrel{\textrm{  Remark~\ref{rmrk:A15} }}{=}
\begin{cases}
0 & y\not = y_i \textrm{ for any }i\\
(\partial_i f)_{\bbF(Y)}(y_1, \ldots, y_n)) & y=y_i \textrm{ for some }i.
\end{cases}  
\end{split}\]
Therefore 
\[
da = \sum_{i=1}^n  (\partial_i f)_{\bbF(Y)}(y_1, \ldots, y_n))\, dy_i.
\]
Therefore 
\[
\imath_v da = \sum_{i=1}^n  (\partial_i f)_{\bbF(Y)}(y_1, \ldots, y_n))\, v(y_i).
\]
On the other hand, since $v$ is a $\cin$-ring derivation
\[
v(a) = v \left( (\partial_i f)_{\bbF(Y)}(y_1, \ldots, y_n)) \right) = \sum_{i=1}^n  (\partial_i f)_{\bbF(Y)}(y_1, \ldots, y_n))\, v(y_i)
\]
as well. Thus
\[
v(a) = \imath_v (da)
\]
for any $a\in \bbF(Y)$ and we are done.
\end{proof}

\begin{theorem} \label{thm:bider_gen}
Let $\bbF(X)$ be a $\cin$-ring freely generated by a set $X$, $\scM$ an $\bbF(X)$-module  and 
\[ 
    b:X\times X \to \scM, \qquad (x,y) \mapsto b_{xy}
\]       
a function with $b_{xy} = - b_{yx}$ for all $x,y\in X$.  Then there exists (a unique) skew-symmetric $\cin$-ring biderivation (cf.\ Definition~\ref{def:bracket})
\[
B: \bbF(X)\times \bbF(X) \to \scM
\]
with
\[
B(x,y) = b_{xy}
\]
for all $x,y\in X$.
\end{theorem}

\begin{proof}
The map $b:X\times X\to \scM$ corresponds to a map
\[
\bar{b}:X\to \Hom_\Set (X, \mathbf{U}(\scM)),
\] 
where  $\mathbf{U}(\scM)$ is the set underlying the module $\scM$.
Since $\Free_{\bbF(X)}$ is left-adjoint to the forgetful functor $\mathbf{U}:\bbF(X)\textrm{-}\Mod: \to \Set$, we have a bijection
\[
\Hom_\Set (X,  \mathbf{U}(\scM) )\xrightarrow{\simeq} \Hom_{ \Free_{\bbF(X)}\textrm{-}\Mod} (\Free_{\bbF(X)}(X), \scM).
\]
Since $\Free_{\bbF(X)}(X)$ is freely generated by $X$, the map  $\bar{b}$ gives rise to 
\[
\overline{B}: \Free_{\bbF(X)} \to
\Hom_{ \Free_{\bbF(X)}\textrm{-}\Mod} (\Free_{\bbF(X)}(X), \scM),
\]
which, in turn, corresponds to a $\bbF(X)$-bilinear map 
\[
\widehat{B}: \Free_{\bbF(X)}(X) \times \Free_{\bbF(X)}(X) \to \scM.
\]
Skew-symmetry of the map $b$ implies that the map $\widehat{B}$ is skew-symmetric as well.  We get a linear map
\[
\widetilde{B}:\Lambda^2 \Free_{\bbF(X)}(X)  \to \scM.
\]
By Lemma~\ref{lem:Kahler_free}
\[
\Free_{\bbF(X)}(X) = \Omega^1_{\bbF(X)}.
\]
We therefore have 
\[
d\wedge d: \bbF(X)\times \bbF(X)  \to \Lambda^2 \Free_{\bbF(X)}(X),\qquad 
(f, g) \mapsto df \wedge dg,
\]
a universal biderivation.  Following the map $d\wedge d$ by the map $\widetilde{B}$ we get
\[
B:= \widetilde{B} \circ d\wedge d:  \bbF(X)\times \bbF(X)  \to \scM,
\]
a skew-symmetric biderivation with the desired properties.
\end{proof}

\section{Biderivations and their Jacobiators} \label{app:B}
Recall that given a skew-symmetric biderivation $\lbr\cdot\,, \cdot \rbr: \scA \times \scA \to \scA$,  the corresponding {\sf Jacobiator} $J_{\lbr\cdot\,,\, \cdot \rbr}:\scA\times \scA \times \scA \to \scA$ is defined by 
\[
J_{\lbr\cdot\,,\, \cdot \rbr} (a,b, c) := \lbr a, \lbr b, c \rbr\rbr + \lbr b, \lbr c, a  \rbr\rbr
+ \lbr c, \lbr a, b \rbr\rbr\]
for all $a,b,c\in \scA$.
The goal of this appendix is to prove that if the Jacobiator vanishes on the set of generators of the $\cin$-ring, then it vanishes identically and consequently $\lbr\cdot, \cdot \rbr$ is a Poisson bracket:

\begin{theorem}\label{thm:Jacobiator}
Let $\lbr\cdot\,, \cdot \rbr: \scA \times \scA \to \scA$ be a bracket on a $\cin$-ring $\scA$.  Let $\Pi: \cin(\R^X)\to \scA$ be a surjective map of $\cin$-rings (so the set $\Pi(X)$ is a set of generators of the $\cin$-ring $\scA$).  If the Jacobiator $J_{\lbr\cdot\,,\, \cdot \rbr} $ vanishes on $\Pi(X)\times \Pi(X)\times \Pi(X) \subseteq \scA\times \scA\times \scA$ then it is identically zero.  Hence, the pair $(\scA, \lbr\cdot\,, \cdot \rbr)$ is a Poisson $\cin$-ring.
\end{theorem}

To prove the theorem we need a lemma.

\begin{lemma}\label{lem:Jacobiator}
The Jacobiator $J:= J_{\lbr\cdot\,,\, \cdot \rbr}:\scA\times \scA \times \scA \to \scA$ of a bracket 
$\lbr\cdot\,, \cdot \rbr: \scA \times \scA \to \scA$ on a $\cin$-ring $\scA$ is a $\cin$-ring triderivation.  That is, it is a $\cin$-ring derivation in each of the three slots.
\end{lemma}
\begin{proof}
Since the bracket $\lbr\cdot\,, \cdot \rbr$  is skew-symmetric, the Jacobiator $J$ is alternating.  Therefore, it is enough to show that $J$ is a derivation in the first slot: for any $n\geq 0$, $a_1,\ldots, a_n, b, c\in \scA$ and $f\in \scA$
\begin{equation} \label{eq:B.2.1}
J(f_\scA (a_1, \dots, a_n), b, c) = \sum_{i=1}^n {\partial_i f}_{\scA}(a_1,\ldots, a_n) \, J(a_i, b, c).
\end{equation}
We set 
\begin{equation}\label{eq:1-2nd-deriv}
f_i := ({\partial_i f})_{\scA}(a_1,\ldots, a_n), \qquad f_{ij}:= ({\partial_i \partial_j f})_{\scA}(a_1,\ldots, a_n).    
\end{equation}
Using the fact that  the bracket is a biderivation, we expand the three terms of $J(f_\scA (a_1, \dots, a_n), b, c)$:
 \[
    \lbr f_\scA (a_1, \dots, a_n), \lbr b, c\rbr \rbr = \sum_{i=1}^n f_i  \lbr a_i, \lbr b, c\rbr \rbr ; 
\]
\[
\begin{split} 
\lbr b, \lbr c,   f_\scA (a_1, \dots, a_n)\rbr\rbr &= \lbr b, \sum_{i=1}^n f_i \lbr c, a_i \rbr\rbr\\
&=  \sum_{i=1}^n \lbr b, f_i \rbr \lbr c, a_i \rbr +  \sum_{i=1}^n f_i \lbr b, \lbr c, a_i \rbr\rbr
;
\end{split} 
\]
and 
\[
\begin{split} 
\lbr c, \lbr  f_\scA (a_1, \dots, a_n), b\rbr\rbr &= \lbr c, \sum_{i=1}^n f_i \lbr a_i , b \rbr\rbr\\
&=  \sum_{i=1}^n \lbr c, f_i \rbr \lbr a_i, b \rbr +  \sum_{i=1}^n f_i \lbr c , \lbr a_i , b \rbr\rbr.
\end{split} 
\]
Summing the expressions above we obtain
\[
J(f_\scA(a_1,\ldots, a_n), b, c) = \sum_i f_i \underbrace{\left( \lbr a_i, \lbr b, c\rbr\rbr + \lbr b, \lbr c, a_i\rbr\rbr + \lbr c, \lbr a_i, b\rbr\rbr \right)}_{J(a_i, b, c)} + E 
\]
where 
\[
E : = \sum_i \left( \lbr b, f_i\rbr \lbr c, a_i\rbr + \lbr c, f_i\rbr \lbr a_i, b\rbr \right).
\]
We now argue that $E= 0$, which implies that   \eqref{eq:B.2.1} holds.
Since $\lbr \cdot \,, \cdot  \rbr$ is a biderivation 
\[ E=     \sum_{i,j}
f_{ij}
\Bigl(
\lbr b, a_j \rbr \lbr c,a_i \rbr
+
\lbr c, a_j \rbr \lbr a_i,b \rbr
\Bigr), 
\]
where $f_{ij}$ are given by \eqref{eq:1-2nd-deriv}.  Since the biderivation $\lbr\cdot\,, \cdot \rbr$ is skew-symmetric  
\[ 
E = \sum_{i,j} f_{ij}\Bigl(
\lbr b, a_j \rbr \lbr c,a_i \rbr - \lbr b, a_i \rbr\lbr c, a_j \rbr 
\Bigr).
\]
Since the term in parentheses is skew-symmetric in $(i,j)$ and 
$
f_{ij}
$ is symmetric,  the whole double sum vanishes. This finishes the proof.
\end{proof}
\begin{proof}[Proof of Theorem~\ref{thm:Jacobiator}]
By Lemma~\ref{lem:Jacobiator}, the Jacobiator is a triderivation. Therefore, it vanishes identically if and only if it vanishes on a set of \(C^\infty\)-ring generators of \(\scA\).
\end{proof}

\section{An example of a $\cin$-ring where not all $\R$-algebra derivations are \(C^\infty\)-derivations}\label{app:counter-ex}
In this appendix, we provide an example of a $\cin$-ring $\scA$ and an $\R$-algebra derivation $X:\scA\to \scA$  which is not a \(C^\infty\)-derivation. Thus for $\scA$
\[
\Der (\scA)\nsubseteq  \CDer(\scA).
\]
To the best of our knowledge,  this is the first explicit example of this kind, see Lemma~\ref{lem:C1} below.  
The proof of existence of $\R$-algebra derivations which are not \(C^\infty\)-derivations when the target is a module is, in effect, due to Osborn \cite{Osborn}.  We recalled Osborn's result in Remark~\ref{alg_Kahler} above. Our proof of Lemma~\ref{lem:C1} is based on this module-valued example. We obtain the present example by way of the square zero extension.  

It is a lucky coincidence that for a smooth ($\cin$) manifold $M$ any $\R$-algebra derivation $v:\cin(M)\to \cin(M)$ is automatically a $\cin$-ring derivation.  This result generalizes to point-determined $\cin$-rings \cite{KL} and, more generally, to  finitely jet determined $\cin$-rings  \cite{Yamashita}: if $\scA$ is a finitely jet determined $\cin$-ring and $v:\scA\to \scA$ is an $\R$-algebra derivation then $v$ is a $\cin$-ring derivation. Since point determined and jet determined $\cin$-rings play no role in this paper, we will not define the terms and refer the curious reader to \cite{MR} for definitions.

\begin{lemma} \label{lem:C1}
Let $\scA$ be a $\cin$-ring, $\scM$ an $\scA$-module and $v: \scA\to \scM$ an $\R$-algebra derivation that is not a $\cin$-ring derivation. Let $\scA \ltimes \scM$ be the square zero extension.  The map
\[
X: \scA \ltimes \scM \to \scA \ltimes \scM, \qquad X(a, m):= (0, v(a))
\] 
is an $\R$-algebra derivation which is not a $\cin$-ring derivation.
\end{lemma}

\begin{proof}
Recall from Remark~\ref{rmrk:m-bm-sz} that for a function
$f \in C^\infty(\mathbb{R}^n)$ and 
$(a_1,m_1),\dots,(a_n,m_n)\in \scA\ltimes \scM$ the corresponding operation on the square zero extension $\scA \ltimes \scM$ is 
\begin{equation}
f_{\scA\ltimes \scM}((a_1,m_1),\dots,(a_n,m_n)) := 
(f_\scA(a_1,\dots,a_n), 
\sum_{i=1}^n 
(\partial_i f)_{\scA}(a_1,\dots,a_n)\,\cdot\, m_i).
\end{equation}
Hence the addition (in the $\R$-algebra underlying) $\scA\ltimes \scM$ is given by
\[
\begin{split}
+: (\scA\ltimes \scM)\times (\scA\ltimes \scM) &\to \scA\ltimes \scM,\\ \quad 
((a_1, m_1),(a_2,  m_2)) &\mapsto (a_1, m_1)+ (a_2,  m_2):=  (a_1 +a_2, m_1+ m_2)
\end{split}
\]
(we take  $f(x,y) = x+y)$) and the multiplication is given by 
\[
\begin{split}
\cdot : (\scA\ltimes \scM)\times  (\scA\ltimes \scM) &\to \scA\ltimes \scM,\\ \quad   ((a_1,m_1), (a_2,m_2) )&\mapsto (a_1, m_1) \cdot (a_2, m_2) := (a_1 a_2, \, a_1m_2 + a_2m_1).
\end{split}
\]
(we take  $f(x,y) = xy$).

It is easy to see that the map $X$ is $\R$-linear.  Moreover,
\[
\begin{split}
 X((a, m) \cdot (a', m') )&= X(a_1a_2, a_1 m_2 +a_2 m_1)  = (0, v(a_1a_2))\\
 &= (a_1, m_1) (0, v(a_2)) + (a_2, m_2) (0, v(a_1) )\\
 &= (a_1, m_1) X(a_2, m_2) + (a_2, m_2) X(a_1, m_1 ).
\end{split}
\]
Thus $X$ is an $\R$-algebra derivation.

Since $v$ is not a $\cin$-ring derivation, there is $n>0$, $f\in \cin(\R^n)$ and $a_1,\ldots, a_n\in \scA$ so that
\[
v(a_1,\ldots, a_n) \not = \sum_{i=1}^n 
(\partial_i f)_{\scA}(a_1,\dots,a_n)\,\cdot\, v(a_i)
\]
And then 
\[
\begin{split}
 X(f_{\scA\ltimes \scM}((a_1,0), \ldots, (a_n, 0)))&= X ((f_\scA(a_1,\dots,a_n),\, 
\sum_{i=1}^n 
(\partial_i f)_{\scA}(a_1,\dots,a_n)\,\cdot\, m_i)\\
&= (0, v(f_\scA (a_1, \ldots, a_n))\\
&\not = (0, \sum_{i=1}^n 
(\partial_i f)_{\scA}(a_1,\dots,a_n)\,\cdot\, v(a_i))\\
&= \sum_{i=1}^n ((\partial_i f)_{\scA}(a_1,\dots,a_n), 0) \cdot (0, v(a_i))\\
&= \sum_{i=1}^n ((\partial_i f)_{\scA}(a_1,\dots,a_n), 0)\cdot X(a_i, 0).
\end{split}
\]
Hence $X$ is not a $\cin$-ring derivation.
\end{proof}
\begin{corollary}\label{lem:C2}
    There is a $\cin$-ring $\scA$ and an $\R$-algebra derivation $X:\scA \to \scA$ which is not a $\cin$-ring derivation. 
\end{corollary} 
   
\begin{proof}
We apply Lemma~\ref{lem:C1} to $\scA = \cin(\R)$, 
$\scM = \Omega^1 _{C^\infty(\R), \mathrm{alg}}$ 
the module of $\R$-algebraic Kähler differentials and 
$v= d^{\mathrm{alg}}:\cin(\R) \to \Omega^1_{C^\infty(\R), \mathrm{alg}}$ 
the universal $\R$-algebra derivation (cf.\ Remark~\ref{alg_Kahler}). Since $f(x) = e^x$ is not an algebraic function
\[
d^{\mathrm{alg}}(e^x)\neq e^x d^{\mathrm{alg}}x
\]
by a theorem of Osborn \cite{Osborn}. Hence, 
$d^{\mathrm{alg}}$ is not a $\cin$-derivation.  By  Lemma~\ref{lem:C1}
the map 
\[
X:\cin(\R)\ltimes \Omega^1 _{C^\infty(\mathbb R), \mathrm{alg}}  \to 
\cin(\R)\ltimes \Omega^1 _{C^\infty(\mathbb R), \mathrm{alg}} 
\]
is an \(\mathbb R\)-algebra derivation 
which is not a \(C^\infty\)-ring derivation.  
\end{proof}

\end{document}